\documentclass[11pt, reqno, psamsfonts]{amsart}
\pdfoutput=1

\usepackage{amssymb}
\usepackage{amsthm}
\usepackage{amsmath}
\usepackage{latexsym}
\usepackage[T1]{fontenc}
\usepackage[utf8]{inputenc}
\usepackage[russian, french, english]{babel}
\usepackage{graphicx}
\usepackage{wrapfig}
\usepackage[justification=centering, labelfont=bf]{caption}
\usepackage{subcaption}
\usepackage{mathtools}
\usepackage[hidelinks]{hyperref}
\usepackage{tikz-cd}
\usepackage{amsbsy}
\usepackage{enumitem}
\usepackage{mathrsfs}
\usepackage{multicol}
\usepackage[titletoc]{appendix}
\usepackage[shortcuts]{extdash}
\usepackage[foot]{amsaddr}
\usepackage[oldstylenums]{kpfonts}
\usepackage{framed}
\usepackage{array}
\usepackage{tikz}
\usetikzlibrary{shapes,snakes}
\usetikzlibrary{arrows.meta}
\usepackage[spacing=true,kerning=true,babel=true]{microtype}

\makeatletter
\renewcommand{\tocsection}[3]{%
	\indentlabel{\@ifnotempty{#2}{\ignorespaces#1 #2.\quad}}\bfseries#3}
\renewcommand{\tocsubsection}[3]{%
	\indentlabel{\@ifnotempty{#2}{\ignorespaces#1 #2\quad}}#3}

\newcommand\@dotsep{4.5}
\def\@tocline#1#2#3#4#5#6#7{\relax
	\ifnum #1>\c@tocdepth 
	\else
	\par \addpenalty\@secpenalty\addvspace{#2}%
	\begingroup \hyphenpenalty\@M
	\@ifempty{#4}{%
		\@tempdima\csname r@tocindent\number#1\endcsname\relax
	}{%
	\@tempdima#4\relax
}%
\parindent\z@ \leftskip#3\relax \advance\leftskip\@tempdima\relax
\rightskip\@pnumwidth plus1em \parfillskip-\@pnumwidth
#5\leavevmode\hskip-\@tempdima{#6}\nobreak
\leaders\hbox{$\m@th\mkern \@dotsep mu\hbox{.}\mkern \@dotsep mu$}\hfill
\nobreak
\hbox to\@pnumwidth{\@tocpagenum{\ifnum#1=1\fi#7}}\par
\nobreak
\endgroup
\fi}
\AtBeginDocument{%
	\expandafter\renewcommand\csname r@tocindent0\endcsname{0pt}
}
\def\l@subsection{\@tocline{2}{0pt}{2.5pc}{5pc}{}}
\makeatother

\usepackage[left=2.3cm,right=2.3cm,top=2.3cm,bottom=2.3cm,bindingoffset=0cm]{geometry}

\usepackage[backend=biber, style=alphabetic, sorting=nyt, maxnames=100]{biblatex}
\title[Measurable Lov\'{a}sz Local Lemma]{Measurable Versions of the Lov\'{a}sz Local Lemma and Measurable Graph Colorings}
\date{}
\author{Anton~Bernshteyn}
\address{Department of Mathematics, University of Illinois at Urbana--Champaign, IL, USA and Department of Mathematical Sciences, Carnegie Mellon University, Pittsburgh, PA, USA}
\email{abernsht@math.cmu.edu}
\thanks{This research is partially supported by the Illinois Distinguished Fellowship.}

\newtheorem{theo}{Theorem}[section]
\newtheorem{prop}[theo]{Proposition}
\newtheorem{lemma}[theo]{Lemma}
\newtheorem{corl}[theo]{Corollary}
\newtheorem{conj}[theo]{Conjecture}
\newtheorem{claim}{Claim}[theo]

\newcommand*{\myproofname}{Proof}
\newenvironment{claimproof}[1][\myproofname]{\begin{proof}[#1]}{\end{proof}}

\theoremstyle{definition}
\newtheorem{defn}[theo]{Definition}
\newtheorem{exmp}[theo]{Example}
\newtheorem{remk}[theo]{Remark}
\newtheorem*{remk*}{Remark}

\theoremstyle{remark}

\newcommand{\0}{\varnothing}
\newcommand{\set}[1]{\{#1\}}
\newcommand{\proj}{\mathrm{proj}}
\newcommand{\dom}{\mathrm{dom}}
\newcommand{\im}{\mathrm{im}}
\newcommand{\supp}{\mathrm{supp}}
\renewcommand{\Top}{\mathbf{Top}}
\newcommand{\HF}{\mathbf{HF}}
\newcommand{\Free}{\mathbf{Free}}
\newcommand{\acts}{\curvearrowright}
\newcommand{\N}{\mathbb{N}}
\newcommand{\Z}{\mathbb{Z}}
\newcommand{\F}{\mathbb{F}}
\newcommand{\R}{\mathbb{R}}
\newcommand{\Q}{\mathbb{Q}}
\newcommand{\LS}{\mathcal{L}}
\newcommand{\ap}{{}^{\operatorname{ap}}}
\renewcommand{\c}{^{\operatorname{c}}}

\renewcommand{\P}{\mathbf{Prob}}
\renewcommand{\U}{\mathbf{Set}}
\newcommand{\Stab}{\mathbf{Stab}}
\newcommand{\id}{\operatorname{id}}
\renewcommand{\epsilon}{\varepsilon}
\renewcommand{\phi}{\varphi}
\renewcommand{\theta}{\vartheta}
\renewcommand{\tilde}{\widetilde}
\renewcommand{\sh}{\partial}
\newcommand{\Oracle}{\mathbb{O}}
\newcommand{\powerset}[1]{\operatorname{Pow}(#1)}
\newcommand{\symdif}{\triangle}
\newcommand{\iso}{\mathcal{I}}
\newcommand{\concat}{{^\smallfrown}}
\newcommand{\defeq}{\coloneqq}
\renewcommand{\leq}{\leqslant}
\renewcommand{\geq}{\geqslant}
\newcommand{\Def}{\mathbf{Def}}
\newcommand{\B}{\mathscr{B}}
\newcommand{\fins}[1]{{[{#1}]^{<\infty}}}
\newcommand{\finf}[2]{{[{#1} \to {#2}]^{<\infty}}}

\newcommand{\Nbhd}{{\mathbf{Nbhd}}}
\renewcommand{\D}{\operatorname{d}}
\newcommand{\code}{\mathrm{code}}

\let\S\mathbbS
\let\prg\temp

\numberwithin{equation}{section}

\makeatletter
\newcommand{\neutralize}[1]{\expandafter\let\csname c@#1\endcsname\count@}
\makeatother

\newenvironment{theobis}[1]
{\renewcommand{\thetheo}{\ref{#1}$'$}
	\neutralize{theo}\phantomsection
	\begin{theo}}
	{\end{theo}}

\newenvironment{theocopy}[1]
{\renewcommand{\thetheo}{\ref{#1}}
	\neutralize{theo}\phantomsection
	\begin{theo}}
	{\end{theo}}

\newenvironment{corlcopy}[1]
{\renewcommand{\thetheo}{\ref{#1}}
	\neutralize{theo}\phantomsection
	\begin{corl}}
	{\end{corl}}

\newcommand{\bemph}[1]{{\upshape#1}} 
\newcommand{\ep}[1]{\bemph{(}#1\bemph{)}} 

\renewcommand{\mkbibnamefirst}[1]{\textsc{#1}}
\renewcommand{\mkbibnamelast}[1]{\textsc{#1}}
\renewcommand{\mkbibnameprefix}[1]{\textsc{#1}}
\renewcommand{\mkbibnameaffix}[1]{\textsc{#1}}
\renewcommand{\finentrypunct}{}

\renewbibmacro{in:}{}

\renewbibmacro*{volume+number+eid}{%
	\printfield{volume}%
	\setunit*{\addnbspace}
	\printfield{number}%
	\setunit{\addcomma\space}%
	\printfield{eid}}

\DeclareFieldFormat[article]{volume}{\textbf{#1}\space}
\DeclareFieldFormat[article]{number}{\mkbibparens{#1}}

\DeclareFieldFormat{journaltitle}{#1,}
\DeclareFieldFormat[thesis]{title}{\mkbibemph{#1}\addperiod}
\DeclareFieldFormat[article, unpublished, thesis]{title}{\mkbibemph{#1},}
\DeclareFieldFormat[book]{title}{\mkbibemph{#1}\addperiod}
\DeclareFieldFormat[unpublished]{howpublished}{#1, }

\DeclareFieldFormat{pages}{#1}

\DeclareFieldFormat[article]{series}{Ser.~#1\addcomma}

\begin{document}
		\maketitle
		
		\begin{abstract}
			In this paper we investigate the extent to which the Lov\'asz Local Lemma (an important tool in probabilistic combinatorics) can be adapted for the measurable setting. In most applications, the Lov\'asz Local Lemma is used to produce a function $f \colon X \to Y$ with certain properties, where $X$ is some underlying combinatorial structure and $Y$ is a (typically finite) set. Can this function $f$ be chosen to be Borel or $\mu$-measurable for some probability Borel measure $\mu$ on $X$ (assuming that $X$ is a standard Borel space)? In the positive direction, we prove that if the set of constraints put on $f$ is, in a certain sense, ``locally finite,'' then there is always a Borel choice for $f$ that is ``$\varepsilon$-close'' to satisfying these constraints, for any $\varepsilon > 0$. Moreover, if the combinatorial structure on~$X$ is ``induced'' by the $[0;1]$-shift action of a countable group~$\Gamma$, then, even without any local finiteness assumptions, there is a Borel choice for~$f$ which satisfies the constraints on an invariant conull set (i.e., with $\epsilon = 0$). A direct corollary of our results is an upper bound on the measurable chromatic number of the graph $G_n$ generated by the shift action of the free group $\mathbb{F}_n$ that is asymptotically tight up to a factor of at most $2$ (which answers a question of Lyons and Nazarov). On the other hand, our result for structures induced by measure-preserving group actions is, at least for amenable groups, sharp in the following sense: a probability measure-preserving action of a countably infinite amenable group satisfies the measurable version of the Lov\'asz Local Lemma if and only if it admits a factor map to the $[0;1]$-shift action. To prove this, we combine the tools of the Ornstein--Weiss theory of entropy for actions of amenable groups with concepts from computability theory, specifically, Kolmogorov complexity.
			
			\bigskip
			
		\end{abstract}
		
		\tableofcontents
		
		\newpage
		
		\section{Introduction}
		
		\subsection{Graph colorings in the Borel and measurable settings}\label{subsection:graphcol}
		
		\mbox{}
		
		\smallskip
		
		\noindent In this paper we investigate the extent to which some classical results in finite combinatorics can be transferred to the measurable setting. Our main object of study will be the so-called Lov\'asz Local Lemma, which is discussed in some detail in the next subsection. Here we give a ``preview'' of particular applications that our general techniques can provide.
		
		Let us start with some definitions.\footnote{Graph-theoretic notation used in descriptive set theory deviates somewhat from the standard in finite combinatorics. For instance, a graph $G$ is identified with its edge set; the notation $E(G)$, common in finite combinatorics, would be in conflict with $E_G$---the equivalence relation whose classes are the connected components of $G$.} A \emph{graph}~$G$ with vertex set $X$ (or a graph \emph{on} $X$) is a symmetric irreflexive binary relation on $X$. In particular, unless stated otherwise, graphs in this paper are undirected and simple. Two vertices $x$, $y \in X$ are \emph{adjacent} in $G$ if $x\,G\,y$. A subset $X' \subseteq X$ is \emph{$G$-invariant} if no vertex in~$X'$ is adjacent to a vertex in $X \setminus X'$. A \emph{connected component} of $G$ is an inclusion-minimal nonempty $G$-invariant subset of $X$. If $X' \subseteq X$, then $G \vert X' \defeq G \cap (X')^2$ denotes the \emph{subgraph} of $G$ \emph{induced} by $X'$ (or the \emph{restriction} of $G$ to $X'$). The \emph{degree} of a vertex $x \in X$ (notation:~$\deg_G(x)$ or simply $\deg(x)$) is the cardinality of the set $G_x \defeq \{y \in X\,:\, x \,G\, y\}$. The \emph{maximum degree} of $G$ (notation:~$\Delta(G)$) is the supremum of $\deg(x)$ over all $x \in X$. A graph $G$ is said to be \emph{locally countable} if $\Delta(G) \leq \aleph_0$ and \emph{locally finite} if $\deg(x) < \aleph_0$ for all $x \in X$. The \emph{girth} of $G$ (notation:~$g(G)$) is the length of the shortest cycle in~$G$ (if $G$ is acyclic, $g(G) = \infty$ by definition). A \emph{proper \ep{vertex} coloring} of $G$ is a map $f \colon X \to Y$, where $Y$ is a set of \emph{colors}, such that $f(x) \neq f(y)$ whenever $x\,G\,y$. The \emph{chromatic number} of $G$ (notation:~$\chi(G)$) is the smallest cardinality of a set $Y$ such that $G$ admits a proper coloring~$f \colon X \to Y$.
		
		We will be interested in the properties of Borel graphs; see \cite{KechrisMarks} for a comprehensive survey of the topic. A~graph~$G$ on a standard Borel space $X$ is \emph{Borel} if it is a Borel subset of $X^2$. An important source of Borel graphs are Borel group actions. Let $\Gamma$ be a countable group acting by Borel automorphisms on a standard Borel space $X$ (in this paper we only consider left group actions). Denote this action by $\alpha \colon \Gamma \acts X$. Let $S \subseteq \Gamma$ be a generating set and define the graph $G (\alpha, S)$ on $X$ via
		$$
		x \,G(\alpha,S)\, y \,\vcentcolon\Longleftrightarrow\, x \neq y \text{ and } \gamma \cdot x =y \text{ for some } \gamma \in S \cup S^{-1}.
		$$
		Then $G(\alpha,S)$ is locally countable and Borel.
		
		For a Borel graph $G$ on $X$, its \emph{Borel chromatic number} (notation:~$\chi_{\operatorname{B}}(G)$) is the smallest cardinality of a standard Borel space $Y$ such that $G$ admits a Borel proper coloring $f \colon X \to Y$. Borel chromatic numbers were first introduced and systematically studied by Kechris, Solecki, and Todorcevic~\cite{KechrisSoleckiTodorcevic}. Clearly, $\chi(G) \leq \chi_{\operatorname{B}}(G)$. One of the starting points of Borel combinatorics is the observation that this inequality can be strict. In fact, Kechris, Solecki, and Todorcevic~\cite[Example~3.1]{KechrisSoleckiTodorcevic} gave an example of an \emph{acyclic} locally countable Borel graph $G$ such that $\chi_{\operatorname{B}}(G) = 2^{\aleph_0}$ (note that if $G$ is acyclic, then $\chi(G) \leq 2$). On the other hand, they showed~\cite[Proposition~4.6]{KechrisSoleckiTodorcevic} that if $\Delta(G)$ is finite, then $\chi_{\operatorname{B}}(G) \leq \Delta(G)+1$, in analogy with the finite case.
		
		The bound $\chi(G) \leq \Delta(G)+1$ is rather weak; in fact, Brooks's theorem in finite combinatorics asserts that $\chi(G) \leq \Delta(G)$ for all $G$ apart from a few natural exceptions~\cite[Theorem~14.4]{BondyMurty}. As it turns out, there is no hope for any result along these lines in the Borel setting: Marks~\cite[Theorem~1.3]{Marks} showed that the Borel chromatic number of an acyclic Borel graph $G$ with maximum degree $d\in\N$ can attain the value $d+1$ (and, in fact, \emph{any} value between $2$ and $d+1$).
		
		Marks's results indicate that the Borelness requirement is too restrictive to allow any interesting analogs of classical coloring results. It is reasonable, therefore, to try asking for somewhat less. For instance, we can only require that ``most'' of the graph should be colored, in an appropriate sense of the word ``most.'' Natural candidates for such a notion of largeness are Baire category and measure. If $\tau$ is a Polish topology on $X$ that is compatible with the Borel structure on $X$, then the \emph{$\tau$-Baire-measurable chromatic number} of $G$ is defined as follows:
		\[
		\chi_\tau(G) \defeq \min \{\chi_{\operatorname{B}}(G \vert X')\,:\, X' \text{ is a $\tau$-comeager $G$-invariant Borel subset of }X\}.
		\]
		Similarly, if $\mu$ is a probability Borel measure on $X$, then the \emph{$\mu$-measurable chromatic number} of $G$ is defined to be
		\[
		\chi_\mu(G) \defeq \min \{\chi_{\operatorname{B}}(G \vert X')\,:\, X' \text{ is a $\mu$-conull $G$-invariant Borel subset of }X\}.
		\]
		Like $\chi_{\operatorname{B}}(G)$, both $\chi_\tau(G)$ and $\chi_\mu(G)$ can exceed $\chi(G)$, even for locally finite acyclic graphs. A simple example is the graph $G \defeq G(\alpha, \{1\})$, where $\alpha \colon \Z \acts \S^1$ is an irrational rotation action of $\Z$ on the unit circle~$\S^1$. Each component of $G$ is a bi\=/infinite path, so $G$ is acyclic; but an easy ergodicity argument reveals that $\chi_\tau(G)$, $\chi_\mu(G) > 2$, where~$\tau$ is the usual topology and $\mu$ is the Lebesgue probability measure on $\S^1$. (Since $\Delta(G) = 2$, \cite[Proposition~4.6]{KechrisSoleckiTodorcevic} yields $\chi_\tau(G) =\chi_\mu(G) = \chi_{\operatorname{B}}(G) = 3$.)
		
		Nevertheless, Conley and Miller~\cite[Theorem~B]{ConleyMiller} showed that $\chi_\tau(G)$ cannot differ from $\chi(G)$ ``too much''; namely, they proved that for a locally finite Borel graph $G$ on a standard Borel space $X$, if $\chi(G)$ is finite, then $\chi_\tau(G) \leq 2\chi(G)-1$ with respect to any compatible Polish topology $\tau$ on~$X$. In particular, if $G$ is acyclic (or, more generally, $\chi(G) \leq 2$), then $\chi_\tau(G) \leq 3$.
		
		Our main focus will be on $\mu$-measurable chromatic numbers (and $\mu$\=/measurable analogs of other combinatorial parameters). Here the situation is more intriguing than with Baire-measurable chromatic numbers. Conley, Marks, and Tucker-Drob~\cite[Theorem~1.2]{ConleyMarksTuckerDrob} recently proved a $\mu$\=/measurable analog of Brooks's theorem for graphs with maximum degree at least $3$ (the example of an irrational rotation action shows that Brooks's theorem for graphs with maximum degree $2$ does not hold in the measurable setting). In particular, $\chi_\mu(G)$ can be \emph{strictly less} than $\chi_{\operatorname{B}}(G)$.
		
		On the other hand, in contrast to Baire-measurable chromatic numbers, $\chi_\mu(G)$ cannot be bounded above by \emph{any} function of~$\chi(G)$. An important class of examples where the difference between $\chi_\mu(G)$ and $\chi(G)$ gets arbitrarily large comes from shift actions of free groups. For a countable group $\Gamma$ and a set~$A$, the \emph{shift action} 
		of $\Gamma$ on $A^\Gamma$ (or the \emph{$A$-shift action}) is defined as follows: For all $\gamma$, $\delta \in \Gamma$ and~$x \in A^\Gamma$,
		\[
		(\gamma \cdot x)(\delta) \defeq x(\delta \gamma).
		\]
		Let $S$ be a finite set and let $\F(S)$ be the free group over $S$. Let $\alpha \colon \F(S) \acts [0;1]^{\F(S)}$ be the shift action of $\F(S)$ on $[0;1]^{\F(S)}$ and set $G \defeq G(\alpha, S)$. Let $\lambda$ denote the Lebesgue measure on $[0;1]$ (we will use this notation throughout). Off of a $\lambda^{\F(S)}$-null set, the action $\sigma_{\F(S)}$ is free, so every connected component of $G$ is an infinite $2|S|$-regular tree and hence is $2$-colorable. However, as Lyons and Nazarov~\cite{LyonsNazarov} pointed out, a result of Frieze and \L uczak~\cite{FriezeLuczak} implies that $\chi_{\lambda^{\F(S)}}(G) \geq |S|/\ln(2|S|)$ for sufficiently large $|S|$ (see also~\cite[Theorem~5.44]{KechrisMarks}, where this lower bound is established for arbitrary $S$). In particular, $\chi_{\lambda^{\F(S)}}(G) \to \infty$ as $|S| \to \infty$. Note that the group $\F(S)$ for $|S| \geq 2$ is nonamenable; in fact, Conley and Kechris~\cite{ConleyKechris} mention that there are no known examples of graphs~$G$ induced by probability measure-preserving actions of amenable groups such that $\chi_\mu(G) > \chi(G)+1$ (see~\cite[Problem~5.19]{KechrisMarks}).
		
		Note that the best known upper bound on $\chi_{\lambda^{\F(S)}}(G)$ is $2|S|$ (given by the measurable Brooks's theorem of Conley--Marks--Tucker-Drob), so the orders of magnitude of the lower and upper bounds are different. Lyons and Nazarov~\cite{LyonsNazarov} asked what the correct value of $\chi_{\lambda^{\F(S)}}(G)$ should be. As an immediate corollary of one of our main results (namely Theorem~\ref{theo:Thm1}), we can show that $|S|/\ln(|S|)$ is the right order. In fact, we have the following general theorem:
		
		\begin{theo}\label{theo:col_large_g}
			Let $\Gamma$ be a countable group with a finite generating set $S \subseteq \Gamma$. Denote $d \defeq |S \cup S^{-1}|$. Let $\alpha \colon \Gamma \acts (X, \mu)$ be a measure-preserving action of $\Gamma$ on a standard probability space~$(X, \mu)$ and set $G \defeq G(\alpha,S)$. Suppose that $\alpha$ factors to the shift action $\Gamma \acts ([0;1]^\Gamma, \lambda^\Gamma)$. If $g(G) \geq 4$, then $\chi_\mu(G) = O(d/\ln d)$; furthermore, if $g(G) \geq 5$, then $\chi_\mu(G) \leq (1+o(1))d/\ln d$.
		\end{theo}
		\begin{corl}\label{corl:colfree}
			Let $S$ be a finite set of size $k$, let $\alpha \colon \F(S)\acts [0;1]^{\F(S)}$ be the $[0;1]$-shift action of the free group $\F(S)$, and let $G \defeq G(\alpha,S)$. Then
			\begin{equation}\label{eq:colfree}
			(1-o(1))\frac{k}{\ln k} \leq \chi_{\lambda^\Gamma}(G) \leq (2 + o(1))\frac{k}{\ln k}.
			\end{equation}
		\end{corl}
		
		Note that, by a result of Bowen~\cite[Theorem~1.1]{Bowen}, any two nontrivial\footnote{Here, a probability measure $\nu$ is said to be \emph{nontrivial} if it is not concentrated on a single point.} shift actions of $\F(S)$, where $|S| \geq 2$, admit factor maps to each other, so \eqref{eq:colfree} holds for any such action as well.
		
		Another extensively studied graph parameter is the so-called chromatic index of a graph. Let $G$ be a graph with vertex set $X$. An \emph{edge coloring} of $G$ is a map $f \colon G \to Y$ such that for all $(x, y) \in G$, $f(x,y) = f(y,x)$. An edge coloring $f$ is \emph{proper} if for all $x$, $y$, $z \in X$ with $x \,G\, y$, $y\,G\,z$, and $x \neq z$, $f(x,y) \neq f(y,z)$. The \emph{chromatic index} of $G$ (notation: $\chi'(G)$) is the smallest cardinality of a set $Y$ such that $G$ admits a proper edge coloring $f \colon G \to Y$. Clearly, $\chi'(G) \geq \Delta(G)$, since all the edges incident to a given vertex have to receive distinct colors. A celebrated theorem of Vizing~\cite[Theorem~17.4]{BondyMurty} asserts that this bound is almost tight; namely, for a finite graph $G$, $\chi'(G) \leq \Delta(G)+1$.
		
		Naturally, for a Borel graph $G$ on a standard Borel space $X$, its \emph{Borel chromatic index} $\chi'_{\operatorname{B}}(G)$ is the smallest cardinality of a standard Borel space $Y$ such that $G$ admits a Borel proper edge coloring $f \colon G \to Y$ (where $G$ inherits its Borel structure from~$X^2$). Clearly, $\chi'(G) \leq \chi'_{\operatorname{B}}(G)$. Marks~\cite[Theorem~1.4]{Marks} showed that the Borel chromatic index of an acyclic Borel graph $G$ with maximum degree $d \in \N$ can be as large as $2d-1$ (and this bound is tight---finding a proper edge coloring of a graph with maximum degree $d$ is equivalent to finding a proper \emph{vertex} coloring of an auxiliary graph with maximum degree $2d-2$).
		
		One can define the $\mu$-measurable chromatic index of a Borel graph $G$ by analogy with its $\mu$\=/measurable chromatic number; namely,
		\[
		\chi'_\mu(G) \defeq \min\{\chi'_{\operatorname{B}}(G \vert X') \,:\, X' \text{ is a $\mu$-conull $G$-invariant Borel subset of }X\}.
		\]
		Cs\'oka, Lippner, and~Pikhurko~\cite[Theorem~1.4]{CLP} proved that Vizing's theorem holds measurably for locally finite bipartite graphs and that $\chi'_\mu(G) \leq \Delta(G) + o(\Delta(G))$ in general, provided that the measure $\mu$ is $G$\=/invariant. Theorem~\ref{theo:Thm1} gives a different proof of the second part of this result for graphs induced by shift actions (with a slightly worse lower order term); moreover, it implies the following ``list version'':
		\begin{theo}\label{theo:FreeEdge}
			For every $d \in \N$, there exists $k = d + o(d)$ such that the following holds. Let~$\Gamma$ be a countable group with a finite generating set $S \subseteq \Gamma$ such that $|S \cup S^{-1}| = d$. For each $\gamma \in S \cup S^{-1}$, let $L(\gamma)$ be a finite set such that $L(\gamma) = L(\gamma^{-1})$ and $|L(\gamma)| \geq k$ for all $\gamma \in S \cup S^{-1}$. Let $\alpha \colon \Gamma \acts (X, \mu)$ be a measure-preserving action of $\Gamma$ on a standard probability space $(X, \mu)$ and let $G \defeq G(\alpha,S)$. Suppose that $\alpha$ factors to the shift action $\Gamma \acts ([0;1]^\Gamma, \lambda^\Gamma)$. Then there exists a $\Gamma$-invariant $\mu$-conull Borel subset $X' \subseteq X$ and a Borel proper edge coloring $f$ of $G \vert X'$ such that for all $x \in X'$, $f(x, \gamma \cdot x) \in L(\gamma)$.
		\end{theo}		
		
		One can further relax the conditions on a coloring to allow a small (but positive) margin of error. Let $G$ be a graph with vertex set $X$. For a map $f \colon X \to Y$, define the \emph{defect} set~$\Def(f) \subseteq X$ by
		\[
		x \in \Def(f) \,\vcentcolon\Longleftrightarrow\, f(x) = f(y) \text{ for some } y \in G_x.
		\]
		In other words, a vertex $x$ belongs to $\Def(f)$ if and only if it shares a color with a  neighbor. If the graph $G$ is Borel, then a Borel map $f \colon X \to Y$ is a \emph{$(\mu, \epsilon)$-approximately proper Borel coloring} of~$G$ if $\mu(\Def(f)) \leq \epsilon$.\footnote{Note that the set $\Def(f)$ is analytic (and hence universally measurable), so this definition makes sense. If~$G$ is locally countable, then $\Def(f)$ is actually Borel.} The \emph{$\mu$-approximate chromatic number} of $G$ (notation:~$\ap\chi_\mu(G)$) is the smallest cardinality of a standard Borel space $Y$ such that for every $\epsilon > 0$, there is a $(\mu, \epsilon)$-approximately proper Borel coloring $f \colon X \to Y$ of $G$. Approximate chromatic numbers were studied extensively by Conley and Kechris~\cite{ConleyKechris}. 
		In particular, they proved that if $G$ is induced by a measure-preserving action of a countable \emph{amenable} group, then its $\mu$-approximate chromatic number is essentially determined by the ordinary chromatic number; more precisely, for such $G$,
		$$
		\ap\chi_\mu(G) = \min\{\chi(G \vert X')\,:\, X' \text{ is a $\mu$-conull $G$-invariant Borel subset of }X\}.
		$$
		However, the lower bound $\ap\chi_{\lambda^{\F(S)}}(G(\alpha, S)) \geq |S|/\ln(2|S|)$, where $\alpha \colon \F(S) \acts [0;1]^{\F(S)}$ is the shift action of the free group $\F(S)$ over a finite set $S$, still holds.
		
		For an edge coloring $f \colon G \to Y$, let $\Def'(f) \subseteq X$ be given by
		\[
		x \in \Def'(f) \,\vcentcolon\Longleftrightarrow\, \begin{array}{l}
		\exists y \in G_x\, \exists z \in G_y\, (z \neq x \text{ and } f(x, y) = f(y, z)); \quad \text{ or }\\
		\exists y \in G_x \,\exists z \in G_x\, (z \neq y \text{ and } f(x, y) = f(x, z)).
		\end{array}
		\]
		In other words, $x \in \Def'(f)$ if and only if $x$ is incident to an edge that shares an endpoint with another edge of the same color. The \emph{$\mu$-approximate chromatic index} $\ap\chi'_\mu(G)$ of a Borel graph~$G$ is defined similarly to $\ap\chi_\mu(G)$. As a corollary of our other general result (namely Theorem~\ref{theo:approxLLL}), Theorems~\ref{theo:col_large_g} and \ref{theo:FreeEdge} can be generalized to arbitrary locally finite Borel graphs in the context of approximate colorings.
		
		\begin{theo}\label{theo:approxVertex}
			Let $G$ be a Borel graph on a standard Borel space $X$ and suppose that $\Delta(G) = d \in \N$. Let $\mu$ be a probability Borel measure on $X$. If $g(G) \geq 4$, then $\ap\chi_\mu(G) = O(d/\ln d)$; furthermore, if $g(G) \geq 5$, then $\ap\chi_\mu(G) \leq (1 + o(1)) d/\ln d$.
		\end{theo}
		
		\begin{theo}\label{theo:approxEdge}
			Let $G$ be a Borel graph on a standard Borel space $X$ and suppose that $\Delta(G) = d \in \N$. Let $\mu$ be a probability Borel measure on $X$. Then $\ap\chi'_\mu(G) = d + o(d)$.
		\end{theo}
		
		\subsection{The Lov\'asz Local Lemma and its applications}\label{subsection:LLL}
		
		\mbox{}
		
		\smallskip
		
		\noindent The Lov\'asz Local Lemma (the LLL for short) is a powerful probabilistic tool developed by Erd\H os and Lov\'asz~\cite{ErdosLovasz}. We refer to~\cite[Chapter~5]{AS00} for background on the Lov\'asz Local Lemma and its applications in combinatorics; several other classical applications can be found, e.g., in~\cite{MolloyReed}.
		
		Given sets $X$ and $Y$,
		\[
		\begin{array}{rll}
		\text{--} & \fins{X} & \text{denotes the set of all finite subsets of $X$;}\\
		\text{--} & \finf{X}{Y} & \text{denotes the set of all partial functions $\phi \colon X \rightharpoonup Y$ with $\dom(\phi) \in \fins{X}$.}
		\end{array}
		\]
		Let $X$ be a set and consider any $S \in \fins{X}$. Even though $X$ itself is just a set with no additional structure, $[0;1]^S$ is a standard Borel space equipped with the Lebesgue probability measure $\lambda^S$. We refer to the Borel subsets $B \subseteq [0;1]^S$ as \emph{bad events} over $X$. Every bad event is a subset of $\finf{X}{[0;1]}$. If $B \subseteq [0;1]^S$ is a nonempty bad event, then we call $S$ the \emph{domain} of $B$ and write $\dom(B) \defeq S$; since $B$ is nonempty, $S$ is determined uniquely. Set $\dom(\0) \defeq \0$. The \emph{probability} of a bad event $B$ is
		\[
			\mathbb{P}[B] \defeq \lambda^{\dom(B)}(B).
		\]
		A function $f \colon X \to [0;1]$ \emph{avoids} a bad event $B$ if there is no $w \in B$ with $w \subseteq f$. An~\emph{instance \ep{of the LLL}} over $X$ is a set $\B$ of bad events over $X$. A~\emph{solution} to an instance $\B$ is a map $f \colon X \to [0;1]$ that avoids all $B \in \B$. For an instance $\B$ and a bad event $B \in \B$, the \emph{neighborhood} of $B$ in $\B$ is
		\[
			\Nbhd_\B(B) \defeq \set{B' \in \B \setminus \set{B} \,:\, \dom(B') \cap \dom(B) \neq \0}.
		\]
		The \emph{degree} of $B$ in $\B$ is	
		\[
			\deg_\B(B) \defeq |\Nbhd_\B(B)|.
		\]
		Let
		\[
			p(\B) \defeq \sup_{B \in \B} \mathbb{P}[B] \qquad \text{and} \qquad d(\B) \defeq \sup_{B \in \B} \deg_\B(B).
		\]
		An instance $\B$ is \emph{correct for the Symmetric LLL} (the SLLL for short) if
		\[
			e\cdot p(\B) \cdot (d(\B) + 1) < 1,
		\]
		where $e = 2.71\ldots$ denotes the base of the natural logarithm.
		
		\begin{theo}[{Erd\H os--Lov\'asz~\cite{ErdosLovasz}; \textbf{Symmetric Lov\'asz Local Lemma---finite case}}]\label{theo:SLLL_fin}
			Let $\B$ be an instance of the~LLL over a finite set $X$. If $\B$ is correct for the~SLLL, then $\B$ has a solution.
		\end{theo}
		
		The Symmetric LLL was introduced by Erd\H os and Lov\'asz (with $4$ in place of $e$) in their seminal paper~\cite{ErdosLovasz}; the constant was later improved by Lov\'asz (the sharpened version first appeared in~\cite{S}). Theorem~\ref{theo:SLLL_fin} is a special case of the~SLLL in the so-called \emph{variable framework} (the~name is due to Kolipaka and Szegedy~\cite{KolipakaSzegedy}), which encompasses most typical applications (with a notable exception of the ones concerning random permutations, see, e.g., \cite{ES}). For the full statement of the~SLLL, see~\cite[Corollary~5.1.2]{AS00} (deducing Theorem~\ref{theo:SLLL_fin} from \cite[Corollary~5.1.2]{AS00} is routine; see, e.g.,~\cite[41]{MolloyReed}).
		
		Theorem~\ref{theo:SLLL_fin} can be also extended to instances $\B$ with $d(\B) = \infty$, provided that for $B \in \B$, $\mathbb{P}[B]$ decays sufficiently fast as $|\dom(B)|$ increases. An instance $\B$ is \emph{correct for the General LLL} (the~GLLL for short), or simply \emph{correct}, if the neighborhood of each $B \in \B$ is countable, and there exists a function $\omega \colon \B \to [0;1)$ such that for all $B \in \B$,
		\[
		\mathbb{P}[B] \leq \omega(B) \prod_{B' \in \Nbhd_\B(B)} (1 - \omega(B')).
		\]
		
		\begin{theo}[{\textbf{General Lov\'asz Local Lemma---finite case}; \cite[Lemma~5.1.1]{AS00}}]\label{theo:GLLL_fin}
			Let $\B$ be an instance of the~LLL over a finite set $X$. If $\B$ is correct for the~GLLL, then $\B$ has a solution.
		\end{theo}
		
		A standard calculation (see~\cite[proof of Corollary~5.1.2]{AS00}) shows that if an instance $\B$ is correct for the~SLLL, then it is also correct for the~GLLL, hence the name ``General LLL.''
		
		\begin{remk}\label{remk:empty}
			If $\B$ is a correct instance of the~LLL, then we may assume that $\dom(B) \neq \0$ for all $B \in \B$. Indeed, there are only two bad events with empty domain: $\0$ and $\set{\0}$. The event $\0$ is always avoided, so it does not matter if $\0 \in \B$ or not. On the other hand, $\set{\0}$ cannot be avoided; in particular, if $\B$ is correct, then $\set{\0} \not \in \B$.
		\end{remk}
		
		\begin{remk}\label{remk:other_spaces}
			The definition of bad events can be naturally extended to include subsets of~$\finf{X}{Y}$ for standard probability spaces $(Y, \nu)$ other than $([0;1], \lambda)$; indeed, in standard combinatorial applications, $Y$ is often a finite set. However, any standard probability space $(Y, \nu)$ can be ``simulated'' by $([0;1], \lambda)$, in the sense that there exists a Borel map $\varphi \colon [0;1] \to Y$ such that $\varphi_\ast(\lambda) = \nu$. As far as the~LLL is concerned, a set $B \subseteq \finf{X}{Y}$ can be replaced by its ``pullback'' $\varphi^\ast(B) \subseteq \finf{X}{[0;1]}$ defined via
			\[
				w \in \varphi^{\ast}(B) \,\vcentcolon\Longleftrightarrow\, \varphi \circ w \in B.
			\]
			Therefore, no generality is lost when only working with subsets of $\finf{X}{[0;1]}$.
		\end{remk}
		
		Theorems~\ref{theo:SLLL_fin} and~\ref{theo:GLLL_fin} also hold in the case when the ground set $X$ is infinite. In most applications, one may assume that each bad event $B$ is an open subset of $[0;1]^{\dom(B)}$ and obtain infinitary analogs of the~LLL through standard compactness arguments (see, e.g., \cite[Theorem~5.2.2]{AS00}). Yet, a different proof is required in general. Kun~\cite[Lemma~13]{Kun} showed that the infinite version of the~LLL can be derived using the effective approach developed by Moser and Tardos~\cite{MoserTardos}.
		
		\begin{theo}[{{Kun~\cite[Lemma~13]{Kun}}; \textbf{General Lov\'asz Local Lemma---infinite version}}]\label{theo:LLLvbls}
			Let $\B$ be an instance of the~LLL over an arbitrary set $X$. If $\B$ is correct for the~GLLL, then $\B$ has a solution.
		\end{theo}

		Since the Moser--Tardos theory will play a crucial role in our investigation, we present its main tools, including a proof of Theorem~\ref{theo:LLLvbls}, in Section~\ref{section:MoserTardos}.
		
		As a simple example, let $H$ be a $k$-uniform hypergraph with vertex set $X$, i.e., a collection of $k$\=/element subsets of $X$, called the \emph{edges} of $H$. A \emph{proper $2$-coloring} of~$H$ is a map $f \colon X \to 2$ such that every edge $S \in H$ contains vertices of both colors. For $S \in H$, let $w_{S, 0}$, $w_{S, 1} \colon S \to 2$ denote the constant $0$ and $1$ functions respectively and define $B_S \defeq \set{w_{S, 0}, w_{S, 1}}$. Set
		\[
			\B \defeq \set{B_S \,:\, S \in H}.
		\]
		As explained in Remark~\ref{remk:other_spaces}, $\B$ can be viewed as an instance over $X$. The proper $2$-colorings of $H$ are precisely the solutions to $\B$. It is straightforward to check the conditions under which $\B$ is correct for the SLLL, and, after an easy calculation, one recovers the following theorem due to Erd\H{o}s and Lov\'{a}sz, which historically was the first application of the LLL:
		
		\begin{theo}[Erd\H os--Lov\'asz~\cite{ErdosLovasz}]\label{theo:hypcol}
			Let $H$ be a $k$-uniform hypergraph and suppose that every edge of~$H$ intersects at most $d$ other edges. If $e(d+1)\leq 2^{k-1}$, then $H$ is $2$-colorable.\footnote{The best currently known bound that guarantees $2$-colorability of $H$ is $d \leq c (k/\ln k)^{1/2} 2^k$ for some positive absolute constant $c$, due to Radhakrishnan and Srinivasan~\cite[Theorem~4.2]{RS}. Their proof also relies on the LLL.}
		\end{theo}
		
		To illustrate the types of results one can obtain using the~LLL, 
		we describe a few other applications below.
		
		\subsubsection*{Kim's and Johansson's theorems}
		
		\mbox{}
		
		\smallskip
		
		\noindent Let $G$ be a ``sparse'' graph, in that it does not contain any ``short'' cycles. Can one show that $\chi(G)$ is much smaller than $\Delta(G)$, the bound given by Brooks's theorem? It is well-known that there exist $d$\=/regular graphs with arbitrarily large girth and with chromatic number at least $(1/2-o(1))d/\ln d$. After a series of partial results by a number of researchers~(see~\cite[Section~4.6]{JensenToft} for a survey), Kim~\cite{Kim} proved an upper bound that (asymptotically) exceeds the lower bound only by a factor of $2$:
		
		\begin{theo}[Kim~\cite{Kim}; see also~{\cite[Chapter~12]{MolloyReed}}]\label{theo:Kim}
			Let $G$ be a graph with maximum degree $d \in \N$. If $g(G) \geq 5$, then $\chi(G) \leq (1+o(1)) d/ \ln d$.
		\end{theo}
		
		Shortly after, Johansson~\cite{Johansson} reduced the girth requirement and extended Kim's result (modulo a constant factor) to triangle-free graphs.
		
		\begin{theo}[Johansson~\cite{Johansson}; see also~{\cite[Chapter~13]{MolloyReed}}]\label{theo:Johansson}
			Let $G$ be a graph with maximum degree $d \in \N$. If $g(G) \geq 4$, then $\chi(G) = O\left(d/\ln d\right)$.
		\end{theo}
		
		The proofs of Theorems~\ref{theo:Kim} and~\ref{theo:Johansson} are examples of a particular general approach to coloring problems. The key idea is to iterate applications of the~LLL so that on each stage, the~LLL produces only a partial coloring of~$G$---but this coloring is also made to satisfy some additional requirements. These requirements allow the process to be repeated, until finally the uncolored part of the graph becomes so sparse that a single application of the LLL (or a basic greedy algorithm) can finish the proof. Dealing with such iterated applications of the~LLL will be one of the major difficulties we will have to face in Section~\ref{section:groups}. An interested reader is referred to \cite{MolloyReed} for an excellent exposition of both proofs; we also discuss them briefly in Appendix~\ref{app:KJK} (omitting most of the details).\footnote{Recently, Molloy~\cite{Mol17} showed that the bound $\chi(G) \leq (1+o(1))\Delta(G)/\ln\Delta(G)$ from Theorem~\ref{theo:Kim} holds for triangle-free graphs as well. Unfortunately, the proof techniques used in~\cite{Mol17} cannot be adapted using our machinery.}
		
		\subsubsection*{Kahn's theorem}
		
		\mbox{}
		
		\smallskip
		
		\noindent As mentioned in Subsection~\ref{subsection:graphcol}, Vizing's theorem asserts that if $\Delta(G)$ is finite, then $\chi'(G) \leq \Delta(G)+1$. There are several known proofs of Vizing's theorem, none of them using the~LLL.
		
		An important generalization of graph coloring, so-called \emph{list coloring}, was introduced independently by Vizing~\cite{Vizing} and Erd\H os, Rubin, and Taylor~\cite{ERT}. Let $G$ be a graph with vertex set $X$. A \emph{list assignment} for $G$ is a function $L \colon X \to \powerset{Y}$, where $Y$ is a set and $\powerset{Y}$ denotes its powerset. An \emph{$L$-coloring} of $G$ is a map $f \colon X \to Y$ such that $f(x) \in L(x)$ for all $x \in X$. The \emph{list chromatic number} of~$G$ (notation: $\chi_\ell(G)$) is the smallest $k$ such that $G$ admits a proper $L$-coloring whenever $|L(x)| \geq k$ for all $x \in X$. Clearly, $\chi_\ell(G) \geq \chi(G)$ since if $L(x) = Y$ for all $x \in X$, then an $L$-coloring is simply a coloring with color set $Y$. Perhaps surprisingly, this inequality can be strict; in fact, there can be no upper bound on $\chi_\ell(G)$ in terms of $\chi(G)$, as there exist bipartite graphs with arbitrarily large list chromatic numbers.
		
		\emph{List edge colorings} and the \emph{list chromatic index} $\chi'_\ell(G)$ of a graph $G$ are defined similarly, \emph{mutatis mutandis}. The following conjecture is one of the major open problems in graph theory:
		\begin{conj}[\textbf{List Edge Coloring Conjecture};~{\cite[Conjecture~17.8]{BondyMurty}}]\label{conj:list}
			For every finite graph~$G$, \[\chi_\ell'(G) = \chi_\ell(G).\]
		\end{conj}
		
		As a step towards settling Conjecture~\ref{conj:list}, Kahn~\cite{Kahn} proved the following asymptotic version of Vizing's theorem for list colorings:
		
		\begin{theo}[Kahn~\cite{Kahn}; see also~{\cite[Chapter~14]{MolloyReed}}]\label{theo:Kahn}
			Let $G$ be a graph with maximum degree $d \in \N$. Then $\chi'_\ell(G) = d + o(d)$.
		\end{theo}
		
		Note that, in contrast to Vizing's theorem, Kahn's proof is based on the~LLL; in fact, it is similar to the proofs of Kim's and Johansson's theorems in that it uses iterated applications of the~LLL to produce partial colorings with some additional properties. Note that Kahn's theorem yields an LLL-based proof of the bound $\chi'(G) = d+o(d)$ for \emph{ordinary} edge colorings as well.
		
		\subsubsection*{Nonrepetitive and acyclic colorings}
		
		\mbox{}
		
		\smallskip
		
		\noindent The~LLL can be also applied to produce upper bounds on more ``exotic'' types of chromatic numbers. Here we only mention two examples. A nonempty finite sequence $s$ is \emph{nonrepetitive} if it cannot be decomposed as $s = u \concat v \concat v \concat w$ for some finite sequences $u$, $v$, $w$ with $v \neq \0$ (here $\concat$ denotes concatenation). A coloring $f$ of a graph $G$ on a set $X$ is \emph{nonrepetitive} if for any finite path $x_1$--\ldots--$x_k$ in $G$, the sequence $(f(x_1), \ldots, f(x_k))$ is nonrepetitive. Note that a nonrepetitive coloring is, in particular, proper since if $x\,G\, y$ and $f(x) = f(y) = c$, then the sequence $(f(x), f(y)) = (c,c)$, corresponding to the path $x$--$y$ of length one, is repetitive. The smallest number of colors necessary to color $G$ nonrepetitively is called the \emph{Thue number}\footnote{Thue initiated the study of nonrepetitive sequences. While it is easy to see that there are no nonrepetitive sequences of length $4$ over an alphabet of size $2$, Thue's famous theorem~\cite{Thue} asserts that there exist arbitrarily long nonrepetitive sequnces over an alphabet of size $3$.} of $G$ and is denoted by $\pi(G)$. The following theorem of Alon, Grytczuk, Ha\l{}uszczak, and Riordan~\cite{AGHR} gives an upper bound on $\pi(G)$ in terms of $\Delta(G)$:
		
		\begin{theo}[Alon--Grytczuk--Ha\l{}uszczak--Riordan~{\cite[Theorem~1]{AGHR}}]\label{theo:nonrep}
			Let $G$ be a graph with maximum degree $d \in \N$. Then $\pi(G) = O(d^2)$.
		\end{theo}
		
		A proper coloring $f$ of a graph $G$ is \emph{acyclic} if every cycle in $G$ receives at least three different colors. The least number of colors needed for an acyclic proper coloring of $G$ is called the \emph{acyclic chromatic number} of $G$ and is denoted by $a(G)$. In 1976, Erd\H os conjectured that $a(G) = o(\Delta(G)^2)$; 15 years later, Alon, McDiarmid, and Reed~\cite{AMcDR} confirmed Erd\H os's hypothesis.
		
		\begin{theo}[Alon--McDiarmid--Reed~{\cite[Theorem~1.1]{AMcDR}}]\label{theo:acyclic}
			Let $G$ be a graph with maximum degree $d \in \N$. Then $a(G) = O(d^{4/3})$.
		\end{theo}
		
		Each of Theorems~\ref{theo:nonrep} and~\ref{theo:acyclic} is proved via a single application of the~LLL to a carefully constructed correct instance.
		
		\subsection{Overview of our main results and the structure of the paper}
		
		\mbox{}
		
		\smallskip
		
		\noindent Let $X$ be a standard Borel space. An instance $\B$ over $X$ is \emph{Borel} if
		\[
			\bigcup \B \defeq \set{w \in \finf{X}{[0;1]} \,:\, w \in B \text{ for some } B \in \B}
		\]
		is a Borel subset of $\finf{X}{[0;1]}$. In general, given a correct Borel instance $\B$ over $X$, one cannot guarantee the existence of a Borel solution~\cite[Theorem~1.6]{CJMST-D}. Suppose, however, that $\mu$ is a probability Borel measure on $X$. When can one ensure that there is a ``large'' (in~terms of $\mu$) Borel subset of $X$ on which $\B$ admits a Borel solution?
		
		\subsubsection*{The Moser--Tardos theory}
		
		\mbox{}
		
		\smallskip
		
		\noindent In our investigation, we rely heavily on the algorithmic approach to the~LLL due to Moser and Tardos~\cite{MoserTardos}. The original motivation behind Moser and Tardos's work was to develop a randomized algorithm which, given a correct instance $\B$ over a finite set $X$, quickly finds a solution to $\B$. It turns out that the Moser--Tardos method naturally extends to the case when $X$ is infinite, leading to the possibility of analogs of the~LLL that are ``constructive'' in various senses; a notable example is the computable version of the~LLL due to Rumyantsev and Shen~\cite{RSh}. In Section~\ref{section:MoserTardos} we describe (a generalized version~of) the Moser--Tardos algorithm and consider its behavior in the Borel setting. The Moser--Tardos technique was first used in the measurable framework in~\cite{Kun}.
		
		\subsubsection*{A universal combinatorial structure---hereditarily finite sets}
		
		\mbox{}
		
		\smallskip
		
		\noindent By definition, an instance of the~LLL over a set $X$ puts a set of constraints on a map $f \colon X \to [0;1]$. For example, if $X$ is the vertex set of a graph~$G$, then by solving instances over $X$ one finds vertex colorings of $G$ with desired properties. However, sometimes we want to consider \emph{edge} colorings instead, or maybe maps defined on some other combinatorial structures ``built'' from~$G$, such as, say, paths of length $2$, or cycles, etc. Additionally, even when looking for vertex colorings, it is sometimes necessary to assign to each vertex several colors at once, which can be viewed as replacing every element of~$X$ by finitely many ``copies'' of it and coloring each ``copy'' independently. In order to cover all potential combinatorial applications, we enlarge the set $X$, adding points for various combinatorial data that can be built from the elements of $X$. We call the resulting ``universal'' combinatorial structure the \emph{amplification} of $X$ and denote it by $\HF(X)$ (here the letters ``HF'' stand for ``hereditarily finite''). Roughly speaking, the points of $\HF(X)$ correspond to all sets that can be obtained from~$X$ by repeatedly taking finite subsets. The precise construction of $\HF(X)$ is described in Section~\ref{sec:HF}. All our results are stated for instances over $\HF(X)$; however, to simplify the current discussion, we will be only talking about instances over $X$ in this subsection.
		
		\subsubsection*{Approximate LLL}
		
		\mbox{}
		
		\smallskip
		
		\noindent Our first main result is the approximate~LLL, which we state and prove in Section~\ref{sec:approx}. Let $X$ be a set. For an instance $\B$ over $X$ and a map $f \colon X \to [0;1]$, the \emph{defect} $\Def_\B(f)$ of $f$ with respect to $\B$ is the set of all $x \in X$ such that $x \in \dom(w)$ for some $w \in B \in \B$ with $w \subseteq f$. Thus, $f$ is a solution to $\B$ if and only if $\Def_\B(f) = \0$. An instance $\B$ is \emph{locally finite} if $\deg_\B(B) < \infty$ for all $B \in \B$. For locally finite instances, we prove the following:
		
		\begin{theocopy}{theo:approxLLL}[{\textbf{Approximate LLL}}]
			Let $\B$ be a correct locally finite Borel instance over a standard probability space $(X, \mu)$. Then for any $\epsilon > 0$, there exists a Borel function $f \colon X \to [0;1]$ with $\mu(\Def_\B(f)) \leq \epsilon$.
		\end{theocopy}
		
		Most (but not all) standard applications of the~LLL only consider locally finite instances; for example, any instance that is correct for the~SLLL is locally finite. Among the examples listed in \prg\ref{subsection:LLL}, Theorems~\ref{theo:hypcol}, \ref{theo:Kim}, \ref{theo:Johansson}, and \ref{theo:Kahn} only use locally finite instances; in particular, Theorem~\ref{theo:approxLLL} immediately yields Theorems~\ref{theo:approxVertex} and \ref{theo:approxEdge} on approximate chromatic numbers of Borel graphs. On the other hand, Theorems~\ref{theo:nonrep} and \ref{theo:acyclic} apply the~LLL to instances that are in general \emph{not} locally finite, as there can be infinitely many paths or cycles passing through a given vertex in a locally finite graph.
		
		We point out that in their recent study~\cite{CGMPT}, carried out independently from this work, Cs\'oka, Grabowski, M\'ath\'e, Pikhurko, and Tyros use an approach similar to ours in order to establish a purely Borel version of the~LLL for a class of instances satisfying stronger boundedness assumptions (namely having uniformly subexponential growth).
		
		\subsubsection*{Measure-preserving group actions}
		
		\mbox{}
		
		\smallskip
		
		\noindent Our second main result is the measurable version of the LLL for probability measure-preserving actions of countable groups, which we present in Section~\ref{section:groups}. It shows that under certain additional restrictions on the correct instance $\B$, one can find a Borel function that solves it on a conull subset---even when $\B$ is not locally finite. To motivate these restrictions, consider a graph $G$ on a set~$X$. Combinatorial problems related to $G$ usually require solving instances of the~LLL that possess the following two properties:
		\begin{itemize}
			\item[--] the correctness of a solution can be verified separately within each component of $G$;
			\item[--] the instance only depends on the graph structure of $G$, in other words, it is invariant under the (combinatorial/abstract) automorphisms of $G$. 
		\end{itemize}
		These two properties are captured in the following definition: Let $\alpha \colon \Gamma \acts (X, \mu)$ be a measure\-/preserving action of a countable group $\Gamma$ on a standard probability space $(X, \mu)$ and let $\iso_\alpha$ denote the set of all equivariant bijections $\phi \colon O \to O'$ between $\alpha$-orbits. An \emph{instance \ep{of the~LLL}} over $\alpha$ is a Borel instance $\B$ over $X$ such that:
		\begin{itemize}
			\item[--] for all $B \in \B$, $\dom(B)$ is contained within a single orbit of $\alpha$; and
			\item[--] the set $\B$ is ($\mu$-almost everywhere) invariant under the functions $\phi \in \iso_\alpha$.
		\end{itemize}
		A basic measurable version of the~LLL for probability measure-preserving group actions is as follows:
		
		\begin{corlcopy}{corl:no_iterations}
			Let $\alpha \colon \Gamma \acts (X, \mu)$ be a measure-preserving action of a countable group $\Gamma$ on a standard probability space $(X, \mu)$. Suppose that $\alpha$ factors to the shift action $\Gamma \acts ([0;1]^\Gamma, \lambda^\Gamma)$ and let $\B$ be a correct instance over $\alpha$. Then there exists a Borel function $f \colon X \to [0;1]$ with $\mu(\Def_\B(f)) =0$.
		\end{corlcopy}	
		
		Corollary~\ref{corl:no_iterations} is sufficient for many applications; for instance, it yields measurable analogs of Theorems~\ref{theo:nonrep} and \ref{theo:acyclic}. However, a more general result is required to derive Theorems~\ref{theo:col_large_g} and~\ref{theo:FreeEdge}. As mentioned in \prg\ref{subsection:LLL}, to establish their combinatorial counterparts (namely Theorems~\ref{theo:Kim}, \ref{theo:Johansson}, and~\ref{theo:Kahn}) the~LLL is applied iteratively to a series of instances, with each next instance defined \emph{using the solutions to the previous ones}: Even though the very first instance $\B_0$ is invariant under \emph{all} functions $\phi \in \iso_\alpha$, as soon as a solution $f_0$ to $\B_0$ is fixed, the next instance $\B_1$ is only guaranteed to be invariant under those $\phi \in \iso_\alpha$ that additionally \emph{preserve} the value of $f_0$, so Corollary~\ref{corl:no_iterations} can no longer be used.
		
		To formalize this complication, we define a game between two players, called the \emph{LLL~Game}. A~run of the LLL Game over an action $\alpha \colon \Gamma \acts (X, \mu)$ looks like this:
		
		\smallskip
		\begin{center}
		\begin{tabular}{ c || c | c | c | c | c | c | c | c }			
			Player~I  & $\B_0$ &       & $\B_1$ &       &  \ldots & $\B_n$ & & \ldots\\
			\hline
			Player~II &        & $f_0$ &        & $f_1$ &  \ldots &        & $f_n$ & \ldots\\
		\end{tabular}
		\end{center}
		\smallskip
		
		\noindent On his first turn, Player~I chooses a correct instance $\B_0$ over $\alpha$. Player~II responds by choosing a $\mu$\=/measurable solution $f_0$ to~$\B_0$. Player~I then picks a new correct Borel instance $\B_1$, this time only invariant under the functions $\phi \in \iso_\alpha$ that preserve~$f_0$. Player~II must respond by finding a $\mu$\=/measurable solution $f_1$ to~$\B_1$. On the next step, Player~I selects a correct Borel instance $\B_2$ invariant under the functions $\phi \in \iso_\alpha$ that preserve both $f_0$ and $f_1$; and so on. Player~II wins if the game continues indefinitely and loses if at any step, she is presented with an instance that has no $\mu$\=/measurable solution. Our result, Theorem~\ref{theo:Thm1}, asserts that Player~II has a winning strategy in this game:
		\begin{theocopy}{theo:Thm1}[{\textbf{Measurable LLL for group actions}}]
			Let $\alpha \colon \Gamma \acts (X, \mu)$ be a measure-preserving action of a countable group $\Gamma$ on a standard probability space $(X,\mu)$. If $\alpha$ factors to the shift action $\Gamma \acts ([0;1]^\Gamma, \lambda^\Gamma)$, then Player~II has a winning strategy in the LLL Game over $\alpha$.
		\end{theocopy}

		\subsubsection*{A partial converse}
		
		\mbox{}
		
		\smallskip
		
		\noindent Finally, we turn to the following natural question:
		\begin{leftbar}
			\noindent Is it necessary to assume that $\alpha$ admits a factor map to the $[0;1]$-shift action in order to establish Theorem~\ref{theo:Thm1} and Corollary~\ref{corl:no_iterations}, or is this assumption just an artifact of our proof?
		\end{leftbar}
		\noindent In Section~\ref{section:converse}, we demonstrate that, at least for amenable groups, this assumption is indeed necessary; furthermore, a probability measure\=/preserving free ergodic action $\alpha$ of a countably infinite amenable group~$\Gamma$ factors to the $[0;1]$-shift action \emph{if and only if} it satisfies the conclusion of Corollary~\ref{corl:no_iterations}. In fact, a much weaker version of the~LLL than Corollary~\ref{corl:no_iterations} already yields a factor map to the $[0;1]$-shift, which, in particular, shows that Theorem~\ref{theo:approxLLL} fails for instances that are not locally finite.
		
		To establish these results, we combine the tools of the Ornstein--Weiss theory of entropy for actions of amenable groups with concepts from computability theory. 
		By a theorem of Ornstein and Weiss, a free ergodic probability measure\=/preserving action $\alpha\colon \Gamma \acts (X, \mu)$ of a countably infinite amenable group $\Gamma$ factors to the $[0;1]$-shift action if and only if $H_\mu(\alpha) = \infty$, where $H_\mu(\alpha)$ is the so-called \emph{Kolmogorov--Sinai entropy} of $\alpha$. Intuitively, $H_\mu(\alpha)$ measures how ``unpredictable'' or ``random'' the interaction of $\alpha$ with a Borel map $f \colon X \to k \in \N$ can be. 
		Therefore, in proving a converse to Theorem~\ref{theo:Thm1}, we have to apply the~LLL in order to exhibit Borel functions $f$ whose behavior is highly ``random.'' Notice that entropy is a ``global'' parameter that depends on $f$ as a whole, while the~LLL can only constrain a function ``locally.'' In other words, we require a way to certify high entropy in a ``local,'' or ``pointwise,'' manner. To that end, we use \emph{Kolmogorov complexity}---a deterministic alternative to entropy defined in the language of computability theory---to measure the ``randomness'' of a given Borel function \emph{at each point}. The crux of our argument is Lemma~\ref{lemma:complexity_vs_entropy}, which is of independent interest. It gives a lower bound on the Kolmogorov--Sinai entropy of a Borel function in terms of the average value of its pointwise Kolmogorov complexity. The proof of Lemma~\ref{lemma:complexity_vs_entropy} invokes the result of Ornstein and Weiss concerning the existence of quasi-tilings in amenable groups and is inspired by previous work of Brudno~\cite{Brudno} in the case of $\Z$-actions.
		
	\subsection*{Acknowledgments}
	
	\mbox{}
	
	\smallskip
	
	\noindent This work is partially supported by the Illinois Distinguished Fellowship. I would like to thank Andy Zucker for his comments on an earlier version of this paper.
	I am grateful to Anush Tserunyan for introducing me to the field of Borel combinatorics and for her constant support and encouragement. I am also grateful to the anonymous referee for carefully reading the manuscript and providing helpful comments and suggestions.
		
		\section{Preliminaries}
		
		
		
		\noindent We use $\N \defeq \set{0,1,\ldots}$ to denote the set of all nonnegative integers and identify each $k \in \N$ with the set $\set{i \in \N \,:\, i < k}$. 
		A function~$f$ is identified with its graph, i.e., the set $\set{(x, y) \,:\, f(x) = y}$; this enables the use of standard set-theoretic notation, such as $\cup$, $\cap$, $\subseteq$, etc., for functions. In particular, $\0$ denotes the empty function as well as the empty set. For a function $f$ and a subset $S$ of its domain, $f\vert S$ denotes the restriction of $f$ to $S$. We write $f \colon X \rightharpoonup Y$ to indicate that $f$ is a partial function from $X$ to $Y$, i.e., a function of the form $f \colon X' \to Y$ with $X' \subseteq X$.
		
		Our standard references for descriptive set theory are~\cite{Kechris} and~\cite{Anush}. Below we only review the most basic facts and terminology used throughout the paper without mention.
		
		A \emph{standard Borel space} $(X, \mathfrak{B})$ is a set $X$ together with a $\sigma$-algebra $\mathfrak{B}$ of \emph{Borel sets} such that there is a compatible Polish (i.e., separable completely metrizable) topology $\tau$ on $X$ with $\mathfrak{B}$ as its $\sigma$-algebra of Borel sets. We will suppress the notation for the $\sigma$-algebra and denote a standard Borel space $(X, \mathfrak{B})$ simply by $X$. A function $f \colon X \to Y$ between standard Borel spaces $X$ and $Y$ is \emph{Borel} if $f$-preimages of Borel subsets of $Y$ are Borel in $X$. Due to the Borel isomorphism theorem~\cite[Theorem~13.10]{Anush}, all countable standard Borel spaces are discrete and all uncountable ones are isomorphic to each other.
		
		We use $\P(X)$ to denote the set of all probability Borel measures on a standard Borel space $X$. If $\mu \in \P(X)$, then the pair $(X, \mu)$ is called a \emph{standard probability space}. A measure $\mu \in \P(X)$ is \emph{atomless} if $\mu(\{x\}) = 0$ for all $x \in X$. The measure isomorphism theorem~\cite[Theorem~10.6]{Anush} asserts that all standard probability spaces $(X, \mu)$ with atomless $\mu$ are Borel isomorphic. If $X$ is a standard Borel space and $X' \subseteq X$ is a Borel set, then we identify $\P(X')$ with a subset of~$\P(X)$ in the natural way. In particular, given $\mu \in \P(X')$, we also use $\mu$ to denote the extension of $\mu$ to~$X$ (i.e., the pushforward $\iota_\ast(\mu)$ of $\mu$ under the inclusion map $\iota \colon X' \to X$); similarly, if $\mu \in \P(X)$ and $X'$ is $\mu$\=/conull, then we use $\mu$ to denote the restriction of $\mu$ to $X'$. The Lebesgue measure on the unit interval $[0;1]$ is denoted by $\lambda$.
		
		A subset $A$ of a standard Borel space $X$ is \emph{analytic} if it is the image of a Borel set under a Borel function. Somewhat informally, a set is analytic if it can be defined using existential (but not universal) quantifiers ranging over Borel sets. Analytic subsets of $X$ are universally measurable, i.e., $\mu$-measurable for every $\mu \in \P(X)$~\cite[Corollary~14.10]{Anush}. The complement of an analytic set is said to be \emph{co-analytic}. If a set is both analytic and co-analytic, then it is Borel~\cite[Corollary~12.7]{Anush}.
		
		Recall that for sets $X$ and $Y$,
		\[
		\begin{array}{rll}
		\text{--} & \fins{X} & \text{denotes the set of all finite subsets of $X$;}\\
		\text{--} & \finf{X}{Y} & \text{denotes the set of all partial functions $\phi \colon X \rightharpoonup Y$ with $\dom(\phi) \in \fins{X}$.}
		\end{array}
		\]
		If~$X$ is a standard Borel space, then $\fins{X}$ is also naturally equipped with a standard Borel structure.\footnote{One way to see this is to notice that if $\tau$ is a compatible Polish topology on $X$, then $\fins{X}$ is a Borel subset of $\mathcal{K}(X, \tau)$, the Polish space of all compact subsets of $(X, \tau)$ equipped with the Vietoris topology~\cite[Subsection~3.D]{Anush}.} For any standard Borel space $X$, there exists a Borel map $f \colon \fins{X}\setminus \set{\0} \to X$ such that $f(S) \in S$ for all $S \in [X]^{<\infty}\setminus \set{\0}$; for example, if $<$ is a Borel linear ordering of $X$ (which exists as $X$ is Borel isomorphic to a Borel subset of $\R$, say), then the function $S \mapsto \min_< S$ is Borel. If $X$ and $Y$ are standard Borel spaces, then $\finf{X}{Y}$ is also a standard Borel space, which can be identified with a Borel subset of $\fins{X \times Y}$.
		
		For sets $X$, $Y$, elements $x \in X$, $y \in Y$, and a subset $A \subseteq X \times Y$, we use the following notation:
		\[
			A_x \defeq \{y \in Y \,:\, (x, y) \in A\} \qquad\text{and}\qquad A^y \defeq \{x \in X \,:\, (x, y) \in A\}.
		\]
		
		The following fundamental result is used without mention:
		\begin{theo}[\textbf{Luzin--Novikov theorem};~{\cite[Theorem~18.10]{Kechris}}]\label{theo:LN}
			Let $X$ and $Y$ be standard Borel spaces and let $A \subseteq X \times Y$ be a Borel set such that for all $x \in X$, the set $A_x$ is countable. Then $A$ can be written as a countable union
			\[
				A = \bigcup_{n=0}^\infty A_n,
			\]
			where the sets $(A_n)_{n=0}^\infty$ are pairwise disjoint and for each $n \in \N$ and $x \in X$, $|(A_n)_x| \leq 1$. In particular, the set
			$\proj_X(A) \defeq \{x \in X \,:\, A_x \neq \0\}$
			is Borel.
		\end{theo}
		Informally, the Luzin--Novikov theorem implies that if a set is defined only using quantifiers ranging over countable sets, then it is Borel.
		
		On a couple of occasions, we will need the following fact.
		\begin{prop}[\textbf{Countable colorings of locally finite graphs}]\label{prop:lfcoloring}
			Let $G$ be a locally finite analytic graph on a standard Borel space $X$. Then $\chi_{\operatorname{B}}(G) \leq \aleph_0$.
		\end{prop}
		\begin{proof}
			Let $(B_n)_{n=0}^\infty$ be a countable family of Borel subsets of $X$ that separates points and is closed under complements and finite intersections. In particular, for any $x \in X$ and $S \subseteq X \setminus \{x\}$, if $S$ is finite, then there is $n \in \N$ such that $x \in B_n$ but $S \cap B_n = \0$.
			
			Define a set $Z \subseteq X \times \N$ as follows:
			\begin{align*}
			(x,n) \in Z \,\vcentcolon&\Longleftrightarrow\, x \in B_n \quad \text{and} \quad G_x \cap B_n = \0 \\
			\,&\Longleftrightarrow\, x \in B_n \quad \text{and} \quad \forall y \in B_n \,(y \not \in G_x).
			\end{align*}
			The second line in the above definition makes it clear that the set $Z$ is co-analytic. For all $x \in X$, there is $n \in \N$ such that $(x, n) \in Z$, so the Novikov separation theorem~\cite[Theorem~28.5]{Kechris} gives a Borel function $f \colon X \to \N$ such that for all $x \in X$, $(x, f(x)) \in Z$. Then $f$ is a Borel proper coloring of $G$.
		\end{proof}
		
		Proposition~\ref{prop:lfcoloring} also follows from the general characterization of analytic graphs with countable Borel chromatic numbers due to Kechris, Solecki, and Todorcevic~\cite[Theorem~6.3]{KechrisSoleckiTodorcevic}.
		
		\section{Moser--Tardos theory}\label{section:MoserTardos}
		
		\noindent As mentioned in the introduction, a major role in our arguments is played by ideas stemming from the algorithmic proof of the~LLL due to Moser and Tardos~\cite{MoserTardos}. In this section we review their method and introduce some convenient notation and terminology. Most results of this section are essentially present in~\cite{MoserTardos}; nevertheless, we include a fair amount of detail for completeness. Some proofs are deferred until Appendix~\ref{app:MoserTardos}.

		For the rest of this section, fix a set $X$ and a correct instance $\B$ over $X$. Motivated by algorithmic applications, Moser and Tardos only consider the case when the ground set $X$ is finite; however, their technique naturally extends to the case of infinite $X$.
		
		Let $\dom(\B) \defeq \set{\dom(B) \,:\, B \in \B}$. For the reasons explained in Remark~\ref{remk:empty}, we may assume that $\0 \not \in \dom(\B)$. For $S \in \dom(\B)$, define
		\[
			\B_S \defeq \bigcup \set{B \in \B \,:\, \dom(B) = S} = \set{w \colon S \to [0;1] \,:\, w \in B \text{ for some } B \in \B}.
		\]
		The correctness of $\B$ implies that the set $\set{B \in \B \,:\, \dom(B) = S}$ is countable. Therefore, $\B_S$ is a Borel subset of~$[0;1]^S$. For brevity, we write
		\[
			\mathbb{P}[S] \defeq \lambda^S(\B_S).
		\]
		(Note that this notation implicitly depends on $\B$.)
		
		We say that a family $A$ of sets is \emph{disjoint} if the elements of $A$ are pairwise disjoint.
		
		\begin{defn}[\textbf{Moser--Tardos process}]
			A \emph{table} is a map $\theta \colon X \times \N \to [0;1]$. Fix a table $\theta$ and consider the following inductive construction:
			\begin{leftbar}
					\noindent Set $t_0(x) \defeq 0$ for all $x \in X$.
					
					\smallskip
					
					\noindent {\sc Step $n \in \N$}: Define
					\[
						f_n(x) \defeq \theta(x, t_n(x)) \text{ for all } x \in X \qquad \text{and} \qquad A_n' \defeq \{S \in \dom(\B) \,:\, f_n \supseteq w \text{ for some } w \in \B_S\}.
					\]
					Choose $A_n$ to be an arbitrary maximal disjoint subset of~$A_n'$ and let
					\[
					t_{n+1}(x) \defeq \begin{cases}
					t_n(x) + 1 &\text{if $x \in S$ for some $S \in A_n$};\\
					t_n(x) &\text{otherwise}.
					\end{cases}
					\]
			\end{leftbar}
			\noindent A sequence $\mathcal{A} = (A_n)_{n=0}^\infty$ of subsets of~$\dom(\B)$ obtained via the above procedure is called a \emph{Moser--Tardos process} with input $\theta$.
		\end{defn}
		
		\begin{remk*}
			Since each set $A_n$ in a Moser--Tardos process is disjoint, for every $x \in X$ with $t_{n+1}(x) > t_n(x)$, there is a \emph{unique} set $S \in A_n$ such that $x \in S$.
		\end{remk*}

		\begin{prop}\label{prop:finitestep}
			Let $\mathcal{A} = (A_n)_{n=0}^\infty$ be a Moser--Tardos process. For $n \in \N$, let
			\[
				X_n \defeq \{x \in X\,:\, x \in S \text{ for some } S \in A_n\}.
			\]
			Then $f_n$ avoids all bad events $B \in \B$ with $\dom(B) \cap X_n = \0$.
		\end{prop}
		\begin{proof}
			If $\dom(B) \cap X_n = \0$, then $\dom(B)$ is disjoint from all $S \in A_n$. Since we assume $\dom(B) \neq \0$, this implies $\dom(B) \not \in A_n$. By the choice of $A_n$, we then get $\dom(B) \not \in A'_n$, as desired.
		\end{proof}
		
		Suppose that $\mathcal{A}$ is a Moser--Tardos process. By definition, the sequence $t_0(x)$, $t_1(x)$, \ldots{} is non\=/decreasing for all $x \in X$. We say that an element $x \in X$ is \emph{$\mathcal{A}$-stable} if the sequence $t_0(x)$, $t_1(x)$, \ldots{} is eventually constant. Let $\Stab(\mathcal{A})\subseteq X$ denote the set of all $\mathcal{A}$-stable elements of $X$. For $x \in \Stab(\mathcal{A})$, define
		\[
			t(x) \defeq \lim_{n \to \infty} t_n(x) \qquad \text{and} \qquad f(x) \defeq \theta(x, t(x)).
		\]
		We have the following limit analog of Proposition~\ref{prop:finitestep}:
		\begin{prop}\label{prop:infinitestep}
			Let $\mathcal{A} = (A_n)_{n=0}^\infty$ be a Moser--Tardos process. Then $f$ avoids all bad events $B \in \B$ with $\dom(B) \subseteq \Stab(\mathcal{A})$.
		\end{prop}
		\begin{proof}
			Fix $B \in \B$ with $\dom(B) \subseteq \Stab(\mathcal{A})$ and choose $n \in \N$ so large that for all $x \in \dom(B)$, we have $t(x) = t_n(x)$. Then $f \vert \dom(B) = f_n \vert \dom(B)$, and thus it remains to show that $f_n$ avoids $B$. Notice that $\dom(B)$ is disjoint from all $S \in A_n$; indeed, if $x \in \dom(B) \cap S$ for some $S \in A_n$, then $t_{n + 1}(x) = t_n(x) + 1$, which contradicts the choice of $n$. Now we are done by Proposition~\ref{prop:finitestep}.
		\end{proof}
		
		For each $S \in \dom(\B)$, define the \emph{index} $\mathbf{Ind}(S, \mathcal{A}) \in \N \cup \set{\infty}$ of $S$ in $\mathcal{A}$ by
		\[
			\mathbf{Ind}(S, \mathcal{A}) \defeq |\{n \in \N\,:\, S \in A_n\}|.
		\]
		Note that for all $x \in X$,
		\begin{equation}\label{eq:stabilizing}
			\lim_{n \to \infty} t_n(x) = \sum_{S \in \dom(\B)\,:\, S \ni x} \mathbf{Ind}(S, \mathcal{A}),
		\end{equation}
		so $x \in \Stab(\mathcal{A})$ if and only if the expression on the right hand side of \eqref{eq:stabilizing} is finite. Our goal therefore is to obtain good upper bounds on the numbers $\mathbf{Ind}(S, \mathcal{A})$. To that end, we look at certain patterns in the table $\theta$.
		
		A \emph{pile} is a nonempty finite set $\mathscr{P}$ of functions of the form $\tau \colon S \to \N$ with $S \in \dom(\B)$, satisfying the following requirements:
		\begin{itemize}
			\item[--] the graphs of the elements of $\mathscr{P}$ are pairwise disjoint; in other words, for every pair of distinct functions $\tau$, $\tau' \in \mathscr{P}$ and for each $x \in \dom(\tau) \cap \dom(\tau')$, we have $\tau(x) \neq \tau'(x)$;
			\item[--] for every $\tau \in \mathscr{P}$ and $x \in \dom(\tau)$, either $\tau(x) = 0$, or else, there is $\tau' \in \mathscr{P}$ with $x \in \dom(\tau')$ and $\tau'(x) = \tau(x)-1$.
		\end{itemize}
		The \emph{support} of a pile $\mathscr{P}$ is the set
		\[
			\supp(\mathscr{P}) \defeq \bigcup_{\tau \in \mathscr{P}} \dom(\tau).
		\]
		Note that $\supp(\mathscr{P})$ is a finite subset of $X$.
		
			\begin{figure}[h]
				\centering	
				\begin{tikzpicture}[scale=0.6]
				
				\draw (0, -1) -- (0,6);
				\draw (-1, 0) -- (10,0);
				
				\node at (-0.5, 1) {$0$};
				\node at (-0.5, 3) {$1$};
				\node at (-0.5, 5) {$2$};
				
				\node[circle, draw, inner sep=1pt, outer sep=0](10) at (1,1) {$\tau_1$};
				\node[circle, draw, inner sep=1pt, outer sep=0](20) at (3,1) {$\tau_1$};
				\node[circle, draw, inner sep=1pt, outer sep=0](21) at (3,3) {$\tau_3$};
				\node[circle, draw, inner sep=1pt, outer sep=0](30) at (5,1) {$\tau_3$};
				\node[circle, draw, inner sep=1pt, outer sep=0](31) at (5,3) {$\tau_4$};
				\node[circle, draw, inner sep=1pt, outer sep=0](32) at (5,5) {$\tau_5$};
				\node[circle, draw, inner sep=1pt, outer sep=0](40) at (7,1) {$\tau_2$};
				\node[circle, draw, inner sep=1pt, outer sep=0](41) at (7,3) {$\tau_3$};
				\node[circle, draw, inner sep=1pt, outer sep=0](42) at (7,5) {$\tau_4$};
				\node[circle, draw, inner sep=1pt, outer sep=0](50) at (9,1) {$\tau_2$};
				\node[circle, draw, inner sep=1pt, outer sep=0](51) at (9,3) {$\tau_4$};
				
				\node at (1,-0.5) {$x_1$};
				\node at (3,-0.5) {$x_2$};
				\node at (5,-0.5) {$x_3$};
				\node at (7,-0.5) {$x_4$};
				\node at (9,-0.5) {$x_5$};
				
				\node[anchor=west] at (10, 0) {$X$};
				\node[anchor=south] at (0, 6) {$\N$};
				
				\draw (10) -- (20);
				\draw (21) -- (30) -- (41);
				\draw (31) -- (42) -- (51);
				\draw (40) -- (50);
				
				\end{tikzpicture}
				\caption{$\mathscr{P} = \set{\tau_1, \tau_2, \tau_3, \tau_4, \tau_5}$ is a neat pile of height $4$ with $\supp(\mathscr{P}) = \set{x_1, x_2, x_3, x_4, x_5}$ and $\Top(\mathscr{P}) = \set{\tau_5}$.
				}
			\end{figure}
		
		Let $\mathscr{P}$ be a pile and let $\tau$, $\tau' \in \mathscr{P}$. We say that $\tau'$ \emph{supports} $\tau$, in symbols $\tau' \prec \tau$,
		if there is an element $x \in \dom(\tau) \cap \dom(\tau')$ such that $\tau'(x) = \tau(x) - 1$. A pile $\mathscr{P}$ is \emph{neat} if there does not exist a sequence of functions $\tau_1$, $\tau_2$, \ldots, $\tau_k \in \mathscr{P}$ with $k \geq 2$ such that $\tau_1 \prec \tau_2 \prec \ldots \prec \tau_k \prec \tau_1$. Equivalently, $\mathscr{P}$ is neat if the transitive closure of the relation $\prec$ on $\mathscr{P}$ is a (strict) partial order.
		
		A \emph{top element} in a pile $\mathscr{P}$ is any $\tau \in \mathscr{P}$ for which there is no $\tau' \in \mathscr{P}$ with $\tau \prec \tau'$. The set of all top elements in $\mathscr{P}$ is denoted $\Top(\mathscr{P})$. Notice that if $\mathscr{P}$ is a neat pile, then $\Top(\mathscr{P}) \neq \0$. The \emph{height} $h(\mathscr{P})$ of a neat pile $\mathscr{P}$ is the largest $k \in \N$ such that there is a sequence $\tau_1$, \ldots, $\tau_k \in \mathscr{P}$ with $\tau_1 \prec \ldots \prec \tau_k$ (so necessarily $h(\mathscr{P}) \geq 1$).
		
		We say that a pile $\mathscr{P}$ \emph{appears} in a table $\theta \colon X \times \N \to [0;1]$ if for all $\tau \in \mathscr{P}$, the map
		\[
		\dom(\tau) \to [0;1] \colon x \mapsto \theta(x, \tau(x))
		\]
		belongs to $\B_{\dom(\tau)}$. For $S \in \dom(\B)$, let $\mathbf{Piles}(S)$ denote the set of all neat piles $\mathscr{P}$ with $\Top(\mathscr{P}) = \set{\tau}$ such that the unique top element $\tau$ of $\mathscr{P}$ satisfies $\dom(\tau) = S$. The \emph{index} $\mathbf{Ind}(S, \theta) \in \N \cup \set{\infty}$ of $S$ in~$\theta$ is defined by
		\[
			\mathbf{Ind}(S, \theta) \defeq \left|\set{\mathscr{P} \in \mathbf{Piles}(S)\,:\, \mathscr{P} \text{ appears in } \theta}\right|.
		\]
		The next proposition asserts that $\mathbf{Ind}(S, \theta) \geq \mathbf{Ind}(S, \mathcal{A})$ for any Moser--Tardos process $\mathcal{A}$ with input $\theta$:
		
		\begin{prop}\label{prop:reduction}
			Let $\mathcal{A} = (A_n)_{n=0}^\infty$ be a Moser--Tardos process with input $\theta$ and let $S \in \dom(\B)$. If $n \in \N$ is such that $S \in A'_n$, then there exists a neat pile $\mathscr{P} \in \mathbf{Piles}(S)$ of height precisely $n+1$ that appears in~$\theta$. In particular, $\mathbf{Ind}(S, \mathcal{A}) \leq \mathbf{Ind}(S, \theta)$.
		\end{prop}
		\begin{proof}
			The ``in particular'' part follows, since for different $n$ with $S \in A_n'$, the neat piles given by the first part of the proposition are distinct (they have distinct heights).
			
			To prove the main statement, fix $S \in \dom(\B)$ and $n \in \N$ with $S \in A_n'$. Build $\mathscr{P}$ by ``tracing back'' the steps of the Moser--Tardos process as follows. Start by setting $\mathscr{P}_0$ to be the one-element set $\set{t_n \vert S}$ and let $R_0 \defeq S$. If $k < n$, then, after $R_k \subseteq X$ is determined, let $\mathscr{P}_{k+1}$ be the family of all maps of the form $t_{n-k-1} \vert S'$, where $S'$ is an element of $A_{n-k-1}$ such that $S' \cap R_k \neq \0$, and let $R_{k+1} \defeq R_k \cup \bigcup_{\tau \in \mathscr{P}_{k+1}} \dom(\tau)$. 
			Finally, let $\mathscr{P} \defeq \mathscr{P}_0 \cup \ldots \cup \mathscr{P}_n$. It is straightforward to check that $\mathscr{P}$ is a neat pile with support $R_n$ that has all the desired properties.
		\end{proof}
		
		Given a table $\theta \colon X \times \N \to [0;1]$, we say that an element $x \in X$ is \emph{$\theta$-stable} if
		\[
			\sum_{S \in \dom(\B)\,:\, S \ni x} \mathbf{Ind}(S, \theta) < \infty.
		\]
		The set of all $\theta$-stable elements is denoted $\Stab(\theta)$. Due to Proposition~\ref{prop:reduction}, $\Stab(\theta) \subseteq\Stab(\mathcal{A})$ for every Moser--Tardos process $\mathcal{A}$ with input $\theta$.
		
		Now the strategy is to switch the order of summation and, instead of counting how many piles from $\mathbf{Piles}(S)$ appear in a particular table $\theta$, fix a pile $\mathscr{P}$ and estimate the probability that $\mathscr{P}$ appears in a table $\theta$ chosen \emph{at random}. For a given pile $\mathscr{P}$, the restriction of $\theta$ to $\supp(\mathscr{P}) \times \N$ fully determines whether $\mathscr{P}$ appears in $\theta$ or not. Thus, we may let $\mathbf{App}(\mathscr{P}) \subseteq [0;1]^{\supp(\mathscr{P}) \times \N}$ be the set such that
		\[
			\mathscr{P} \text{ appears in } \theta \,\Longleftrightarrow\, \theta\vert (\supp(\mathscr{P}) \times \N) \in \mathbf{App}(\mathscr{P}).
		\]
		It is easy to see that the set $\mathbf{App}(\mathscr{P})$ is Borel. Since the graphs of the elements of $\mathscr{P}$ are pairwise disjoint, there is a simple expression for the Lebesgue measure of $\mathbf{App}(\mathscr{P})$; namely, we have
		\[
			\lambda^{\supp(\mathscr{P}) \times \N} \left(\mathbf{App}(\mathscr{P})\right) = \prod_{\tau \in \mathscr{P}} \mathbb{P}[\dom(\tau)].
		\]
		Now we are ready to state the cornerstone result of Moser--Tardos theory:
		\begin{theo}\label{theo:MoserTardos}
			Let $\omega \colon \B \to [0;1)$ be a function witnessing the correctness of $\B$ and let $S \in \dom(\B)$. Then
			\begin{equation}\label{eq:MoserTardos}
				\sum_{\mathscr{P} \in \mathbf{Piles}(S)} \lambda^{\supp(\mathscr{P})\times \N} \left(\mathbf{App}(\mathscr{P})\right) \leq \sum_{\substack{B \in \B \,:\\\dom(B) = S}}\frac{\omega(B)}{1 - \omega(B)}.
			\end{equation}
		\end{theo}
		The proof of Theorem~\ref{theo:MoserTardos} is given in Appendix~\ref{app:MoserTardos}. The following corollary is immediate:
		\begin{corl}\label{corl:summation}
			For all $x \in X$, we have
			\[
				\sum_{\substack{S \in \dom(\B)\,:\\S \ni x}}\, \sum_{\mathscr{P} \in \mathbf{Piles}(S)} \lambda^{\supp(\mathscr{P})\times \N} \left(\mathbf{App}(\mathscr{P})\right) < \infty.
			\]
		\end{corl}
		\begin{proof}
			Let $\omega \colon \B \to [0;1)$ witness the correctness of $\B$. Due to Theorem~\ref{theo:MoserTardos}, it suffices to check that the sum
			\begin{equation}\label{eq:sum}
				\sum_{\substack{S \in \dom(\B)\,:\\S \ni x}} \, \sum_{\substack{B \in \B \,:\\\dom(B) = S}}\frac{\omega(B)}{1 - \omega(B)} = \sum_{\substack{B \in \B \,:\\ \dom(B) \ni x}} \frac{\omega(B)}{1 - \omega(B)}
			\end{equation}
			is finite. We may assume that $\omega(B) = 0$ whenever $\mathbb{P}[B] = 0$. If for all $B \in \B$ with $x \in \dom(B)$, we have $\mathbb{P}[B] = 0$, then the sum \eqref{eq:sum} is $0$ (hence finite). Otherwise, for some $B_0 \in \B$ with $x \in \dom(B)$, we have $\mathbb{P}[B_0] > 0$, and thus the correctness of $\B$ implies
			\[
			\prod_{B \in \Nbhd_\B(B_0)} (1 - \omega(B)) > 0.
			\]
			Therefore,
			\begin{equation}\label{eq:finite}
			\sum_{\substack{B \in \B \,:\\ \dom(B) \ni x}} \omega(B) \leq \sum_{B \in \Nbhd_\B(B_0)} \omega(B) < \infty.
			\end{equation}
			In particular, for all but finitely many events $B \in \B$ with $x \in \dom(B)$, we have $\omega(B) \leq 1/2$, so
			\[
			\frac{\omega(B)}{1 - \omega(B)} \leq 2\omega(B).
			\]
			Together with~\eqref{eq:finite}, this shows that the sum \eqref{eq:sum} is finite, as desired.
		\end{proof}
		
		The next corollary considers the case when the table $\theta$ is chosen randomly from $[0;1]^{X \times \N}$. (Note that the product probability space $([0;1]^{X \times \N}, \lambda^{X \times \N})$ is standard only if $X$ is countable.)
		
		\begin{corl}\label{corl:countable}
			For each $x \in X$, we have
			\[
				\int_{[0;1]^{X \times \N}}\sum_{\substack{S \in \dom(\B)\,:\\ S \ni x}} \mathbf{Ind}(S, \theta) \, \D\lambda^{X \times \N}(\theta) < \infty.
			\]
			In particular,
			\[
				\lambda^{X \times \N} (\set{\theta \in [0;1]^{X \times \N}\,:\, x \in \Stab(\theta)}) = 1.	
			\]
		\end{corl}
		\begin{proof}
			Corollary~\ref{corl:summation} yields
			\begin{align*}
			\int_{[0;1]^{X \times \N}}\sum_{\substack{S \in \dom(\B)\,:\\ S \ni x}} \mathbf{Ind}(S, \theta) \, \D\lambda^{X \times \N}(\theta)
			&= \sum_{\substack{S \in \dom(\B)\,:\\ S \ni x}} \int_{[0;1]^{X \times \N}} \mathbf{Ind}(S, \theta) \, \D\lambda^{X \times \N}(\theta) \\
			&= \sum_{\substack{S \in \dom(\B)\,:\\S \ni x}}\, \sum_{\mathscr{P} \in \mathbf{Piles}(S)} \lambda^{\supp(\mathscr{P})\times \N} \left(\mathbf{App}(\mathscr{P})\right)  <\infty. \qedhere
			\end{align*}
		\end{proof}
		
		We can now deduce the~LLL in the form of Theorem~\ref{theo:LLLvbls}. Since the set $\Nbhd_\B(B)$ is countable for each $B \in \B$, we may, without loss of generality, assume that $X$ is countable. By Corollary~\ref{corl:countable}, each $x \in X$ satisfies
		\[
			\lambda^{X \times \N} (\set{\theta \in [0;1]^{X \times \N}\,:\, x \in \Stab(\theta)}) = 1.	
		\]
		As $X$ is countable, we obtain
		\[
			\lambda^{X \times \N} (\set{\theta \in [0;1]^{X \times \N}\,:\, X = \Stab(\theta)}) = 1.
		\]
		Choose any $\theta$ such that $X = \Stab(\theta)$ and let $\mathcal{A}$ be any Moser--Tardos process with input~$\theta$. Then $\Stab(\mathcal{A}) = X$ and Theorem~\ref{theo:LLLvbls} follows from Proposition~\ref{prop:infinitestep}.
		
		\subsection{Moser--Tardos theory in the Borel setting}\label{subsec:BorelMT}
		
		\mbox{}
		
		\smallskip
		
		
		\noindent Let $X$ be a standard Borel space. Recall that an instance $\B$ over $X$ is \emph{Borel} if $\bigcup \B$ is a Borel subset of $\finf{X}{[0;1]}$. Notice that if $\B$ is a Borel instance over $X$, then $\dom(\B)$ is an analytic subset of~$\fins{X}$.\footnote{In most applications, each bad event $B \in \B$ has positive probability. If that is the case, then $\dom(\B)$ is actually a Borel subset of $\fins{X}$ due to the ``large section'' uniformization theorem~\cite[Corollary~18.7]{Kechris}.} A Moser--Tardos process $\mathcal{A} = (A_n)_{n=0}^\infty$ with Borel input $\theta \colon X \times \N \to [0;1]$ is \emph{Borel} if each $A_n$ is a Borel subset of $\fins{X}$. Note that if $\mathcal{A}$ is a Borel Moser--Tardos process, then the associated maps $t_n \colon X \to \N$ and $f_n \colon X \to [0;1]$ are Borel.
		
		\begin{prop}[\textbf{Borel Moser--Tardos processes}]\label{prop:BorelMoserTardos}
			Let $X$ be a standard Borel space and let $\B$ be a correct Borel instance over~$X$. Let $\theta \colon X \times \N \to [0;1]$ be a Borel table. Then there exists a Borel Moser--Tardos process $\mathcal{A}$ with input $\theta$.
		\end{prop}
		\begin{proof}
			We use the following result of Kechris and Miller:
			
			\begin{lemma}[Kechris--Miller~{\cite[Lemma~7.3]{KechrisMiller}}; \textbf{maximal disjoint subfamilies}]\label{lemma:disjoint}
				Let $X$ be a standard Borel space and let $A \subseteq \fins{X}$ be a Borel set such that for every $x \in X$, the set $\set{S \in A\,:\, x \in S}$ is countable. Then there is a Borel maximal disjoint subset $A_0 \subseteq A$.
			\end{lemma}
			
			On Step~$n$ of the Moser--Tardos process, we are given a Borel map $f_n \colon X \to [0;1]$, so the set
			\[
				A_n' \defeq \set{S \in \dom(\B)\,:\, f_n \supseteq w \text{ for some } w \in \B_S} = \set{S \in \dom(\B)\,:\, f_n\vert S \in \textstyle\bigcup \B}
			\]
			is Borel. Hence, we can use Lemma~\ref{lemma:disjoint} to pick a Borel maximal disjoint subset $A_n \subseteq A_n'$.
		\end{proof}

		\section{Hereditarily finite sets}\label{sec:HF}
		
		
		\noindent In this section we describe the construction of a ``universal'' combinatorial structure over a space~$X$, whose points encode various combinatorial data that can be built from the elements of $X$.
		
		The set $\HF_\0(X)$ of all \emph{hereditarily finite sets} over $X$ is defined inductively as follows\footnote{Here we treat the points of $X$ as \emph{urelements}, i.e., not sets. Formally, we can replace $X$ with, say, the diagonal \[\Delta_X^\N \defeq \{(x, x, x, \ldots)\,:\, x \in X\}\subseteq X^\N,\] ensuring that no point in $X$ is a finite set.}:
		\begin{itemize}
			\item[--] $\HF^{(0)}(X) \defeq X$;
			\item[--] $\HF^{(n+1)}(X) \defeq \HF^{(n)}(X) \cup \fins{\HF^{(n)}(X)}$ for all $n \in \N$;
			\item[--] $\HF_\0(X) \defeq \bigcup_{n = 0}^\infty \HF^{(n)}(X)$ (note that this union is increasing).
		\end{itemize}
		In other words, $\HF_\0(X)$ is the smallest set containing $X$ that is closed under taking finite subsets. For $h \in \HF_\0(X)$, the \emph{underlying set} of $h$, in symbols $\U(h)$, is defined inductively by:
		\begin{itemize}
			\item[--] for $x \in X$, \[\U(x) \defeq \set{x};\]
			\item[--] for $h \in \HF^{(n+1)}(X) \setminus \HF^{(n)}(X)$, \[\U(h) \defeq \bigcup_{h' \in h} \U(h').\]
		\end{itemize}
		Equivalently, $\U(h)$ is the smallest subset $S$ of $X$ such that $h \in \HF_\0(S)$. The \emph{amplification} of $X$ is defined to be
		\[
			\HF(X) \defeq \set{h \in \HF_\0(X)\,:\, \U(h) \neq \0}.
		\]
		If $X$ is a standard Borel space, then so are $\HF_\0(X)$ and $\HF(X)$. The space $\HF(X)$ encodes the ``combinatorics'' of~$X$. For instance, $\HF(X)$ contains (as Borel subsets) the space $X^{<\infty}$ of all nonempty finite sequences of elements of $X$ and the space $X \times \N$, i.e., the union of countably many disjoint copies of $X$.\footnote{To embed $\N$ in $\HF_\0(X)$, we use the standard von Neumann convention $0 = \0$, $1 = \set{\0}$, $2 = \set{\0, \set{\0}}$, etc.} In fact, $\HF(X) \supseteq \HF(X) \times \N$, i.e., $\HF(X)$ contains ``countably many disjoint copies of itself.'' If $G$ is a Borel graph on $X$, then the edge set of~$G$, i.e., the set $\mathcal{E}(G) \defeq \set{\set{x, y}\,:\, x \,G\, y}$, is also a Borel subset of $\HF(X)$. So are other, more complicated, objects associated with~$G$. For instance, the set of all cycles in $G$, i.e., the set of all finite subsets $C \subseteq \mathcal{E}(G)$ whose elements form a cycle, is a Borel subset of $\HF(X)$.
		
		If $X'$ is a Borel subset of $\HF(X)$, then the inclusions
		\[
			\fins{X'}\subseteq \fins{\HF(X)} \quad \text{and} \quad \finf{X'}{[0;1]} \subseteq \finf{\HF(X)}{[0;1]}
		\]
		are Borel as well. Therefore, a Borel instance of the~LLL over $X'$ is also a Borel instance over $\HF(X)$. Because of that, we will restrict our attention to instances over $\HF(X)$, and this will include various combinatorial applications such as vertex coloring or edge coloring.
		
		Functions between sets naturally lift to functions between their amplifications. Namely, given a map $\phi\colon X \to Y$, define $\tilde{\phi}_\0 \colon \HF_\0(X) \to \HF_\0(Y)$ inductively~via:
		\begin{itemize}
			\item[--] for $x \in X$, \[\tilde{\phi}_\0(x) \defeq \varphi(x);\]
			\item[--] for $h \in \HF^{(n+1)}(X) \setminus \HF^{(n)}(X)$, \[\tilde{\phi}_\0(h) \defeq \set{\tilde{\phi}_\0(h')\,:\, h' \in h}.\]
		\end{itemize}
		The \emph{amplification} of $\phi$ is the map $\tilde{\phi} \colon \HF(X) \to \HF(Y)$ given by
		\[
			\tilde{\phi} \defeq \tilde{\phi}_\0 \vert \HF(X).
		\]
		For $S \in \fins{X}\setminus \set{\0}$, we have $\tilde{\phi}(S) = \phi(S)$ (where $\phi(S)$ denotes, as usual, the image of $S$ under $\phi$). If $\varphi$ is injective (resp. surjective), then~$\tilde{\phi}$ is also injective (resp. surjective). 
		
		
		\section{Approximate LLL}\label{sec:approx}
		
		\noindent In this section we state and prove our first main result: the approximate LLL for Borel instances.
		
		Let $(X, \mu)$ be a standard probability space. Suppose that $\B$ is a Borel instance over $\HF(X)$. For each $x \in X$, consider the following set:
		\[
			\sh_x(\B) \defeq \set{S \in \dom(\B) \,:\, x \in \U(h) \text{ for some } h \in S}.
		\]
		We call $\sh_x(\B)$ the \emph{shadow} of $\B$ over $x$. We say that $\B$ is \emph{hereditarily locally finite} if $\sh_x(\B)$ is finite for all $x \in X$. For a Borel map $f \colon \HF(X) \to [0;1]$, its \emph{defect} with respect to $\B$ is the set
		\[
			\Def_\B(f) \defeq \set{x \in X \,:\, f\vert S \in \textstyle\bigcup\B \text{ for some } S \in \sh_x(\B)}.
		\]
		Note that if $B$ is hereditarily locally finite, then $\Def_\B(f)$ is a Borel subset of $X$.
		
		\begin{theo}[\textbf{Approximate LLL}]\label{theo:approxLLL}
			Let $(X, \mu)$ be a standard probability space and let $\B$ be a hereditarily locally finite correct Borel instance over $\HF(X)$. Then for any $\epsilon > 0$, there is a Borel map $f \colon \HF(X) \to [0;1]$ with $\mu(\Def_\B(f)) \leq \epsilon$.
		\end{theo}
		
		\subsection{Proof of Theorem~\ref{theo:approxLLL}}\label{subsection:ProofApprox}
		
		Let $(X, \mu)$ be a standard probability space and let $\B$ be a hereditarily locally finite correct Borel instance over $\HF(X)$. Fix $\epsilon > 0$. For $S \in \dom(\B)$ and $n \in \N$, let $\mathbf{Piles}_n(S)$ denote the set of all neat piles $\mathscr{P} \in \mathbf{Piles}(S)$ of height precisely $n+1$. In particular, we have
		\begin{equation*}
			\mathbf{Piles}(S) = \bigcup_{n=0}^\infty \mathbf{Piles}_n(S),
		\end{equation*}
		and the above union is disjoint. For $n \in \N$, let $D_n$ denote the set of all $x \in X$ such that
		\[
			\sum_{S \in \sh_x(\B)}\, \sum_{\mathscr{P} \in \mathbf{Piles}_n(S)} \lambda^{\supp(\mathscr{P})\times \N} \left(\mathbf{App}(\mathscr{P})\right) > \epsilon/2.
		\]
		It is clear from the definition that the set $D_n$ is analytic; in particular, it is $\mu$-measurable.\footnote{In fact, $D_n$ is Borel. Indeed, if there is $\tau \in \mathscr{P}$ with $\mathbb{P}[\dom(\tau)] = 0$, then $\lambda^{\supp(\mathscr{P})\times \N} \left(\mathbf{App}(\mathscr{P})\right) = 0$; and the set $\set{S \in \fins{\HF(X)} \,:\, \mathbb{P}[S] > 0 }$ is Borel due to the ``large section'' uniformization theorem~\cite[Corollary~18.7]{Kechris}.}
		Due to Corollary~\ref{corl:summation} and the fact that $\B$ is hereditarily locally finite, each $x \in X$ satisfies
		\[
			\sum_{S \in \sh_x(\B)}\, \sum_{\mathscr{P} \in \mathbf{Piles}(S)} \lambda^{\supp(\mathscr{P})\times \N} \left(\mathbf{App}(\mathscr{P})\right) < \infty.
		\]
		Hence we can choose $N \in \N$ so large that $\mu(D_N) \leq \epsilon/2$.
		
		Let $G$ be the graph on $\HF(X)$ given by
		\[
			(H_1, H_2) \in G \,\vcentcolon\Longleftrightarrow\, H_1 \neq H_2 \text{ and } \set{H_1, H_2} \subseteq S \text{ for some } S \in \dom(\B).
		\]
		Clearly, $G$ is analytic. Since $\B$ is hereditarily locally finite, $G$ is locally finite. For $n \in \N$, let~$G^n$ denote the analytic graph on $\HF(X)$ in which distinct elements $H_1$, $H_2 \in \HF(X)$ are adjacent if and only if $G$ contains a path of length at most $n$ joining $H_1$ and $H_2$ (in particular, $G^1 = G$). Since $G$ is locally finite, so is~$G^n$ for each $n \in \N$. By Proposition~\ref{prop:lfcoloring}, $\chi_{\operatorname{B}}(G^n) \leq \aleph_0$ for all $n \in \N$, so let $c \colon \HF(X) \to \N$ be a Borel proper coloring of $G^{2(N+1)}$.
		
		For a function $\theta \colon \N \times \N \to [0;1]$, define a map $\theta_c$ by
		\[
			\theta_c \colon \HF(X) \times \N \to [0;1] \colon (x, n) \mapsto \theta(c(x), n).
		\]
		Note that $\theta_c$ is a Borel table in the sense of the Moser--Tardos algorithm on $\HF(X)$. Let $Q$ be the set of all pairs $(x, \theta)$ with $x \in X$ and $\theta \colon \N \times \N \to [0;1]$ such that
		\[
			\text{there exist } S \in \sh_x(\B) \text{ and } \mathscr{P} \in \mathbf{Piles}_N(S) \text{ such that } \mathscr{P} \text{ appears in } \theta_c.
		\]
		By definition, $Q$ is an analytic subset of $X \times [0;1]^{\N \times \N}$.\footnote{Again, one can show that $Q$ is actually Borel.} Recall that for $x \in X$ and $\theta \colon \N \times \N \to [0;1]$, we use $Q_x$ and $Q^\theta$ to denote the corresponding fibers of $Q$.
		
		\begin{lemma}\label{lemma:smallfibers}
			For all $x \in X \setminus D_N$, we have
		$\lambda^{\N \times \N} (Q_x) \leq \epsilon/2
		$.
		\end{lemma}
		\begin{proof}
			If $\mathscr{P}$ is a neat pile with a unique top element $\tau$, then for every $\tau' \in \mathscr{P}$, there exists a sequence $\tau_1$, \ldots, $\tau_k \in \mathscr{P}$ such that $\tau_1 = \tau'$, $\tau_k=\tau$, and $\tau_1 \prec \ldots \prec \tau_k$. In particular, $\dom(\tau_i) \cap \dom(\tau_{i+1}) \neq \0$ for all $1 \leq i < k$, so the distance in $G$ between any element of $\dom(\tau')$ and any element of $\dom(\tau)$ is at most $k \leq h(\mathscr{P})$. Therefore, the distance in $G$ between any two elements of $\supp(\mathscr{P})$ is at most~$2h(\mathscr{P})$.
			
			Fix any $x \in X \setminus D_N$ and let $S \in \sh_x(\B)$ and $\mathscr{P} \in \mathbf{Piles}_N(S)$. Since $h(\mathscr{P}) = N+1$, the distance in $G$ between any two elements of $\supp(\mathscr{P})$ is at most $2(N+1)$; in other words, any two distinct elements of $\supp(\mathscr{P})$ are adjacent in $G^{2(N+1)}$. Therefore, the coloring $c$ is injective on $\supp(\mathscr{P})$. This implies that the map
			\[
				[0;1]^{\N \times \N} \to [0;1]^{\supp(\mathscr{P}) \times \N} \colon \theta \mapsto \theta_c \vert (\supp(\mathscr{P}) \times \N)
			\]
			is measure-preserving. Since
			\[
				\mathscr{P} \text{ appears in } \theta_c \,\Longleftrightarrow\, \theta_c \vert (\supp(\mathscr{P}) \times \N) \in \mathbf{App}(\mathscr{P}),
			\]
			we may conclude
			\begin{align*}
				\lambda^{\N \times \N} (\set{\theta \in [0;1]^{\N \times \N} \,:\, \mathscr{P} \text{ appears in }\theta_c}) = \lambda^{\supp(\mathscr{P})\times \N} (\mathbf{App}(\mathscr{P})).
			\end{align*}
			Therefore,
			\begin{align*}
				\lambda^{\N \times \N}(Q_x) \leq& \sum_{S \in \sh_x(\B)}\, \sum_{\mathscr{P} \in \mathbf{Piles}_N(S)} \lambda^{\N \times \N} (\set{\theta \in [0;1]^{\N \times \N} \,:\, \mathscr{P} \text{ appears in }\theta_c}) \\
				=& \sum_{S \in \sh_x(\B)} \,\sum_{\mathscr{P} \in \mathbf{Piles}_N(S)} \lambda^{\supp(\mathscr{P})\times \N} (\mathbf{App}(\mathscr{P})) \,\leq\, \epsilon/2,
			\end{align*}
			by the definition of $D_N$.
		\end{proof}
		
		Using Fubini's theorem and Lemma~\ref{lemma:smallfibers}, we get
		\begin{align*}
			(\mu \times \lambda^{\N \times \N}) (Q) = \int_{X} \lambda^{\N \times \N}(Q_x) \D\mu(x) \leq \mu(D_N) + (1-\mu(D_N)) \cdot\epsilon/2  \leq \epsilon.
		\end{align*}
		Therefore, Fubini's theorem yields some $\theta \colon \N \times \N \to [0;1]$ with $\mu(Q^\theta) \leq \epsilon$. Fix any such~$\theta$ and let $\mathcal{A}=(A_n)_{n=0}^\infty$ be an arbitrary Borel Moser--Tardos process with input $\theta_c$. Let $t_n$ and $f_n$ denote the associated maps.
		
		\begin{lemma}
			$\Def_\B(f_N) \subseteq Q^\theta$.
		\end{lemma}
		\begin{proof}
			If $x \in \Def_\B(f_N)$, then, by definition, there is $S \in \sh_x(\B)$ such that $f_N\vert S \in \B_S$, i.e., $S \in A_N'$. By Proposition~\ref{prop:reduction}, there is $\mathscr{P} \in \mathbf{Piles}_N(S)$ that appears in $\theta_c$. Therefore, $(x, \theta) \in Q$, as desired.
		\end{proof}
		
		Finally, we obtain $\mu(\Def_\B(f_N)) \leq \mu(Q^\theta) \leq \epsilon$, and the proof of Theorem~\ref{theo:approxLLL} is complete.

		\section{The LLL for probability measure-preserving group actions}\label{section:groups}
		
		
		\subsection{Definitions and the statement of the theorem}
		
		\mbox{}
		
		\smallskip
		
		\noindent As discussed in the introduction, we would like to establish a measurable version of the~LLL for Borel instances that, in a certain sense, ``respect'' some additional structure on the space $X$, specifically, an action of a countable group $\Gamma$. To make this idea precise, we introduce L-systems---objects consisting of a standard probability space equipped with a family of functions (``partial isomorphisms'') under which any instance of the~LLL that we might consider must be invariant. We then define the~LLL Game over an L-system, which captures the need for iterated applications of the~LLL.
		
		\subsubsection*{Equivalence relations}
		
		\mbox{}
		
		\smallskip
		
		\noindent We identify an equivalence relation $E$ on a set $X$ with the set of pairs $\set{(x,y) \,:\, x \,E\, y}$. In particular, if $X$ is a standard Borel space, then $E$ is \emph{Borel} if it is a Borel subset of $X^2$. We use~$X/E$ to denote the set of all $E$-classes. A set $X' \subseteq X$ is \emph{$E$-invariant} if it is a union of $E$-classes; i.e., for all $x \in X'$ and $y \in X$ with $x \, E \, y$, we have $y \in X'$. For $S \subseteq X$, we use $[S]_E$ to denote the \emph{$E$-saturation} of $S$, i.e., the smallest $E$-invariant subset of $X$ that contains $S$. For brevity, given $x \in X$, we write $[x]_E$ instead of $[\set{x}]_E$.
		 
		 We say that an equivalence relation $E$ is \emph{countable} if every $E$-class is countable. It follows from the Luzin--Novikov Theorem~\ref{theo:LN} that if $E$ is a countable Borel equivalence relation on a standard Borel space $X$, then the $E$-saturation of every Borel subset of $X$ is Borel.
		 
		 Given an equivalence relation $E$ on $X$, we write (somewhat ambiguously)
		 \begin{align*}
			 \fins{E} &\defeq \set{S \in \fins{X} \,:\, S \text{ is contained in a single $E$-class}}\\
			 \text{and} \qquad \finf{E}{Y} &\defeq \set{w \in \finf{X}{Y} \,:\, \dom(w) \in \fins{E}}.
		 \end{align*}
		 An \emph{instance \ep{of the LLL}} over $E$ is an instance $\B$ over $X$ such that $\dom(\B) \subseteq \fins{E}$.
		 
		 \begin{exmp}[\textbf{Equivalence relations induced by graphs}]
		 	Let $G$ be a graph on a set~$X$. We use~$E_G$ to denote the equivalence relation on $X$ whose classes are the connected components of $G$.
		 \end{exmp}
		 
		 \begin{exmp}[\textbf{Equivalence relations induced by group actions}]
		 	Let $\alpha \colon \Gamma \acts X$ be an action of a group $\Gamma$ on a set $X$. Then $E_\alpha$ denotes the corresponding \emph{orbit equivalence relation}, i.e., the equivalence relation whose classes are the orbits of $\alpha$. Notice that if $S \subseteq \Gamma$ is a generating set, then $E_\alpha = E_{G(\alpha, S)}$.
		 \end{exmp}
		 
		 \subsubsection*{Isomorphism structures}
		 
		 \mbox{}
		 
		 \smallskip
		 
		 \noindent An \emph{isomorphism structure} on an equivalence relation $E$ on a set $X$ is a family $\iso$ of bijections between $E$\=/classes which forms a groupoid\footnote{A \emph{groupoid} is a category in which every morphism has an inverse.} whose set of objects is $X / E$; more precisely, the following conditions must be fulfilled:
		 \begin{itemize}
		 	\item[--] for each $C \in X/E$, the identity map $\id_C \colon C \to C$ belongs to $\iso$;
		 	\item[--] for each $\phi \in \iso$, we have $\phi^{-1}\in\iso$;
		 	\item[--] for all $\phi$, $\psi \in \iso$, if $\im(\phi) = \dom(\psi)$, then $\psi \circ \phi \in \iso$.
		 \end{itemize}
		 The following are the main examples of isomorphism structures we will be considering.
		 
		 \begin{exmp}[\textbf{Isomorphism structures induced by graphs}]
		 	Let $G$ be a graph on a set~$X$. Define the isomorphism structure $\iso_G$ on $E_G$ as follows: A bijection $\phi\colon C_1 \to C_2$ between components $C_1$ and $C_2$ belongs to $\iso_G$ if and only if it is an isomorphism between the graphs $G\vert C_1$ and $G \vert C_2$.
		 \end{exmp}
		 
		 \begin{exmp}[\textbf{Isomorphism structures induced by group actions}]
		 	Let $\alpha \colon \Gamma \acts X$ be an action of a group $\Gamma$ on a set $X$. The isomorphism structure $\iso_\alpha$ on $E_\alpha$ is defined as follows: A bijection $\phi\colon O_1 \to O_2$ between orbits $O_1$ and $O_2$ belongs to $\iso_\alpha$ if and only if it is \emph{$\Gamma$\=/equivariant}, i.e., $\phi(\gamma \cdot x) = \gamma \cdot \phi(x)$ for all $x \in O_1$ and $\gamma \in \Gamma$. Notice that if $S \subseteq \Gamma$ is a generating set, then $\iso_\alpha \subseteq \iso_{G(\alpha,S)}$.
		 \end{exmp}
		 
		 Let $E$ be a Borel equivalence relation on a standard probability space $(X,\mu)$ and let $\iso$ be an isomorphism structure on $E$. We say that an instance $\B$ over $E$ is \emph{$\iso$-invariant} on a set $X' \subseteq X$ if for all $\phi \in \iso$ with $\dom(\phi) \cup \im(\phi) \subseteq X'$ and for all $B \in \B$ with $\dom(B) \subseteq \im(\phi)$, we have
		 \[
		 \set{w \circ \phi \,:\, w \in B} \in \B.
		 \]
		 An instance $\B$ is \emph{$\mu$-almost everywhere $\iso$-invariant} if it is $\iso$-invariant on an $E$-invariant $\mu$-conull Borel subset $X' \subseteq X$.
		 
		 \subsubsection*{L-Systems and instances of the LLL over them}
		 
		 \mbox{}
		 
		 \smallskip
		 
		 \noindent An \emph{L-system}\footnote{``L'' is for ``Lov\'asz.''} is a tuple $\LS = (X_\LS, E_\LS, \iso_\LS, \mu_\LS)$, where
		 \begin{itemize}
		 	\item[--] $(X_\LS, \mu_\LS)$ is a standard probability space;
		 	\item[--] $E_\LS$ is a countable Borel equivalence relation on $X_\LS$;
		 	\item[--] $\iso_\LS$ is an isomorphism structure on $E_\LS$.
		 \end{itemize}
		 An \emph{instance \ep{of the~LLL}} over an L-system $\LS$ is a $\mu_\LS$-almost everywhere $\iso_\LS$-invariant Borel instance over $E_\LS$. A Borel map $f \colon X_\LS \to [0;1]$ is a \emph{measurable solution} to an instance $\B$ over $\LS$ if $\Def_\B(f)$ is contained in an $E_\LS$-invariant $\mu_\LS$-null Borel subset of $X_\LS$.
		 
		 For a Borel action $\alpha \colon \Gamma \acts (X, \mu)$, let $\LS(\alpha,\mu)$ denote the L-system $(X, E_\alpha, \iso_\alpha, \mu)$ \emph{induced} by $\alpha$. 
		 An instance over $\LS(\alpha,\mu)$ is simply a Borel instance over $X$ such that the domain of each bad event $B \in \B$ is contained within a single $\alpha$-orbit and $\B$ is ($\mu$-almost everywhere) invariant under the $\Gamma$-equivariant bijections between the orbits of $\alpha$.
		 
		 \subsubsection*{Amplifications and expansions}
		 
		 \mbox{}
		 
		 \smallskip
		 
		 \noindent Before we can state the main result of this section, we need a few more definitions describing how to build new L-systems from old ones.
		 
		 Let $E$ be an equivalence relation on a set $X$. Define (somewhat ambiguously)
		 \[
			 \HF(E) \defeq \set{h \in \HF(X) \,:\, \U(h) \in \fins{E}}.
		 \]
		 The \emph{amplification} of $E$ is the equivalence relation $\tilde{E}$ on $\HF(E)$ defined by
		 \[
			 h_1 \, \tilde{E} \, h_2 \,\vcentcolon\Longleftrightarrow\, [\U(h_1)]_E = [\U(h_2)]_E.
		 \]
		 In other words, $\tilde{E}$ is the equivalence relation on $\HF(E)$ whose classes are the sets $\HF(C)$ with $C \in X/E$. For a bijection $\phi \colon C_1 \to C_2$ between $E$-classes, we may extend it to a bijection $\tilde{\phi} \colon \HF(C_1) \to \HF(C_2)$ between the corresponding $\tilde{E}$-classes. The \emph{amplification} of an isomorphism structure $\iso$ on $E$ is the isomorphism structure $\tilde{\iso}$ on $\tilde{E}$ given by
		 \[
			 \tilde{\iso} \defeq \set{\tilde{\phi} \,:\, \phi \in \iso}.
		 \]
		 Given an L-system $\LS = (X, E, \iso, \mu)$, its \emph{amplification} is the L-system
		 \[
			 \HF(\LS) \defeq (\HF(E), \tilde{E}, \tilde{\iso}, \mu). 
		 \]
		 Notice that the measure in $\HF(\LS)$ is the same as in $\LS$ and is concentrated on $X \subseteq \HF(X)$.
		 
		 Another way of obtaining new L-systems is via \emph{expansions}. Let $\iso$ be an isomorphism structure on an equivalence relation $E$ on a set $X$. Given a partial map $f \colon X \rightharpoonup Y$, the \emph{expansion} of $\iso$ by $f$ is the subset $\iso[f] \subseteq \iso$ defined as follows:
		 \[
			 \iso[f] \defeq \set{\phi \in \iso\,:\, f(x) = f(\phi(x)) \text{ for all } x \in \dom(\phi)}.
		 \]
		 Here the equality ``$f(x) = f(\phi(x))$'' should be interpreted as a shorthand for:
		 \[
			 \text{``Either $\set{x, \phi(x)} \subseteq \dom(f)$ and $f(x) = f(\phi(x))$, or else, $\set{x, \phi(x)}\cap\dom(f) = \0$.''}
		 \]
		 For an L-system $\LS = (X, E, \iso, \mu)$ and a Borel map $f \colon X \rightharpoonup Y$, the \emph{expansion} of $\LS$ by $f$ is the L-system
		 \[
		 \LS[f] \defeq (X, E, \iso[f], \mu). 
		 \]
		 The term ``expansion'' conveys the following intuition: If $\iso$ is thought of as a family of isomorphisms between certain substructures of $X$, then expanding $\iso$ by $f$ corresponds to adding $f$ to~$X$ as a new ``predicate'' whose values must be preserved by isomorphisms.
		 
		 \subsubsection*{The LLL game}
		 
		 \mbox{}
		 
		 \smallskip
		 
		 \noindent As we mentioned in the introduction, many combinatorial arguments contain iterated applications of the LLL, where the output of a previous iteration can be used to create an instance for the next one. To accommodate such arguments, we introduce the following definition.
		 
		 \begin{defn}[\textbf{LLL Game}]\label{defn:LLLGame}
		 	The \emph{LLL Game} over an L-system $\LS$ is played as follows. Set $\LS_0 \defeq \LS$. On Step~$n\in\N$, Player~I chooses a correct instance $\B_n$ over~$\LS_n$. Player~II must respond by playing a measurable solution $f_n$ to $\B_n$ and setting $\LS_{n+1} \defeq \LS_n[f_n]$. Player~I wins if Player~II does not have an available move on some finite stage of the game; Player~II wins if the game continues indefinitely. A~run of the~LLL Game looks like this:
		 	\smallskip
		 	\begin{center}
		 		\begin{tabular}{ c || c | c | c | c | c | c | c | c }			
		 			Player~I  & $\B_0$ &       & $\B_1$ &       &  \ldots & $\B_n$ & & \ldots\\
		 			\hline
		 			Player~II &        & $f_0$ &        & $f_1$ &  \ldots &        & $f_n$ & \ldots\\
		 		\end{tabular}
		 	\end{center}
		 \end{defn}
		 
		 One can think of the LLL Game as a struggle between a malevolent combinatorial proof (Player~I) and a descriptive set theorist (Player~II), who wants to adapt this proof to the measurable setting. The proof consists of a series of steps, each of which is an application of the~LLL. The goal of Player~II is to perform these steps measurably; however, she might not know what the steps are in advance, and each time she solves an instance of the LLL, her solution may be ``used against her'' in creating new instances.
		 
		 With Definition~\ref{defn:LLLGame} at hand, we are finally ready to state the main result of this section:
		 
		 \begin{theo}[{\textbf{Measurable LLL for group actions}}]\label{theo:Thm1}
		 	Let $\alpha \colon \Gamma \acts (X, \mu)$ be a measure-preserving action of a countable group $\Gamma$ on a standard probability space $(X,\mu)$. If $\alpha$ factors to the shift action $\Gamma \acts ([0;1]^\Gamma, \lambda^\Gamma)$, then Player~II has a winning strategy in the LLL Game over $\HF(\LS(\alpha, \mu))$.
		 \end{theo}
		 
		 A very specific case of Theorem~\ref{theo:Thm1} is given by the following immediate corollary:
		 
		 \begin{corl}\label{corl:no_iterations}
			 Let $\alpha \colon \Gamma \acts (X, \mu)$ be a measure-preserving action of a countable group $\Gamma$ on a standard probability space $(X, \mu)$. Suppose that $\alpha$ factors to the shift action $\Gamma \acts ([0;1]^\Gamma, \lambda^\Gamma)$ and let $\B$ be a correct instance over $\LS(\alpha, \mu)$. Then there exists a Borel function $f \colon X \to [0;1]$ with $\mu(\Def_\B(f)) =0$.
		\end{corl}
		
		
		\subsection{Outline of the proof}\label{prg:Thm1outline}
		
		\mbox{}
		
		\smallskip
		
		\noindent Let $\mathscr{G}$ denote the class of all L-systems of the form $\LS(\alpha, \mu)$, where $\alpha \colon \Gamma \acts (X,\mu)$ is a measure\-/preserving action of a countable group $\Gamma$ on a standard probability space $(X,\mu)$ that factors to the $[0;1]$-shift action of $\Gamma$. Let $\mathscr{L}$ be the class of all L-systems such that Player~II has a winning strategy in the~LLL Game over $\HF(\LS)$. Our goal is to show $\mathscr{G} \subseteq \mathscr{L}$. To that end, we will introduce an intermediate class~$\mathscr{C}$ such that $\mathscr{G} \subseteq \mathscr{C} \subseteq \mathscr{L}$.
		
		Our strategy for showing that $\mathscr{C} \subseteq \mathscr{L}$ will be to ensure that $\mathscr{C}$ has the following two properties:
		\begin{enumerate}[label={(A$\arabic*$)}]
			\item\label{item:A1} if $\LS \in \mathscr{C}$, then $\HF(\LS) \in \mathscr{C}$;
			\item\label{item:A2} if $\LS \in \mathscr{C}$ and $\B$ is a correct instance over $\LS$, then there exists a measurable solution $f$ to $\B$ such that  $\LS[f] \in \mathscr{C}$.
		\end{enumerate}
		The above conditions imply that $\mathscr{C} \subseteq \mathscr{L}$. Indeed, due to Property~\ref{item:A1}, it is enough to show that for every $\LS \in \mathscr{C}$, Player~II has a winning strategy in the~LLL Game over $\LS$. The existence of such strategy is guaranteed by Property~\ref{item:A2}, since, provided that $\LS_n \in \mathscr{C}$, Player~II can always find a measurable solution $f_n$ to $\B_n$ such that $\LS_{n+1} = \LS_n[f_n] \in \mathscr{C}$.
		
		It is easy to see that Property~\ref{item:A1} fails for $\mathscr{G}$. For instance, if $\LS = (X, E, \iso, \mu) \in \mathscr{G}$, then the measure $\mu$ is $E$\=/invariant, while it is not even $\tilde{E}$-quasi-invariant. To overcome this complication, we will introduce \emph{countable Borel groupoids}---algebraic structures more general than countable groups---and their actions on standard Borel spaces. Every Borel action of a countable Borel groupoid on a standard probability space induces an L-system. We will also define shift actions of countable Borel groupoids, generalizing shift actions of countable groups. Our choice for $\mathscr{C}$ will be the class of all L-systems that admit factor maps to L-systems induced by shift actions of countable Borel groupoids (we define what a factor map between two general L-systems is in \prg\ref{prg:factors}).
				
		\subsection{Factors of L-systems}\label{prg:factors}
		
		\mbox{}
		
		\smallskip

		\noindent In this section we introduce the notion of a factor map between two L-systems. It will allow us to transfer instances of the LLL from a given L-system to a simpler or better-behaved one.
				
				\begin{defn}[\textbf{Factors}]\label{defn:factors}
					Let $\LS_1 = (X_1, E_1, \iso_1, \mu_1)$ and $\LS_2 = (X_2, E_2, \iso_2, \mu_2)$ be L-systems. A Borel partial map $\pi \colon X_1 \rightharpoonup X_2$, defined on an $E_1$-invariant $\mu_1$-conull Borel subset of $X_1$, is called a \emph{factor map} (notation: $\pi \colon \LS_1 \to \LS_2$) if the following requirements are fulfilled:
					\begin{enumerate}[label={(\roman*)}]
						\item $\pi_\ast(\mu_1) = \mu_2$;
						\item the map $\pi$ is \emph{class-bijective}, i.e., for each $E_1$-class $C \subseteq \dom(\pi)$, its image $\pi(C)$ is an $E_2$-class and the restriction $\pi\vert C \colon C \to \pi(C)$ is a bijection;
						\item\label{item:diagram} for all $E_1$-classes $C_1$, $C_2 \subseteq \dom(\pi)$, whenever $\phi_2 \in \iso_2$ is a bijection between $\pi(C_1)$ and $\pi(C_2)$, there is a bijection $\phi_1 \in \iso_1$ between $C_1$ and $C_2$ that makes the following diagram commute:
						\[
							\begin{tikzcd}
								C_1 \arrow{d}{\pi} \arrow[dashrightarrow]{r}{\phi_1} & C_2 \arrow{d}{\pi} \\
							\pi(C_1) \arrow{r}{\phi_2}& \pi(C_2).
						\end{tikzcd}
						\]
					\end{enumerate}
				\end{defn}
				
				\begin{prop}\label{prop:EnhancedFactor}
					Let $\LS_1$ and $\LS_2$ be L-systems with a factor map $\pi \colon \LS_1 \to \LS_2$ between them. Then there exists a factor map from $\HF(\LS_1)$ to $\HF(\LS_2)$.
				\end{prop}
				\begin{proof}
					Let $\tilde{\pi} \colon \HF(\dom(\pi)) \to \HF(X_{\LS_2})$ be the amplification of~$\pi$. Then the restriction of $\tilde{\pi}$ to the set $\HF(E_{\LS_1}) \cap \dom(\tilde{\pi})$ is a factor map from $\HF(\LS_1)$ to $\HF(\LS_2)$.
				\end{proof}
				
				\begin{lemma}\label{lemma:pushforward}
					Let $\LS_1$ and $\LS_2$ be L-systems with a factor map $\pi \colon \LS_1\to \LS_2$ between them. Then for every correct instance~$\B$ over $\LS_1$, there exists a correct instance $\pi(\B)$ over $\LS_2$ such that whenever $f$ is a measurable solution to~$\pi(\B)$, the composition $f \circ \pi$, possibly restricted to a smaller invariant conull Borel subset, is a measurable solution to~$\B$.
				\end{lemma}
				\begin{proof}
					For $i \in \set{1,2}$, let $\LS_i \eqqcolon (X_i, E_i, \iso_i, \mu_i)$. Suppose that $\B$ is a correct instance over $\LS_1$. Restricting $\pi$ to a smaller $E_1$-invariant $\mu_1$-conull Borel subset of $X_1$ if necessary, we arrange that $\B$ is $\iso_1$-invariant on $\dom(\pi)$ and $\im(\pi)$ is a Borel subset of $X_2$. Then we replace $X_1$ and $X_2$ by their invariant conull Borel subsets $\dom(\pi)$ and $\im(\pi)$ respectively. Thus, we now assume that $\pi \colon X_1 \to X_2$ is defined everywhere and is surjective.
					
					Consider any $B \in \B$. Since $\dom(B)$ is contained within a single $E_1$-class, the restriction
					\[
						\pi\vert \dom(B) \colon \dom(B) \to \pi(\dom(B))
					\]
					is bijective; in particular, the inverse
					\[
						(\pi\vert \dom(B))^{-1} \colon \pi(\dom(B)) \to \dom(B)
					\]
					is well-defined. Let
					\[
						\pi(B) \defeq \set{w \circ (\pi\vert \dom(B))^{-1} \,:\, w \in B}.
					\]
					Then $\pi(B)$ is a bad event over $X_2$ with domain $\pi(\dom(B))$. Define
					\[
						\pi(\B) \defeq \set{\pi(B) \,:\, B \in \B}.
					\]
					It is routine to check that $\pi(\B)$ is as desired. The only non-trivial step is to show that $\pi(\B)$ is Borel. To that end, we observe that the set $\bigcup \pi(\B)$ is both analytic and co-analytic, as for $w \in \finf{E_2}{[0;1]}$,
					\begin{align*}
						w \in \textstyle\bigcup\pi(\B) &\,\Longleftrightarrow\, \exists S \in \fins{E_1}\, \left(\pi(S) = \dom(w)\text{ and } w \circ (\pi\vert S) \in \textstyle\bigcup\B\right) \\
					&\,\Longleftrightarrow\, \forall S \in \fins{E_1}\, \left(\pi(S) = \dom(w) \Longrightarrow w \circ (\pi\vert S) \in \textstyle\bigcup\B\right).
					\end{align*}
					The first of these equivalences follows directly from the definition of $\pi(\B)$. To prove the second equivalence, take any $T \in \fins{E_2}$ and suppose that $S$, $S' \in \fins{E_1}$ satisfy $\pi(S) = \pi(S') = T$. Setting $C \defeq [S]_{E_1}$, $C' \defeq [S']_{E_1}$, and $D \defeq [T]_{E_2}$, we see that $\pi(C) = \pi(C') = D$. By part \ref{item:diagram} of Definition~\ref{defn:factors}, there is $\phi \in \iso_1$ that makes the following diagram commute:
					\[
					\begin{tikzcd}
					C \arrow{d}{\pi} \arrow[dashrightarrow]{r}{\phi} & C' \arrow{d}{\pi} \\
					D \arrow{r}{\id_D}& D.
					\end{tikzcd}
					\]
					As the instance $\B$ is $\iso_1$-invariant, we conclude that for all $w \in [0;1]^T$,
					\[
						w \circ (\pi \vert S) \in \bigcup \B \Longleftrightarrow w \circ (\pi \vert S') \in \bigcup \B, 
					\]
					and we are done.
				\end{proof}
				
				For a class $\mathscr{C}$ of L-systems, define the class $\mathscr{C}^\ast$ by
				\[
					\LS \in \mathscr{C}^\ast \,\vcentcolon\Longleftrightarrow\, \LS \text{ admits a factor map to } \LS' \text{ for some } \LS' \in \mathscr{C},
				\]
				so $\mathscr{C}^\ast \supseteq \mathscr{C}$ and $(\mathscr{C}^\ast)^\ast = \mathscr{C}^\ast$. Let $\mathscr{C}$ be a class of L-systems satisfying the following two conditions:
				\begin{enumerate}[label={(B$\arabic*$)}]
					\item\label{item:B1} if $\LS \in \mathscr{C}$, then $\HF(\LS) \in \mathscr{C}^\ast$;
					\item\label{item:B2} if $\LS \in \mathscr{C}$ and $\B$ is a correct instance over $\LS$, then there exists a measurable solution $f$ to $\B$ such that $\LS[f]\in\mathscr{C}^\ast$.
				\end{enumerate}
				Note that if $\pi \colon \LS_1 \to \LS_2$ is a factor map between L-systems $\LS_1 = (X_1, E_1, \iso_1, \mu_1)$ and $\LS_2 = (X_2, E_2, \iso_2, \mu_2)$ and $f \colon X_2 \rightharpoonup Y$ is a Borel function, then $\pi$ is also a factor map from $\LS_1[\pi \circ f]$ to $\LS_2[f]$. Therefore, due to Proposition~\ref{prop:EnhancedFactor} and Lemma~\ref{lemma:pushforward}, if $\mathscr{C}$ satisfies conditions \ref{item:B1} and \ref{item:B2}, then $\mathscr{C}^\ast$ has Properties~\ref{item:A1} and~\ref{item:A2} from~\prg\ref{prg:Thm1outline}, and hence $\mathscr{C} \subseteq \mathscr{C}^\ast \subseteq \mathscr{L}$.
				
				\subsection{Countable Borel groupoids and their actions}\label{prg:cBg}
				
				\mbox{}
				
				\smallskip
				
							
				\begin{defn}[\textbf{Countable Borel groupoids}]
					A \emph{countable Borel groupoid} $(R, \Gamma)$ is a structure consisting of a standard Borel space $R$ together with a countable set $\Gamma$ and Borel maps
					\[
						\begin{array}{clll}
						&\mathbf{a} \colon \Gamma \times R \to R &\colon (\gamma, r) \mapsto \gamma \cdot r \;\;\; &\text{(\emph{action})};\\
						&\mathbf{c} \colon \Gamma^2 \times R \to \Gamma &\colon (\gamma, \delta, r) \mapsto \gamma \circ_r \delta \;\;\; &\text{(\emph{composition})};\\
						&\mathbf{id} \colon R \to \Gamma &\colon r \mapsto \mathbf{1}_r \;\;\; &\text{(\emph{identity})};\\
						\text{and} \;\;\; &\mathbf{inv}\colon \Gamma \times R \to \Gamma &\colon (\gamma, r) \mapsto \gamma^{-1}_r \;\;\; &\text{(\emph{inverse})},
						\end{array}
					\]
					satisfying the following axioms:
					\[
						\begin{array}{lcc}
						 \text{-- } \text{\emph{consistency}: for all $\gamma$, $\delta \in \Gamma$ and $r \in R$,} &\qquad& \gamma \cdot (\delta \cdot r) = (\gamma \circ_r \delta) \cdot r;\\
						 \text{-- } \text{\emph{associativity}: for all $\gamma$, $\delta$, $\epsilon \in \Gamma$ and $r \in R$,} & & \gamma \circ_r (\delta \circ_r \epsilon) = (\gamma \circ_{\epsilon \cdot r} \delta) \circ_r \epsilon;\\
						 \text{-- } \text{\emph{identity}: for all $r \in R$ and $\gamma \in \Gamma$,} & & \mathbf{1}_r \cdot r = r$ \qquad
						 \text{and} \qquad $\mathbf{1}_{\gamma\cdot r} \circ_r \gamma = \gamma \circ_r \mathbf{1}_r = \gamma;\\
						 \text{-- } \text{\emph{inverse}: for all $r \in R$ and $\gamma \in \Gamma$,} & & \gamma^{-1}_r \circ_r \gamma = \mathbf{1}_r$ \qquad \text{and} \qquad  $\gamma \circ_{\gamma \cdot r} \gamma^{-1}_r = \mathbf{1}_{\gamma \cdot r}.
						\end{array}
					\]
				\end{defn}
				
				\begin{figure}[h]
				\[
					\begin{array}{ccc}
					\begin{tikzcd}
					\epsilon\cdot r  \arrow{rr}{\delta} & &\delta \cdot (\epsilon \cdot r) \arrow{d}{\gamma} \\
					r \arrow{u}{\epsilon} \arrow[dashed]{rr}[swap]{\gamma \circ_r (\delta \circ_r \epsilon)} \arrow{rru}[description]{\delta \circ_r \epsilon} &   & \gamma \cdot (\delta\cdot (\epsilon \cdot r))
					\end{tikzcd} &\qquad&
					\begin{tikzcd}
					\epsilon\cdot r  \arrow{rr}{\delta} \arrow{rrd}[description]{\gamma \circ_{\epsilon \cdot r} \delta} & &\delta \cdot (\epsilon \cdot r) \arrow{d}{\gamma} \\
					r \arrow{u}{\epsilon} \arrow[dashed]{rr}[swap]{(\gamma \circ_{\epsilon \cdot r} \delta) \circ_r \epsilon} &   & \gamma \cdot (\delta\cdot (\epsilon \cdot r))
					\end{tikzcd}
					\end{array}
				\]
				\caption{Associativity: the dashed arrows must coincide.}
				\end{figure}
				
	Any countable group $\Gamma$ can be canonically viewed as a countable Borel groupoid in the following way. Let $R \defeq \set{r}$ be a single point. For each $\gamma \in \Gamma$, set $\gamma \cdot r \defeq r$. Now we just transfer compositions, the identity, and inverses directly from the group (we use $\mathbf{1}_\Gamma$ to denote the identity element of $\Gamma$):
	\begin{equation}\label{eq:}
		\gamma \circ_r \delta \defeq  \gamma\delta;\qquad \mathbf{1}_r \defeq \mathbf{1}_\Gamma; \qquad \text{and} \qquad \gamma^{-1}_r \defeq \gamma^{-1}.
	\end{equation}			
	A more general class of examples is given by Borel actions of countable groups. Let $\alpha \colon \Gamma \acts R$ be a Borel action of a countable group $\Gamma$ on a standard Borel space~$R$. Then $(R, \Gamma)$ can be endowed with the structure of a countable Borel groupoid as follows: Set $\gamma \cdot r \defeq \gamma \cdot_\alpha r$ for all $\gamma \in \Gamma$, $r \in R$, and define compositions, identities, and inverses via~\eqref{eq:} (i.e., in a way that does not depend on~$r \in R$).
				
	An interesting example of a countable Borel groupoid is produced by ``bundling'' all countable groups into a single algebraic structure. Let $\mathcal{G}$ be the standard Borel space of all countably infinite groups with ground set $\N$ (which can be viewed as a Borel subset of the Cantor space $2^{\N^3}$). Define a countable Borel groupoid $(\mathcal{G}, \N)$ as follows: For each $n \in \N$ and $\Gamma \in \mathcal{G}$, let $n \cdot \Gamma \defeq \Gamma$. Now set
	\[
		\begin{array}{ll}
			n \circ_\Gamma m & \text{to be the product of $n$ and $m$ as elements of $\Gamma$};\\
			\mathbf{1}_\Gamma & \text{to be the identity element of $\Gamma$};\\
			n^{-1}_\Gamma & \text{to be the inverse of $n$ in $\Gamma$}.
		\end{array}
	\]
			
	The following proposition is a useful and easy-to-check condition that guarantees that a certain structure is a countable Borel groupoid.	
	\begin{prop}\label{prop:UniqueCBG}
		Let $R$ be a standard Borel space and let $E$ be a countable Borel equivalence relation on~$R$. Let $\Gamma$ be a countable set and let $\mathbf{a} \colon \Gamma \times R \to R \colon (\gamma, r) \mapsto \gamma \cdot r$ be a Borel function. Suppose that for each $r \in R$, the map $\gamma \mapsto \gamma \cdot r$ is a bijection between $\Gamma$ and $[r]_E$. Then there is a unique countable Borel groupoid structure on $(R, \Gamma)$ with $\mathbf{a}$ as its action map.
	\end{prop}
	\begin{proof}
	For $r_1$, $r_2 \in R$ with $r_1 \,E\, r_2$, let $\epsilon(r_1, r_2)$ denote the unique element $\epsilon \in \Gamma$ such that $r_2 = \epsilon \cdot r_1$. The only consistent way to turn $(R, \Gamma)$ into a countable Borel groupoid is as follows:
	\[
		\gamma \circ_r \delta \defeq \epsilon(r, \gamma \cdot (\delta \cdot r));\qquad \mathbf{1}_r \defeq \epsilon(r, r);\qquad\text{and}\qquad\gamma^{-1}_r \defeq \epsilon(\gamma \cdot r, r).
	\]
	A straightforward verification shows that the above definition satisfies all the axioms.
	\end{proof}
	
	Now we proceed to the definition of Borel actions of countable Borel groupoids.	
	\begin{defn}[\textbf{Actions}]\label{defn:actions}
		Let $(R, \Gamma)$ be a countable Borel groupoid. A (Borel) \emph{action} $(\rho, \alpha)$ of $(R, \Gamma)$ on a standard Borel space $X$ is a pair of Borel maps $\rho \colon X \to R$ and $\alpha \colon \Gamma \times X \to X \colon (\gamma, x) \mapsto \gamma \cdot_\alpha x$ satisfying the following conditions:
		\[
		\begin{array}{lcc}
		\text{-- } \text{\emph{equivariance}: for all $x \in X$ and $\gamma \in \Gamma$,} &\qquad\qquad& \rho(\gamma \cdot_\alpha x) = \gamma \cdot \rho(x);\\
		\text{-- } \text{\emph{identity}: for all $x \in X$,} & & \mathbf{1}_{\rho(x)} \cdot_\alpha x = x;\\
		\text{-- } \text{\emph{compatibility}: for all $x \in X$ and $\gamma$, $\delta \in \Gamma$} & & \gamma \cdot_\alpha (\delta \cdot_\alpha x) = (\gamma \circ_{\rho(x)} \delta) \cdot_\alpha x.
		\end{array}
		\]
		As with group actions, we will usually simply write $\gamma \cdot x$ for $\gamma \cdot_\alpha x$.
	\end{defn}
				
	Clearly, a (left) group action $\Gamma \acts X$ is also a countable Borel groupoid action if $\Gamma$ is understood as a countable Borel groupoid. Now suppose that a countable group $\Gamma$ acts (in a Borel way) on a standard Borel space $R$. Viewing $(R, \Gamma)$ as a countable Borel groupoid, consider an action $(\rho, \alpha)$ of $(R, \Gamma)$ on some space $X$. By the identity and the compatibility conditions in Definition~\ref{defn:actions}, $\alpha$ is an action of $\Gamma$ on~$X$, while the equivariance condition stipulates that the map $\rho \colon X \to R$ must be $\Gamma$-equivariant. Thus, a Borel action of $(R, \Gamma)$ is the same as a $\Gamma$-space equipped with a Borel $\Gamma$-equivariant map to~$R$.
	If $(\mathcal{G}, \N)$ is the countable Borel groupoid of all countable groups, then an action of $(\mathcal{G}, \N)$ on $X$ consists of a Borel map $\rho \colon X \to \mathcal{G}$ and a $\Gamma$-action on $\rho^{-1}(\Gamma)$ for each $\Gamma \in \mathcal{G}$.
				
	\begin{defn}[\textbf{Shift actions}]\label{defn:shift}
		Let $(R, \Gamma)$ be a countable Borel groupoid and let $Y$ be a standard Borel space. The \emph{$Y$-shift action} $(\rho,\alpha) \colon (R, \Gamma) \acts R \times Y^\Gamma$ is defined as follows: For each $(r, \theta) \in R \times Y^{\Gamma}$, set $\rho(r, \theta) \defeq r$, and for $\gamma \in \Gamma$, define
		\[
			\gamma \cdot_\alpha (r, \theta) \defeq (\gamma \cdot r, \theta'), \qquad \text{where} \qquad \theta'(\delta) \defeq \theta(\delta \circ_r \gamma) \text{ for all }\delta \in \Gamma.
		\]
	\end{defn}
				
	It is routine to check that the $Y$-shift action as defined above is indeed an action of $(R, \Gamma)$. We give the proof here to help the reader get familiar with the definitions. The equivariance condition is satisfied trivially. For the identity condition, observe that if $x = (r, \theta) \in R \times Y^\Gamma$, then
	\[
		\mathbf{1}_{\rho(x)} \cdot x = \mathbf{1}_r \cdot (r, \theta) = (\mathbf{1}_r \cdot r, \theta') = (r, \theta'),
	\]
	where for each $\delta \in \Gamma$,
	\[
		\theta'(\delta) = \theta(\delta \circ_r \mathbf{1}_r) = \theta(\delta),
	\]
	so $\theta' = \theta$, as desired. Finally, for the compatibility condition, we have
	\begin{align*}
		\gamma \cdot (\delta \cdot x) = \gamma \cdot (\delta \cdot (r, \theta)) = \gamma \cdot (\delta \cdot r, \theta') = (\gamma \cdot (\delta \cdot r), \theta'') = ((\gamma \circ_r \delta)\cdot r, \theta''),
	\end{align*}
	where for each $\epsilon \in \Gamma$,
	\[
		\theta''(\epsilon) = \theta'(\epsilon \circ_{\delta \cdot r} \gamma) = \theta((\epsilon \circ_{\delta \cdot r} \gamma) \circ_r \delta) = \theta(\epsilon\circ_r(\gamma \circ_r \delta)),
	\]
	so $\gamma \cdot (\delta \cdot x) = (\gamma \circ_r \delta) \cdot x$, as desired.
	
	Note that for a countable group $\Gamma$, Definition~\ref{defn:shift} is equivalent to the usual definition of the $Y$\=/shift action of $\Gamma$.
				
	By analogy with group actions, we can define L-systems corresponding to actions of countable Borel groupoids. Namely, let $(\rho, \alpha) \colon (R, \Gamma) \acts X$ be a Borel action of a countable Borel groupoid $(R, \Gamma)$ on a standard Borel space $X$. Let $E_{\alpha}$ be the corresponding \emph{orbit equivalence relation} on $X$, defined by
	\[
		x \,E_\alpha\, y \,\vcentcolon\Longleftrightarrow\, \gamma \cdot x = y \text{ for some } \gamma \in \Gamma.
	\]
	This is clearly a countable Borel equivalence relation. Note that $E_\alpha$ does not depend on $\rho$. Let $\iso_{(\rho, \alpha)}$ denote the isomorphism structure on $E_{\alpha}$ such that a bijection $\phi \colon C_1 \to C_2$ between $E_\alpha$-classes $C_1$, $C_2$ belongs to $\iso_{(\rho, \alpha)}$ if and only if $\phi$ is \emph{$(R, \Gamma)$\=/equivariant}, i.e., for all $x \in C_1$ and $\gamma \in \Gamma$, \[
		\rho(\phi(x)) = \rho(x) \qquad \text{and} \qquad \gamma \cdot \phi(x) = \phi(\gamma \cdot x).
	\]
	For $\mu \in \P(X)$, let $\LS(\rho, \alpha, \mu)$ denote the L-system $(X, E_{\alpha}, \iso_{(\rho, \alpha)}, \mu)$. In the case when $|R| = 1$, i.e., $(R, \Gamma)$ is a group, this definition coincides with the one given previously for group actions.
				
	We will be mostly interested in the properties of L-systems induced by shift actions of countable Borel groupoids. More precisely:
				
	\begin{defn}[\textbf{Shift L-systems}]
		A \emph{shift L-system} is any L-system of the form $\LS(\rho, \alpha, \mu \times \nu^\Gamma)$, where $(\rho, \alpha) \colon (R, \Gamma) \acts R \times Y^\Gamma$ is the $Y$-shift action of a countable Borel groupoid $(R, \Gamma)$ for some standard Borel space $Y$, $\mu \in \P(R)$, and $\nu \in \P(Y)$ is atomless. 
	\end{defn}
				
	Thanks to the measure isomorphism theorem, it is enough to consider shift L-systems induced by the $[0;1]$-shift action of $(R, \Gamma)$ with $\nu = \lambda$. However, sometimes it will be more convenient to use other choices for $Y$ and $\nu$; in particular, we will often assume that $Y = [0;1]^S$ and $\nu = \lambda^S$ for some countable set $S$.
				
	\subsection{Factors of L-systems induced by actions of countable Borel groupoids}
	
	\mbox{}
	
	\smallskip
				
	\noindent Let $(R, \Gamma)$ be a countable Borel groupoid and let $(\rho, \alpha) \colon (R, \Gamma) \acts X$ be a Borel action of $(R, \Gamma)$ on a standard Borel space $X$. The action $(\rho, \alpha)$ is \emph{free} if for all $x \in X$ and $\gamma \in \Gamma$,
	\[
		\gamma \cdot x = x\,\Longleftrightarrow\, \gamma = \mathbf{1}_{\rho(x)}.
	\]
	The \emph{free part} of $(\rho, \alpha)$ (notation: $\Free(\rho, \alpha)$ or $\Free(X)$ if the action is clear from the context) is the largest $E_{\alpha}$-invariant subset of $X$ on which the action is free. The free part of an action is always an invariant Borel set. For $\mu \in \P(X)$, an action is \emph{free $\mu$-almost everywhere} if its free part is $\mu$-conull. By definition, if $x\in \Free(X)$, then the map $\gamma \mapsto \gamma \cdot x$ is a bijection between $\Gamma$ and the orbit of $x$.
				
				\begin{prop}
					Let $(R, \Gamma)$ be a countable Borel groupoid and let $\mu \in \P(R)$. Let $Y$ be a standard Borel space and let $\nu \in \P(Y)$ be atomless. Then the $Y$-shift action of $(R, \Gamma)$ is free $(\mu \times \nu^\Gamma)$-almost everywhere.
				\end{prop}
				\begin{proof}
					It is enough to argue that $\Free(R \times Y^\Gamma) \supseteq R \times F$, where
					\[
						F \defeq \set{\theta \in Y^\Gamma\,:\, \theta \colon \Gamma \to Y \text{ is injective}},
					\]
					since the set $F$ is $\nu^\Gamma$-conull. Indeed, suppose that $(r, \theta) \in R \times F$ and let $\gamma \in \Gamma$ be such that $\gamma \cdot (r, \theta) = (r, \theta)$. By definition, this means that $\gamma \cdot r = r$ and
					$
						\theta(\mathbf{1}_r) = \theta(\mathbf{1}_r \circ_r \gamma) = \theta(\mathbf{1}_{\gamma \cdot r} \circ_r \gamma) = \theta(\gamma)
					$. 
					Since $\theta$ is injective, this yields $\gamma = \mathbf{1}_r$, as desired.
				\end{proof}
				
				The next lemma will be useful in verifying that certain maps between L-systems induced by actions of countable Borel groupoids are factor maps.
				
				\begin{lemma}\label{lemma:equivariant}
					Let $(R, \Gamma)$ be a countable Borel groupoid and let
					\[
						(\rho_1, \alpha_1) \colon (R, \Gamma) \acts X_1 \qquad \text{and} \qquad (\rho_2, \alpha_2) \colon (R, \Gamma) \acts X_2
					\]
					be two Borel actions of $(R, \Gamma)$. Let $\mu_1 \in \P(X_1)$ and $\mu_2 \in \P(X_2)$. Suppose that $(\rho_2, \alpha_2)$ is $\mu_2$-almost everywhere free. Let $\pi \colon X_1 \rightharpoonup X_2$ be a measure-preserving $(R, \Gamma)$-equivariant Borel map defined on an $E_{\alpha_1}$-invariant $\mu_1$-conull Borel subset of $X_1$. Then $\pi$, possibly restricted to a smaller invariant conull Borel subset of $X_1$, is a factor map from $\LS(\rho_1, \alpha_1, \mu_1)$ to $\LS(\rho_2, \alpha_2, \mu_2)$.
				\end{lemma}
				\begin{proof}
					For $i \in \set{1, 2}$, let $E_i \defeq E_{\alpha_i}$ and $\iso_i \defeq \iso_{(\rho_i, \alpha_i)}$. Let $C\subseteq \dom(\pi)$ be an $E_1$-class. The equivariance of $\pi$ implies that $\pi(C)$ is an $E_2$-class. Since $(\rho_2, \alpha_2)$ is free $\mu_2$-almost everywhere, we may assume that $\pi(C) \subseteq \Free(X_2)$, in which case the map $\pi\vert C \colon C \to \pi(C)$ is a bijection.
					
					It remains to check the existence of $\varphi_1 \in \iso_1$ that closes the following diagram:
					\[
						\begin{tikzcd}
						C_1 \arrow{d}{\pi} \arrow[dashrightarrow]{r}{\varphi_1} & C_2 \arrow{d}{\pi} \\
						\pi(C_1) \arrow{r}{\varphi_2}& \pi(C_2).
						\end{tikzcd}
					\]
					Again, since $(\rho_2, \alpha_2)$ is free $\mu_2$-almost everywhere, we may assume that the maps
					\[
						\pi\vert C_1 \colon C_1 \to \pi(C_1) \qquad \text{and} \qquad \pi \vert C_2 \colon C_2 \to \pi(C_2)
					\]
					are bijections. Since $\phi_2$ is $(R, \Gamma)$-equivariant,  \[\varphi_1 \defeq (\pi\vert C_2)^{-1} \circ \varphi_2 \circ (\pi\vert C_1)\] is an equivariant bijection from $C_1$ to $C_2$; in other words, $\phi_1 \in \iso_1$, as desired.		
				\end{proof}
				
				Let $\Gamma$ be a countable group and let $\alpha_1 \colon \Gamma \acts (X_1, \mu_1)$ and $\alpha_2 \colon \Gamma \acts (X_2, \mu_2)$ be two probability measure\=/preserving actions of $\Gamma$. If $\alpha_2$ is free $\mu_2$-almost everywhere, then, by Lemma~\ref{lemma:equivariant}, a factor map $\pi\colon (X_1, \mu_1) \to (X_2, \mu_2)$ in the usual ergodic theory sense induces a factor map between the L\=/systems $\LS(\alpha_1, \mu_1)$ and $\LS(\alpha_2, \mu_2)$.
				
				\subsection{Closure properties of the class of shift L-systems}
				
				\mbox{}
				
				\smallskip
				
				\noindent In this subsection we show that the class of shift L-systems is closed under (certain) expansions and under amplifications. 
				
				\begin{lemma}\label{lemma:expansion}
					Let $(R, \Gamma)$ be a countable Borel groupoid and let $\LS = \LS(\rho, \alpha, \mu \times (\lambda^2)^\Gamma)$, where $\mu \in \P(R)$, be the shift L-system induced by the $[0;1]^2$-shift action of $(R, \Gamma)$. Let $Y$ be a standard Borel space and let
					\[
						f \colon R \times ([0;1]^2)^\Gamma \to Y
					\]
					be a Borel function that does not depend on the third coordinate, i.e., for all $r \in R$ and $\theta$, $\omega$, $\omega' \in [0;1]^\Gamma$,
					\[
						f(r, \theta, \omega) = f(r, \theta, \omega').
					\]
					Then $\LS[f]$ admits a factor map to a shift L-system.
				\end{lemma}
				\begin{proof}
					Set $Q \defeq R \times[0;1]^\Gamma$. The $[0;1]$-shift action of $(R, \Gamma)$ on $Q$ turns $(Q, \Gamma)$ into a countable Borel groupoid via
					\begin{equation*}
						\gamma \circ_{(r, \theta)} \delta \defeq\gamma\circ_r\delta; \qquad 
					\mathbf{1}_{(r, \theta)} \defeq \mathbf{1}_r;\qquad \text{and} \qquad
					\gamma^{-1}_{(r, \theta)} \defeq \gamma^{-1}_r.
					\end{equation*}
					Let $(\sigma, \alpha')$ denote the $[0;1]$-shift action of $(Q, \Gamma)$. If we identify $([0;1]^2)^\Gamma$ with $[0;1]^\Gamma \times [0;1]^\Gamma$ in the natural way, then \[
					R \times ([0;1]^2)^\Gamma = R \times [0;1]^\Gamma \times [0;1]^\Gamma = Q \times [0;1]^\Gamma,
					\] and, in fact, $\alpha' = \alpha$. By definition, for all $r \in R$ and $\theta$, $\omega \in [0;1]^\Gamma$,
					\[
						\sigma(r, \theta, \omega) = (r, \theta),
					\]
					so the value $f(x)$ is determined by $\sigma(x)$ for all $x$. Therefore, the identity function
					\[
					\id \colon R \times ([0;1]^2)^\Gamma \to Q \times [0;1]^\Gamma
					\]
					is a factor map from $\LS[f]$ to the shift L-system $\LS' \defeq \LS(\sigma, \alpha, \mu \times \lambda^\Gamma \times \lambda^\Gamma)$ induced by $(\sigma, \alpha)$.
				\end{proof}
				
				
				\begin{lemma}\label{lemma:LebesgueHF0}
					If $\LS$ is a shift L-system, then $\HF(\LS)$ factors to a shift L-system.
				\end{lemma}
				\begin{proof}
					Suppose that $\LS$ is induced by a shift action of a countable Borel groupoid $(R, \Gamma)$. We will proceed in three steps. First, we will construct a countable Borel groupoid $(Q, \Delta)$, where $\Delta = \HF(\Gamma)$. Then we will show that $\HF(\LS)$ is induced by an (almost everywhere) free action of $(Q, \Delta)$. Finally, we will define a measure-preserving $(Q, \Delta)$-equivariant Borel map from this action to the $[0;1]$-shift action of $(Q, \Delta)$, which will give us a desired factor map, thanks to Lemma~\ref{lemma:equivariant}.
					
					
					{\sc Step 1.} Let $\mu \in \P(R)$ and consider the $[0;1]$-shift action $(R, \Gamma) \acts R \times [0;1]^\Gamma$. Let $E$ denote the induced orbit equivalence relation. Define
					\[
						Q \defeq \HF(E \vert \Free(R \times [0;1]^\Gamma)) \qquad\text{and} \qquad \Delta \defeq \HF(\Gamma).
					\]
					Note that for each $x \in \Free(R \times [0;1]^\Gamma)$, the following map is a bijection between $\Gamma$ and $[x]_{E}$:
					\[
						\phi_x \colon \Gamma \to [x]_{E} \colon \gamma \mapsto \gamma \cdot x.
					\]
					Therefore, its amplification
					\[
						\tilde{\phi}_x\colon \HF(\Gamma) = \Delta \to \HF([x]_E) = [x]_{\tilde{E}}
					\]
					is a bijection between $\Delta$ and $[x]_{\tilde{E}}$. Fix a Borel map $x_0 \colon Q \to \Free(R \times [0;1]^\Gamma)$ such that $x_0(q) \in \U(q)$ for all $q \in Q$, and let \[\tilde{\phi}_q \defeq \tilde{\phi}_{x_0(q)}.\] Then for each $q \in Q$, the map $\tilde{\phi}_q$ is a bijection from $\Delta$ to $[x_0(q)]_{\tilde{E}} = [q]_{\tilde{E}}$. For $q \in Q$ and $\delta \in \Delta$, define
					\[
						\delta \cdot q \defeq \tilde{\phi}_q(\delta).
					\]
					Since $\tilde{\phi}_q \colon \Delta \to [q]_{\tilde{E}}$ is a bijection for each $q \in Q$, by Proposition~\ref{prop:UniqueCBG}, $(Q, \Delta)$ is equipped with a unique countable Borel groupoid structure. It is useful to observe that
					 \[
					 \delta \cdot q = \tilde{\phi}_q(\delta) = \tilde{\phi}_{x_0(q)}(\delta) = \delta \cdot x_0(q).
					 \]

					{\sc Step 2.} Now we turn to the shift L-system $\LS$. Suppose that $\LS = \LS(\rho, \alpha, \mu \times (\lambda \times \nu)^\Gamma)$, where \[(\rho, \alpha) \colon (R, \Gamma) \acts R \times ([0;1]\times Y)^\Gamma\] is the $([0;1]\times Y)$-shift action of $(R, \Gamma)$, $\mu \in \P(R)$, and $\nu \in \P(Y)$ is atomless. Here $Y$ is an arbitrary standard Borel space; we will specify a concrete choice for $Y$ later. Let 
					\[
						F \defeq \Free(R \times [0;1]^\Gamma) \times Y^\Gamma.
					\]
					Then $F$ is a conull $E_\alpha$-invariant Borel subset of $\Free(\rho, \alpha)$. We will now define a free action of $(Q, \Delta)$ on $H \defeq \HF(E_\alpha\vert F)$ (which is a conull $\tilde{E}_\alpha$-invariant Borel subset of~$\HF(E_\alpha)$).
					
					The construction is analogous to the one from Step~1. For each $x \in F$, define $\phi_x \colon \Gamma \to [x]_{E_\alpha}$ by
					\[
						\phi_x \colon \Gamma \to [x]_{E_\alpha} \colon \gamma \mapsto \gamma \cdot x.
					\]
					Then $\phi_x$ is a bijection between $\Gamma$ and $[x]_{E_\alpha}$. Therefore, $\tilde{\phi}_x \colon \Delta \to [x]_{\tilde{E}_\alpha}$ is a bijection from $\Delta$ to $[x]_{\tilde{E}_\alpha}$. Let $\sigma \colon F \to \Free(R \times [0;1]^\Gamma)$ denote the projection on the first two coordinates, i.e.,
					\[
						\sigma(r, \theta, y) \defeq (r, \theta) \qquad \text{for all $(r, \theta) \in \Free(R \times [0;1]^\Gamma)$ and $y \in Y^\Gamma$}.
					\]
					For every $h \in H$, we have $\tilde{\sigma}(h) \in Q$ and the map $\sigma \vert \U(h) \colon \U(h) \to \U(\tilde{\sigma}(h))$ is a bijection. Let $x_0(h)$ be the unique element of $\U(h)$ such that
					\[
						\sigma(x_0(h)) = x_0(\tilde{\sigma}(h)).
					\]
					Define $\tilde{\phi}_h \defeq \tilde{\phi}_{x_0(h)}$. Then $\tilde{\phi}_h \colon \Delta \to [h]_{\tilde{E}_\alpha}$ is a bijection. Hence, if we let
					\[
						\beta \colon \Delta \times H \colon (\delta, h) \mapsto \delta \cdot h \defeq \tilde{\phi}_h(\delta),
					\]
					then $(\tilde{\sigma}, \beta)$ is a free action of $(Q, \Delta)$ on $H$. Note that we again have
					\[
						\delta \cdot h = \tilde{\phi}_h(\delta) = \tilde{\phi}_{x_0(h)}(\delta) = \delta \cdot x_0(h).
					\]
					It is clear that the restriction of $\HF(\LS)$ to $H$ coincides with $\LS(\tilde{\sigma}, \beta, \mu \times (\lambda \times \nu)^\Gamma)$.
					
					{\sc Step 3.} So far we have constructed a countable Borel groupoid $(Q, \Delta)$ and a free action $(\tilde{\sigma}, \beta)$ of $(Q, \Delta)$ that essentially (i.e., up to an invariant null set) induces the L-system $\HF(\LS)$. It remains to define a factor map from that action to the L-system induced by the $[0;1]$-shift action of $(Q, \Delta)$.
					
					To that end, choose $Y$ to be $[0;1]^\Delta$ and $\nu$ to be $\lambda^\Delta$. Consider any $h \in H$. Suppose that $x_0(h) = (x, y)$, where $x \in \Free(R\times[0;1]^\Gamma)$ and $y \in Y^\Gamma = ([0;1]^\Delta)^\Gamma = [0;1]^{\Gamma \times \Delta}$. Define $\xi(h) \in [0;1]$ by
					\[
					\xi(h) \defeq y(\mathbf{1}_{\rho(x)}, \mathbf{1}_{\tilde{\sigma}(h)}).
					\]
					Here $\rho(x) \in R$ and $\tilde{\sigma}(h) \in Q$, so $\mathbf{1}_{\rho(x)} \in \Gamma$ and $\mathbf{1}_{\tilde{\sigma}(h)} \in \Delta$. Now define $\xi^\Delta \colon H \to [0;1]^\Delta$ by setting
					\[
					\xi^\Delta(h)(\delta) \defeq \xi (\delta \cdot h) \qquad \text{for all } h \in H \text{ and }\delta \in \Delta.
					\]
					By construction, the map
					\[
					(\tilde{\sigma}, \xi^\Delta) \colon H \to Q \times [0;1]^\Delta
					\]
					is $(Q, \Delta)$-equivariant. Due to Lemma~\ref{lemma:equivariant}, we only need to check that this map is measure-preserving.
					
					At this point, it is useful to recall that the measure on $H$, namely $\mu \times \lambda^\Gamma \times \lambda^{\Gamma \times \Delta}$, is concentrated on $F$. Since we have the freedom to choose the measure on $Q$, we can take it to be
					\[\tilde{\sigma}_\ast(\mu \times \lambda^\Gamma \times \lambda^{\Gamma\times \Delta}) = \sigma_\ast(\mu \times \lambda^\Gamma \times \lambda^{\Gamma \times \Delta}) = \mu \times \lambda^\Gamma,\]
					where the first equality follows from the fact that $\mu \times \lambda^\Gamma \times \lambda^{\Gamma \times \Delta}$ is concentrated on $F$, while the second equality is a consequence of the definition of $\sigma$. To finish the proof, it suffices to show that for each $x \in \Free(R \times [0;1]^\Gamma)$, the map
					\[
						\xi^\Delta_{x} \colon [0;1]^{\Gamma \times \Delta} \to [0;1]^\Delta \colon y \mapsto \xi^\Delta(x, y)
					\] 
					satisfies $(\xi^\Delta_{x})_\ast(\lambda^{\Gamma \times \Delta})=\lambda^\Delta$. To this end, fix some $x \in \Free(R \times [0;1]^\Gamma)$. For each $\delta \in \Delta$, let $\gamma_{x, \delta}$ be the unique element of $\Gamma$ such that
					$
						x_0(\delta \cdot x) = \gamma_{x, \delta} \cdot x
					$.
					Observe that the map
					\[
					\Delta \to \Gamma \times \Delta \colon \delta \mapsto (\gamma_{x, \delta}, \mathbf{1}_{\delta \cdot x})
					\]
					is injective. Indeed, we have
					\[
					\tilde{\phi}_x(\delta) = \delta \cdot x = \mathbf{1}_{\delta \cdot x} \cdot (\delta \cdot x) = \mathbf{1}_{\delta \cdot x} \cdot  x_0(\delta \cdot x)= \mathbf{1}_{\delta \cdot x} \cdot (\gamma_{x, \delta} \cdot x),
					\]
					and the map $\tilde{\phi}_x$ is injective. Let $y \in [0;1]^{\Gamma \times \Delta}$. We claim that
					\begin{equation}\label{eq:proj}
					\xi^\Delta_x(y)(\delta) = y(\gamma_{x, \delta}, \mathbf{1}_{\delta \cdot x}).
					\end{equation}
					Indeed, by definition,
					\[
					x_0 (\delta \cdot (x, y)) = \gamma_{x, \delta} \cdot (x,y) = (\gamma_{x, \delta} \cdot x, y') \qquad \text{and} \qquad \tilde{\sigma}(\delta \cdot (x, y)) = \delta \cdot x,
					\]
					where $y'$ is a particular element of $[0;1]^{\Gamma\times \Delta}$.
					Therefore,
					\[
					\xi^\Delta_x(y)(\delta) = \xi(\delta \cdot (x,y)) = y'(\mathbf{1}_{\rho(\gamma_{x, \delta} \cdot x)}, \mathbf{1}_{\delta \cdot x}) = y'(\mathbf{1}_{\gamma_{x, \delta} \cdot \rho(x)}, \mathbf{1}_{\delta \cdot x}).
					\]
					Since
					$
					\mathbf{1}_{\gamma_{x, \delta} \cdot \rho(x)} \circ_{\rho(x)} \gamma_{x, \delta} = \gamma_{x,\delta}
					$,
					by the definition of the shift action, we get
					\[
					y'(\mathbf{1}_{\gamma_{x, \delta} \cdot \rho(x)}, \mathbf{1}_{\delta \cdot x}) = y(\gamma_{x, \delta}, \mathbf{1}_{\delta \cdot x}),
					\]
					as desired. Equation \eqref{eq:proj} shows that $\xi^\Delta_x$ acts as the projection on the set of coordinates \[\{(\gamma_{x,\delta}, \mathbf{1}_{\delta \cdot x})\,:\, \delta \in \Delta\}.\] Therefore, it pushes $\lambda^{\Gamma \times \Delta}$ forward to $\lambda^\Delta$, and the proof is complete.
				\end{proof}
				
				\subsection{The Moser--Tardos algorithm for shift L-systems}
				
				\mbox{}
				
				\smallskip
				
				\noindent In this subsection we use Moser--Tardos theory to show that any correct instance $\B$ over a shift L\=/system $\LS$ admits a measurable solution. To do this, we will reduce $\B$ to a family $(\B_r)_{r \in R}$ of correct instances over the (countable) set $\Gamma$ indexed by the elements of $R$, where $(R, \Gamma)$ is the countable Borel groupoid whose shift action induces $\LS$.
				
				\begin{lemma}\label{lemma:main}
					Let $\LS$ be a shift L-system. Then every correct instance over $\LS$ has a measurable solution.
				\end{lemma}
				\begin{proof}
					Let $(R, \Gamma)$ be a countable Borel groupoid, let $\mu \in \P(R)$, and let $(\rho, \alpha) \colon (R, \Gamma) \acts R \times [0;1]^{\Gamma \times \N}$ be the $[0;1]^\N$-shift action of $(R, \Gamma)$. Let $\LS \defeq \LS(\rho, \alpha, \mu \times \lambda^{\Gamma \times \N})$. We use the following notation:
					\[
						X \defeq R \times [0;1]^{\Gamma \times \N}, \qquad E \defeq E_\alpha, \qquad \text{and} \qquad \iso \defeq \iso_{(\rho, \alpha)}.
					\]
					Suppose $\B$ is a correct instance over $\LS$. Due to Propositions~\ref{prop:BorelMoserTardos} and~\ref{prop:infinitestep}, it is enough to show that there exists a Borel table $\xi \colon X \times \N \to [0;1]$ such that
					\begin{equation}\label{eq:conullinvset}
						(\mu \times \lambda^{\Gamma \times \N})(\set{x \in X \,:\, \gamma \cdot x \in \Stab(\xi) \text{ for all } \gamma \in \Gamma}) = 1.
					\end{equation}
					We claim that the map
					\[
						\xi \colon X \times \N \to [0;1] \colon ((r, \theta), n) \mapsto \theta(\mathbf{1}_r)(n)
					\]
					satisfies~\eqref{eq:conullinvset}. Note that for every $\gamma \in \Gamma$,
					\[
						\xi(\gamma \cdot (r, \theta), n) = \theta(\gamma)(n).
					\]
					For each $x \in X$, there is a surjection
					\[
						\phi_x \colon \Gamma \to [x]_E \colon \gamma \mapsto \gamma \cdot x
					\]
					from $\Gamma$ onto $[x]_E$. Since the action $(\rho, \alpha)$ is free almost everywhere, $\phi_x$ is bijective for almost all $x \in X$. Hence, for almost every $x\in X$, the map $\phi_x$ can be used to define a correct instance $\B_x$ over $\Gamma$ by ``pulling back'' the restriction of $\B$ to $[x]_E$. Formally, we set
					\[
						\B_x \defeq \set{ \set{f \circ \phi_x \,:\, f \in B} \,:\, B \in \B}.
					\]
					Note that whenever $r \in R$ and $\theta$, $\omega \in [0;1]^{\Gamma \times \N}$ and both $(r, \theta)$ and $(r, \omega)$ belong to the free part of the action $(\rho, \alpha)$, the map $\gamma \cdot (r, \theta) \mapsto \gamma \cdot (r, \omega)$ is a well-defined $(R, \Gamma)$-equivariant bijection between $[(r, \theta)]_E$ and $[(r, \omega)]_E$. Therefore, since $\B$ is almost everywhere $\iso$-invariant, the following definition makes sense for almost all $r \in R$:
					\[
						\B_r \defeq \B_{(r, \theta)} \text{ for almost all } \theta \in [0;1]^{\Gamma \times \N}.
					\]
					Now, for almost every $r \in R$ and for all $\gamma \in \Gamma$, using Corollary~\ref{corl:countable}, we obtain
					\begin{align*}
						&\lambda^{\Gamma \times \N} (\set{\theta \in [0;1]^{\Gamma \times \N}\,:\, \text{$\gamma \cdot (r, \theta)$ is $\xi$-stable with respect to $\B$}})\\
						=\,&\lambda^{\Gamma \times \N} (\set{\theta \in [0;1]^{\Gamma \times \N}\,:\, \text{$\gamma$ is $\theta$-stable with respect to $\B_r$}}) \,=\, 1.
					\end{align*}
					An application of Fubini's theorem yields \eqref{eq:conullinvset}.
				\end{proof}
				
				\subsection{Completing the proof of Theorem~\ref{theo:Thm1}}
				
				\mbox{}
				
				\smallskip
				
				\noindent Now we have all the necessary ingredients to prove the following generalization of Theorem~\ref{theo:Thm1}:
				
				\begin{theo}[\textbf{Measurable LLL for shift L-systems}]
					Let $\LS$ be an L-system that admits a factor map to a shift L-system. Then Player~II has a winning strategy in the~LLL~Game over $\HF(\LS)$.
				\end{theo}
				\begin{proof}
					We need to verify that the class $\mathscr{C}$ of shift L-systems satisfies conditions~\ref{item:B1} and~\ref{item:B2} from~\prg\ref{prg:factors}. Condition~\ref{item:B1} is given by Lemma~\ref{lemma:LebesgueHF0}. It remains to show that if $\LS$ is a shift L-system and $\B$ is a correct instance over $\LS$, then there is a measurable solution $f$ to $\B$ such that $\LS[f]$ factors to another shift L-system.
					
					To that end, suppose that $\LS$ is induced by the $[0;1]^2$-shift action of a countable Borel groupoid $(R, \Gamma)$ with measure $\mu \times (\lambda^2)^\Gamma$, where $\mu \in \P(R)$. Consider the L-system $\LS'$ induced by the $[0;1]$-shift action of $(R, \Gamma)$ with measure $\mu \times \lambda^\Gamma$. The projection onto the first two coordinates, i.e., the map
					\[
						\pi \colon R \times [0;1]^\Gamma \times [0;1]^\Gamma \to R \times [0;1]^\Gamma \colon (r, \theta, \omega) \mapsto (r, \theta),
					\]
					is $(R, \Gamma)$-equivariant and measure-preserving, so, by Lemma~\ref{lemma:equivariant}, it is a factor map from $\LS$ to~$\LS'$. Due to Lemma~\ref{lemma:pushforward}, there is a correct instance $\pi(B)$ over $\LS'$ such that whenever $f'$ is a measurable solution to $\pi(\B)$, then $f' \circ \pi$ is a measurable solution to $\B$ (modulo an invariant null set). Lemma~\ref{lemma:main} does indeed provide a measurable solution $f'$ to $\pi(\B)$, so let $f \defeq f' \circ \pi$. By definition, $f$ does not depend on the third coordinate. Therefore, by Lemma~\ref{lemma:expansion}, $\LS[f]$ factors to a shift L-system, as desired.
				\end{proof}

		\section{The converse of Theorem~\ref{theo:Thm1} for actions of amenable groups}\label{section:converse}
		
		\noindent Corollary~\ref{corl:no_iterations} asserts that if a probability measure-preserving action $\alpha \colon \Gamma \acts (X, \mu)$ of a countable group~$\Gamma$ factors to the $[0;1]$-shift action, then every correct instance $\B$ over $\alpha$ admits a Borel solution $\mu$-almost everywhere. In this section we show that if $\Gamma$ is amenable, then the converse also holds. In fact, we will prove that even (seemingly) much weaker assumptions already imply the existence of a factor map to the $[0;1]$-shift.
		
		To articulate these weaker assumptions, we need a few definitions. An instance $\B$ over a set~$X$ is \emph{$\epsilon$-correct}, where $0 < \epsilon \leq 1$, if the neighborhood of each $B \in \B$ is countable, and there exists a function $\omega \colon \B \to [0;1)$ such that for all $B \in \B$,
		\[
			\mathbb{P}[B] \leq \epsilon^{|\dom(B)|}  \omega(B) \prod_{B' \in \Nbhd_\B(B)} (1 - \omega(B')).
		\]
		Hence, correct is the same as $1$-correct, and
		\[
			\text{$B$ is $\epsilon$-correct}\,\Longrightarrow\,\text{$B$ is $\epsilon'$-correct} \quad \text{whenever $0<\epsilon \leq \epsilon'\leq 1$}.
		\]
		An instance $\B$ over a set $X$ is \emph{discrete} if there exist a finite set $S$ and a Borel function $\phi \colon [0;1] \to S$ such that for all $B \in \B$ and $w$, $w' \colon \dom(B) \to [0;1]$ with $\phi \circ w = \phi \circ w'$, we have
		\[
			w \in B \,\Longleftrightarrow\, w' \in B.
		\]
		In other words, $\B$ is discrete if the bad events in $\B$ can be identified with subsets of $\finf{X}{S}$, where~$S$ is equipped with the probability measure $\phi_\ast(\lambda)$ (see Remark~\ref{remk:other_spaces}). Most instances of the~LLL that appear in combinatorial applications are discrete. If $\phi_\ast(\lambda)$ is the uniform probability measure on $S$, then $\B$ is said to be \emph{uniformly discrete}.
		
		Given a graph $G$ on a set $X$, an \emph{instance \ep{or the~LLL}} over $G$ is an instance $\B$ over $X$ such that:
		\begin{itemize}
			\item[--] for each $B \in \B$, the (finite) graph $G \vert \dom(B)$ is connected;
			\item[--] if $B \in \B$, $S \subseteq X$, and $\phi \colon S \to \dom(B)$ is an isomorphism between $G \vert S$ and $G \vert \dom(B)$, then
			\[
				\set{w \circ \phi \,:\, w \in B} \in \B.
			\]
		\end{itemize}
		Note that if $\alpha \colon \Gamma \acts X$ is an action of a countable group $\Gamma$ generated by a set $S \subseteq \Gamma$, then every instance over $G(\alpha, S)$ is in particular an instance over $\alpha$.
		
		Now we are ready to state the first version of the converse theorem.
		
		\begin{theo}\label{theo:Thm2}
			Let $\alpha \colon \Gamma \acts (X, \mu)$ be a free ergodic measure-preserving action of a countably infinite amenable group~$\Gamma$ on a standard probability space $(X, \mu)$. Suppose that $S \subseteq \Gamma$ is a finite generating set and let $G \defeq G(\alpha, S)$. The following statements are equivalent:
			\begin{enumerate}[label=\ep{\normalfont{\roman*}}]
				\item there exists $\epsilon \in (0;1]$ such that for every $\epsilon$-correct uniformly discrete Borel instance $\B$ over $G$, there is a Borel map $f\colon X \to [0;1]$ with $\mu(\Def_\B(f)) < 1$;
				\item $\alpha$ factors to the shift action $\Gamma \acts ([0;1]^\Gamma, \lambda^\Gamma)$.
			\end{enumerate}
		\end{theo}
		
		In general, the conclusion of Theorem~\ref{theo:Thm2} fails for infinite $S$. To see this, consider any free ergodic measure-preserving action $\alpha \colon \Gamma \acts (X,\mu)$ of a countably infinite amenable group $\Gamma$ on a standard probability space $(X,\mu)$ and set $G \defeq G(\alpha, \Gamma)$. We claim that for every correct Borel instance $\B$ over~$G$, there is a Borel map $f\colon X \to [0;1]$ with $\mu(\Def_\B(f)) = 0$, regardless of the choice of $\alpha$. Indeed, \[G = E_\alpha \setminus \set{(x, x) \,:\, x\in X},\] so $G$ only depends on the orbit equivalence relation $E_\alpha$ and not on the action $\alpha$ itself. Since, by a theorem of Dye and Ornstein--Weiss~\cite[Theorem~10.7]{KechrisMiller}, all free probability measure-preserving ergodic actions of countable amenable groups are orbit-equivalent, we may replace $\alpha$ by a $[0;1]$-shift action and apply Corollary~\ref{corl:no_iterations}.
		
		However, by keeping track of slightly more information than just the graph $G(\alpha, S)$, one can still establish an analog of Theorem~\ref{theo:Thm2} for infinite $S$ (and in particular for groups that are not finitely generated). An \emph{\ep{$S$-}labeled graph} on $X$ is a family $G = (G_\gamma)_{\gamma \in S}$ of graphs on $X$ indexed by the elements of a given countable set $S$. Note that the sets $G_\gamma$ are not required to be disjoint, i.e., the same edge can receive more than one label. A labeled graph $G$ on a standard Borel space is \emph{Borel} if each $G_\gamma$ is Borel. An isomorphism between labeled graphs $G_1$ and $G_2$ must preserve the labeling, i.e., it has to be an isomorphism between each $(G_1)_\gamma$ and $(G_2)_\gamma$ individually. For an $S$-labeled graph $G$ on $X$ and a subset $X' \subseteq X$, let $G\vert X'$ denote the $S$-labeled graph on $X'$ given by $(G \vert X')_\gamma \defeq G_\gamma \vert X'$. For an $S$-labeled graph~$G$, its \emph{underlying graph} is $\bigcup_{\gamma \in S} G_\gamma$. A labeled graph $G$ is \emph{connected} if its underlying graph is connected. The definition of an instance over $G$ extends verbatim to the case when $G$ is labeled. If $\alpha \colon \Gamma \acts X$ is an action of a countable group $\Gamma$ on a set $X$ and $S \subseteq \Gamma$ is a generating set, then $G_\ell(\alpha, S)$ denotes the $S$-labeled graph on $X$ given by
		\[
			(x,y) \in  (G_\ell(\alpha, S))_\gamma \,\vcentcolon\Longleftrightarrow\, x \neq y \text{ and } (\gamma \cdot x = y \text{ or } \gamma \cdot y = x).
		\]
		Thus, the underlying graph of $G_\ell(\alpha, S)$ is $G(\alpha, S)$. Now we have the following:
		\begin{theobis}{theo:Thm2}\label{theo:Thm2bis}
			Let $\alpha \colon \Gamma \acts (X, \mu)$ be a free ergodic measure-preserving action of a countably infinite amenable group~$\Gamma$ on a standard probability space $(X, \mu)$. Let $S \subseteq \Gamma$ be a generating set and let $G \defeq G_\ell(\alpha, S)$. The following statements are equivalent:
			\begin{enumerate}[label=\ep{\normalfont{\roman*}}]
				\item there exists $\epsilon \in (0;1]$ such that for every $\epsilon$-correct uniformly discrete Borel instance $\B$ over $G$, there is a Borel map $f\colon X \to [0;1]$ with $\mu(\Def_\B(f)) < 1$;
				\item $\alpha$ factors to the shift action $\Gamma \acts ([0;1]^\Gamma, \lambda^\Gamma)$.
			\end{enumerate}
		\end{theobis}
		
		Notice that Theorems~\ref{theo:Thm2} and \ref{theo:Thm2bis} also demonstrate that the local finiteness requirement in the statement of Theorem~\ref{theo:approxLLL} is necessary.
		
		\subsection{Outline of the proof}
		
		\mbox{}
		
		\smallskip
		
		\noindent The proofs of Theorems~\ref{theo:Thm2} and \ref{theo:Thm2bis} are almost identical, so we will present them simultaneously. We only have to show the forward implication in both statements (the other direction is handled by Corollary~\ref{corl:no_iterations}). Here we briefly sketch our plan of attack.
		
		For simplicity, assume that $\Gamma = \Z$ and let $\alpha \colon \Z \acts (X, \mu)$ be a free ergodic probability measure\=/preserving action of $\Z$ on a standard probability space $(X, \mu)$. There is a simple criterion, called \emph{Sinai's factor theorem}, that determines whether there is a factor map $\pi \colon (X, \mu) \to ([0;1]^\Z, \lambda^\Z)$: Such $\pi$ exists if and only if $\alpha$ has infinite \emph{Kolmogorov--Sinai entropy}. The Kolmogorov--Sinai entropy of $\alpha$ is defined as follows. Consider any Borel function $f \colon X \to I$ to a finite set $I$. The \emph{Shannon entropy} of $f$ measures how ``uncertain'' the value $f(x)$ is when $x \in X$ is chosen randomly with respect to $\mu$; formally,
		\[
			h_\mu(f) \defeq - \sum_{i \in I} \mu(f^{-1}(i)) \log_2\mu(f^{-1}(i)).
		\]
		Now the action comes into play: Given $x \in X$ and $n \in \N$, we record the sequence of values
		\[
			f((-n) \cdot x),\, f((-n+1) \cdot x),\, \ldots,\, f(n \cdot x);
		\]
		this gives us a tuple of elements of $I$ of length $2n+1$. Let $f_n \colon X \to I^{2n+1}$ be the corresponding function. We can compute the \emph{average} amount of uncertainty in $f_n(x)$ per symbol; in other words, we can look at the quantity $h_\mu(f_n(x))/(2n+1)$. It turns out that, as $n$ grows, this quantity decreases, so there exists a limit
		\[
			H_\mu(\alpha, f) \defeq \lim_{n \to \infty} \frac{h_\mu(f_n)}{2n+1}.
		\]
		This limit is called the \emph{Kolmogorov--Sinai entropy} of $f$ with respect to $\alpha$. The \emph{Kolmogorov--Sinai entropy} of the action $\alpha$ itself measures the ``maximum level of uncertainty'' that can be achieved with respect to $\alpha$; formally, it is defined as
		\[
			H_\mu(\alpha) \defeq \sup_f H_\mu(\alpha, f),
		\]
		where $f$ is ranging over all Borel functions from $X$ to a finite set. As mentioned previously, $\alpha$ factors to the $[0;1]$-shift action if and only if $H_\mu(\alpha)= \infty$.
		
		How can we use the~LLL to prove that $H_\mu(\alpha) = \infty$? By definition, we have to exhibit Borel functions~$f$ with arbitrarily large values of $H_\mu(\alpha, f)$. But $H_\mu(\alpha, f)$ is, in some sense, a ``global'' parameter---it is defined in terms of the measures of certain subsets of $X$---while instances of the~LLL can only put ``local'' constraints on the function~$f$. However, high value of $H_\mu(\alpha,f)$ indicates that the functions $f_n$ behave very ``randomly'' or ``unpredictably.'' Thus, what we need is a way to measure ``randomness'' or ``unpredictability'' deterministically, which we can then apply to the values of $f_n$ at each point instead of looking at the function $f_n$ as a whole.
		
		There is indeed a convenient deterministic analog of Shannon's entropy, namely the so-called \emph{Kolmogorov complexity}. Roughly speaking, a finite sequence $w$ of symbols has high Kolmogorov complexity if there is no way to encode it by a significantly shorter sequence. Our instance of the~LLL will require $f_n(x)$ to have high Kolmogorov complexity for all $n \in \N$ and $x \in X$. We will show that solving this instance, even partially, guarantees that $H_\mu(\alpha, f)$ must also be high.
		
		The structure of the rest of this section is as follows. In \prg\ref{prg:preliminaries} we list the necessary definitions and preliminary results regarding the structure of amenable groups, Kolmogorov--Sinai entropy of their actions (including the version of Sinai's factor theorem with a general amenable group in place of $\Z$), and Kolmogorov complexity. In \prg\ref{prg:complexity_vs_entropy} we prove the main lemma that connects Kolmogorov complexity and Kolmogorov--Sinai entropy. Finally, 
		\prg\ref{prg:finish} completes the proof by constructing a series of instances of the~LLL whose solutions necessarily have high Kolmogorov complexity and hence high Kolmogorov--Sinai entropy.
		
		\subsection{Preliminaries}\label{prg:preliminaries}
		
		\subsubsection*{Background on amenable groups}
		
		\mbox{}
		
		\smallskip
		
		\noindent For a group $\Gamma$, subsets $S$, $T \subseteq \Gamma$, and an element $\gamma \in \Gamma$, let
		\[
			\gamma S \defeq \set{\gamma\delta \,:\, \delta \in S}, \qquad S\gamma \defeq \set{\delta\gamma \,:\, \delta \in S}, \qquad \text{and}\qquad ST \defeq \set{\delta_1\delta_2 \,:\, \delta_1 \in S, \delta_2\in T}.
		\]
		Recall that a countable group $\Gamma$ is \emph{amenable} if it admits a \emph{F\o{}lner sequence}, i.e., a sequence $(\Phi_n)_{n=0}^\infty$ of nonempty finite subsets of $\Gamma$ such that for all $\gamma \in \Gamma$,
		\begin{equation}\label{eq:Folner}
			\lim_{n \to \infty} \frac{|\gamma \Phi_n \symdif \Phi_n|}{|\Phi_n|} = 0,
		\end{equation}
		where $\symdif$ denotes symmetric difference of sets. Note that if $S \subseteq \Gamma$ is a generating set and~\eqref{eq:Folner} holds for all $\gamma \in S$, then $(\Phi_n)_{n=0}^\infty$ is a F\o{}lner sequence (see~\cite[Remark~5.12]{KechrisMiller}).
				
		\begin{prop}\label{prop:connectedFolner}
			Let $\Gamma$ be a countably infinite amenable group and let $S \subseteq \Gamma$ be a generating set. Let $G \defeq \operatorname{Cay}(\Gamma, S)$ denote the corresponding Cayley graph. Then $\Gamma$ admits a F\o lner sequence $(\Phi_n)_{n=0}^\infty$ such that every \ep{finite} graph $G \vert \Phi_n$ is connected. 
		\end{prop}
		\begin{proof}
			Let $\gamma_0$, $\gamma_1$, \ldots{} be a list of all the elements of $S$ in an arbitrary order, possibly with repetitions (so the list is infinite even if $S$ is finite) and let $(\Phi_n)_{n=0}^\infty$ be a F\o lner sequence for $\Gamma$. By passing to a subsequence if necessary, we can arrange that for all $n \in \N$,
			\begin{equation}\label{eq:Folner1}
				\sum_{i=0}^n\frac{|\gamma_i \Phi_n \symdif \Phi_n|}{|\Phi_n|} \leq \frac{1}{n}.
			\end{equation}
			Suppose $G\vert \Phi_n$ has $k_n$ connected components and let $\Phi_{n, 1}$, \ldots, $\Phi_{n, k_n} \subseteq \Phi_n$ denote their vertex sets. For all $i \in \N$ and $1 \leq j_1 < j_2 \leq k_n$, we have $\gamma_i \Phi_{n, j_1} \cap \Phi_{n, j_2} = \0$, so
			\begin{equation}\label{eq:componentwise}
				\gamma_i \Phi_n \symdif \Phi_n = \bigcup_{j = 1}^{k_n} (\gamma_i \Phi_{n,j} \symdif \Phi_{n,j}),
			\end{equation}
			and the union on the right-hand side of \eqref{eq:componentwise} is disjoint. Therefore,
			\[
				\sum_{i=0}^{n}\frac{|\gamma_i\Phi_n \symdif \Phi_n|}{|\Phi_n|} \,=\, \frac{\sum_{j=1}^{k_n} \sum_{i=0}^{n} |\gamma_i \Phi_{n,j} \symdif \Phi_{n,j}|}{\sum_{j = 1}^{k_n} |\Phi_{n,j}|}.
			\]
			If for all $1 \leq j \leq k_n$, we have
			\[
				\sum_{i=0}^{n}\frac{|\gamma_i \Phi_{n,j} \symdif \Phi_{n,j}|}{|\Phi_{n,j}|} > \frac{1}{n},
			\]
			then
			\[
				\frac{\sum_{j=1}^{k_n} \sum_{i=0}^{n} |\gamma_i \Phi_{n,j} \symdif \Phi_{n,j}|}{\sum_{j = 1}^{k_n} |\Phi_{n,j}|}\,>\, 
				\frac{\sum_{j=1}^{k_n} \frac{1}{n} |\Phi_{n,j}|}{\sum_{j = 1}^{k_n} |\Phi_{n,j}|} = \frac{1}{n},
			\]
			which contradicts~\eqref{eq:Folner1}. Hence, there is some $1 \leq j_n \leq k_n$ such that
			\[
			\sum_{i=0}^{n}\frac{|\gamma_i \Phi_{n,j_n} \symdif \Phi_{n,j_n}|}{|\Phi_{n,j_n}|} \leq \frac{1}{n}.
			\]
			Then $(\Phi_{n, j_n})_{n=0}^\infty$ is a desired F\o lner sequence consisting of connected sets.
		\end{proof}
		
		\begin{corl}\label{corl:goodFolner}
			Let $\Gamma$ be a countably infinite amenable group and let $S \subseteq \Gamma$ be a generating set. Let $G \defeq \operatorname{Cay}(\Gamma, S)$ denote the corresponding Cayley graph. Then $\Gamma$ admits a F\o lner sequence $(\Phi_n)_{n=0}^\infty$ such that:
			\begin{itemize}
				\item[--] for each $n \in \N$, the graph $G \vert \Phi_n$ is connected;
				\item[--] $\mathbf{1} \in \Phi_0 \subset \Phi_1 \subset \ldots$, where $\mathbf{1}$ is the identity element of $\Gamma$;
				\item[--] $\bigcup_{n=0}^\infty \Phi_n = \Gamma$;
				\item[--] $\lim_{n \to \infty} |\Phi_n|/\log_2n = \infty$.
			\end{itemize}
		\end{corl}
		\begin{proof}
			Proposition~\ref{prop:connectedFolner} gives a F\o lner sequence $(\Phi_n)_{n=0}^\infty$ satisfying the first condition. Since $\Gamma$ is infinite, we have $|\Phi_n| \to \infty$ as $n \to \infty$. If $\mathbf{1} \not \in\Phi_n$ for some $n \in \N$, then choose any $\gamma \in \Phi_n$ and replace $\Phi_n$ with $\Phi_n \gamma^{-1}$. Now we construct a new sequence $(\Phi'_n)_{n=0}^\infty$ inductively. Let $\gamma_0$, $\gamma_1$, \ldots{} be a list of all the elements of $S$ in an arbitrary order, possibly with repetitions. For $n \in \N$, set $S_n \defeq \set{\gamma_0, \ldots, \gamma_n}$. Let $B_n$ denote the collection of all the elements of $\Gamma$ that can be expressed as products of at most $n$ elements of $S_n$. Note that $\mathbf{1} \in B_n$ and the graph $G \vert B_n$ is connected. Let $\Phi_0' \defeq \Phi_0$. On step $n+1$, choose $N$ large enough so that
			\[
			|\Phi_N| > n \cdot \left(\sum_{i=0}^n |\Phi'_i| + |B_n| + \log_2n\right),
			\]
			and define
			\[
			\Phi'_{n+1} \defeq \Phi_N \cup \bigcup_{i=0}^n \Phi'_{i} \cup B_n.
			\]
			Clearly, $(\Phi'_n)_{n=0}^\infty$ is a F\o lner sequence satisfying all the requirements.
		\end{proof}
				
		We will need a result of Ornstein and Weiss on the existence of quasi-tilings in amenable groups. A~family $A_1$, \ldots, $A_k$ of finite sets is said to be \emph{$\epsilon$-disjoint}, $\epsilon > 0$, if there exist pairwise disjoint subsets $B_1 \subseteq A_1$, \ldots, $B_k \subseteq A_k$ such that for all $1 \leq i \leq k$,
			\[
				|B_i| \geq (1 - \epsilon) |A_i|.
			\]
		A finite set $A$ is \emph{$(1-\epsilon)$-covered} by $A_1$, \ldots, $A_k$ if
			\[
			\left|A \cap \bigcup_{i=1}^k A_i\right| \geq (1-\epsilon) |A|.
			\]
		Let $\Gamma$ be a countable group and let $A$, $A_1$, \ldots, $A_k$ be finite subsets of $\Gamma$. An \emph{$\epsilon$-quasi-tiling} of~$A$ by the sets $A_1$, \ldots, $A_k$ is a collection $C_1$, \ldots, $C_k$ of finite subsets of $\Gamma$ such that:
		\begin{itemize}
			\item[--] for each $1 \leq i \leq k$, we have $A_iC_i \subseteq A$ and the family of sets sets $(A_i \gamma)_{\gamma \in C_i}$ is $\epsilon$-disjoint;
			\item[--] the sets $A_1C_1$, \ldots, $A_kC_k$ are pairwise disjoint;
			\item[--] $A$ is $(1-\varepsilon)$-covered by the sets $A_1C_1$, \ldots, $A_kC_k$.
		\end{itemize}
			
		\begin{theo}[Ornstein--Weiss~\cite{OrnsteinWeiss}; see also~{\cite[Theorem~2.6]{WardZhang}} and {\cite[Proposition~2.3]{ZhengChenYang}}]\label{theo:quasitiling}
			Let $\Gamma$ be a countable amenable group and let $(\Phi_n)_{n=0}^\infty$ be a F\o lner sequence in $\Gamma$. Then for all $\epsilon > 0$ and for all $n \in \N$, there exist $k$, $\ell_1$, \ldots, $\ell_k$, $m_0 \in \N$ with $n \leq \ell_1 < \ell_2 < \ldots < \ell_k$ such that for each $m \geq m_0$, there exists an $\epsilon$-quasi-tiling of $\Phi_m$ by $\Phi_{\ell_1}$, \ldots, $\Phi_{\ell_k}$.
		\end{theo}
		
		\subsubsection*{Background on Kolmogorov--Sinai entropy}
		
		\mbox{}
		
		\smallskip
		
		\noindent An important invariant of an amenable probability measure-preserving system is its Kolmogorov--Sinai entropy. It is usually defined in terms of finite Borel partitions; however, for our purposes it will be more convenient to define it in terms of Borel functions to a finite set (the two notions are, of course, equivalent).
		
		Let $(X,\mu)$ be a standard probability space. A \emph{\ep{finite} coloring} of $X$ is a function $f \colon X \to I$, where $I$ is a finite set. The \emph{Shannon entropy} of a Borel finite coloring $f\colon X \to I$  is defined to be
		\[
			h_\mu(f) \defeq -\sum_{i \in I} \mu(f^{-1}(i)) \log_2\mu(f^{-1}(i)).
		\]
		Here we adopt the convention that $0 \cdot \log_2 0 = 0$. Note that $0\leq h_\mu(f) \leq \log_2 |I|$.
		
		Let $\alpha \colon \Gamma \acts (X, \mu)$ be a probability measure-preserving action of a countable amenable group $\Gamma$. For a finite coloring $f \colon X \to I$ and a set $\Phi \in \fins{\Gamma}$, let $f^\Phi \colon X \to I^\Phi$ denote the finite coloring defined by setting, for all $x \in X$ and $\gamma \in \Phi$,
		\[
			f^\Phi(x)(\gamma) \defeq f(\gamma \cdot x).
		\]
		The \emph{Kolmogorov--Sinai entropy} of a Borel finite coloring $f$ with respect to $\alpha$ is given by
		\begin{equation}\label{eq:entropy}
			H_\mu(\alpha, f) \defeq \lim_{n \to \infty} \frac{h_\mu (f^{\Phi_n})}{|\Phi_n|},
		\end{equation}
		where $(\Phi_n)_{n=0}^\infty$ is a F\o lner sequence in $\Gamma$. Due to a fundamental result of Ornstein and Weiss~\cite{OrnsteinWeiss}, the limit in~\eqref{eq:entropy} always exists and is independent of the choice of $(\Phi_n)_{n=0}^\infty$. Note that we again have $0 \leq H_\mu(\alpha, f) \leq \log_2|I|$, where $I$ is the range of $f$. The \emph{Kolmogorov--Sinai entropy} of $\alpha$ is defined as follows:
		\[
			H_\mu(\alpha) \defeq \sup \set{H_\mu(\alpha, f) \,:\, f \text{ is a Borel finite coloring of }X}.
		\]
		We will use the following special case of a generalization of Sinai's factor theorem to actions of arbitrary amenable groups proven by Ornstein and Weiss:
		
		\begin{theo}[Ornstein--Weiss \cite{OrnsteinWeiss}]
			Let $\alpha \colon \Gamma \acts (X,\mu)$ be a free ergodic measure-preserving action of a countably infinite amenable group $\Gamma$ on a standard probability space $(X, \mu)$. Suppose that $H_\mu(\alpha) = \infty$. Then there exists a factor map $\pi \colon (X, \mu) \to ([0;1]^\Gamma, \lambda^\Gamma)$ to the $[0;1]$-shift action of $\Gamma$.
		\end{theo}

		\subsubsection*{Background on Kolmogorov complexity}
		
		\mbox{}
		
		\smallskip
	
		\noindent We will use some basic properties of Kolmogorov complexity, for which our references are \cite{UVSh} and \cite{LiVitanyi}. Let $2^\ast$ denote the set of all finite sequences of zeroes and ones (including the empty sequence). For $w \in 2^\ast$, let $|w|$ be the length of~$w$. For a partial function $D \colon 2^\ast \rightharpoonup 2^\ast$, define $K_D \colon 2^\ast \to \N \cup \set{\infty}$~via
		\[
			K_D (x) \defeq \inf\set{|w|\,:\, D(w) = x}.
		\]
		Given two partial functions $D_1$, $D_2 \colon 2^\ast \rightharpoonup 2^\ast$, we say that $D_1$ \emph{minorizes} $D_2$ (notation:~$D_1 \leq_K D_2$) if there is a constant $c \in \N$ such that for all $x \in 2^\ast$, $K_{D_1}(x) \leq K_{D_2}(x) + c$. Clearly, $\leq_K$ is a preorder. If $\mathscr{C}$ is a class of partial functions $2^\ast \rightharpoonup 2^\ast$, then $D \in \mathscr{C}$ is \emph{optimal} in $\mathscr{C}$ if for all $D' \in \mathscr{C}$,~$D \leq_K D'$.
		
		Given a subset $\Oracle \subseteq 2^\ast$, we say that a function is computable \emph{relative to $\Oracle$} if it can be computed by a Turing machine enhanced with the ability to determine whether a given word $w$ is in $\Oracle$; see \cite[Definition 1.7.7]{LiVitanyi} for the precise definition. In this context, the set $\Oracle$ is usually referred to as an \emph{oracle}. The class of all partial maps $D \colon 2^\ast \rightharpoonup 2^\ast$ that are computable relative to a fixed oracle $\Oracle$ is denoted by $\mathscr{C}_\Oracle$. A cornerstone of the theory of Kolmogorov complexity is the following observation:
		
		\begin{theo}[Solomonoff--Kolmogorov; see~{\cite[Lemma~2.1.1]{LiVitanyi} and \cite[Theorem~1]{UVSh}}]\label{theo:SolKol}
			Fix an oracle~$\Oracle$. There exists a map $D \in \mathscr{C}_\Oracle$ that is optimal in $\mathscr{C}_\Oracle$.
		\end{theo}
		
		In the light of Theorem~\ref{theo:SolKol}, we can define the \emph{Kolmogorov complexity} of a word $x \in 2^\ast$ relative to an oracle~$\Oracle$ to be $K_\Oracle(x) \defeq K_D(x)$,
		for some fixed optimal $D \in \mathscr{C}_\Oracle$. Note that if $D$, $D' \in \mathscr{C}_\Oracle$ are two optimal functions, then there is a constant $c \in \N$ such that $|K_D(x) - K_{D'}(x)| \leq c$ for all $x \in 2^\ast$; in this sense, the value $K_\Oracle(x)$ is defined up to an additive constant.
		
		 The following property of Kolmogorov complexity will play a crucial role in our argument.
		
		\begin{prop}\label{prop:KolmogorovProb}
			Fix an oracle $\Oracle$. Let $c$, $n \in \N$ and let $\nu_n$ denote the uniform probability measure on~$2^n$. Then
			\[
				\nu_n(\set{x \in 2^n\,:\, K_\Oracle(x) \leq n - c}) < 2^{-c+1}.
			\]
		\end{prop}
		\begin{proof}
			Let $D \in \mathscr{C}_\Oracle$ be the optimal function used in the definition of Kolmogorov complexity relative to $\Oracle$. There are exactly $2^{n-c+1}-1$ sequences of zeroes and ones of length at most $n-c$, so there can be at most $2^{n-c+1}-1$ words $x \in 2^\ast$ with $K_\Oracle(x) = K_D(x) \leq n-c$. Therefore,
			\[
			\nu_n(\set{x \in 2^n\,:\, K_\Oracle(x) \leq n - c}) \leq \frac{2^{n-c+1}-1}{2^n} = 2^{-c+1}-2^{-n} < 2^{-c+1},
			\]
			as desired.
		\end{proof}		
		
		\subsection{Kolmogorov complexity vs. Kolmogorov--Sinai entropy}\label{prg:complexity_vs_entropy}
		
		\mbox{}
		
		\smallskip
		
		\noindent For the rest of this section, we fix a countably infinite amenable group $\Gamma$, a generating set $S\subseteq \Gamma$, a standard probability space $(X, \mu)$, and a free ergodic measure-preserving action $\alpha \colon \Gamma \acts X$. We also fix a F\o lner sequence $(\Phi_n)_{n=0}^\infty$ in $\Gamma$ satisfying the requirements of Corollary~\ref{corl:goodFolner}, i.e., such that:
		\begin{itemize}
			\item[--] for each $n \in \N$, the graph $\operatorname{Cay}(\Gamma, S) \vert \Phi_n$ is connected;
			\item[--] $\mathbf{1} \in \Phi_0 \subset \Phi_1 \subset \ldots$, where $\mathbf{1}$ is the identity element of $\Gamma$;
			\item[--] $\bigcup_{n=0}^\infty \Phi_n = \Gamma$;
			\item[--] $\lim_{n \to \infty} |\Phi_n|/\log_2n = \infty$.
		\end{itemize}
		Let $\Oracle$ be an oracle relative to which the following data are computable:
		\begin{enumerate}
			\item[--] the group structure of $\Gamma$ and a fixed linear ordering $<$ on $\Gamma$ (we may assume, for instance, that the ground set of $\Gamma$ is $\N$);
			\item[--] the sequence $(\Phi_n)_{n=0}^\infty$ (meaning that the function $n \mapsto |\Phi_n|$ and the characteristic function of the set $\set{(\gamma, n) \in \Gamma \times \N \,:\, \gamma \in \Phi_n}$ are computable relative to $\Oracle$).
		\end{enumerate}
		Given a set $\Phi \in \fins{\Gamma}$ and a function $w \colon \Phi \to 2^s$, we can use the ordering on $\Gamma$ to identify $w$ with a sequence of zeroes and ones of length $s|\Phi|$. This identification enables us to talk about the Kolmogorov complexity $K_\Oracle(w)$ of $w$. For a Borel coloring $f \colon X \to 2^s$, a point $x \in X$, and $n \in \N$, let
		\[
			f_n(x) \defeq f^{\Phi_n}(x).
		\]
		Note that the map $X \times \N \to \N \colon (x, n)\mapsto K_\Oracle(f_n(x))$ is Borel.
		
		The following lemma connects Kolmogorov complexity and Kolmogorov--Sinai entropy:
		
		\begin{lemma}[\textbf{High complexity $\Longrightarrow$ high entropy}]\label{lemma:complexity_vs_entropy}
			Let $s \in \N$ and let $f \colon X \to 2^s$ be a Borel coloring of $X$. Then
			\[
				\limsup_{m \to \infty} \int_X\frac{K_\Oracle(f_m(x))}{|\Phi_m|} \D{\mu(x)} \leq H_\mu(\alpha,f). 
			\]
		\end{lemma}
		\begin{proof}
			Our argument is inspired by the work of Brudno~\cite{Brudno}, who established a close relationship between Kolmogorov complexity and Kolmogorov--Sinai entropy in the case of $\Z$-actions (see also~\cite{Moriakov} for an extension of Brudno's theory to a wider class of amenable groups).
			
			Fix $\epsilon \in (0;1)$. Choose $n \in \N$ large enough so that for all $\ell \geq n$,
			\[
				\frac{h_\mu(f_\ell)}{|\Phi_\ell|} \leq H_{\mu}(\alpha,f) + \epsilon.
			\]
			For most of the proof, $n$ and $s$ will be treated as fixed constants. In particular, the implied constants in asymptotic notation may depend on $n$ and $s$.
			
			Using Theorem~\ref{theo:quasitiling}, choose $k$, $\ell_1$, \ldots, $\ell_k$, $m_0 \in \N$ so that $n \leq \ell_1 < \ldots < \ell_k$ and for every $m \geq m_0$, there exists an $\epsilon$-quasi-tiling of $\Phi_m$ by $\Phi_{\ell_1}$, \ldots, $\Phi_{\ell_k}$.
			For each $m \geq m_0$, let $C_{m,1}$, \ldots, $C_{m, k}$ be an $\epsilon$-quasi-tiling of $\Phi_m$ by $\Phi_{\ell_1}$, \ldots, $\Phi_{\ell_k}$, chosen in such a way that the map \[m \mapsto (C_{m,i})_{i=1}^k\] is computable relative to $\Oracle$. (For instance, we can choose the sequence of finite sets $C_{m,1}$, \ldots, $C_{m, k}$ to be the first in some computable ordering.)
			
			We will now devise a binary code for pairs of the form $(m, w)$, where $m \geq m_0$ and $w \colon \Phi_m \to 2^s$. The decoding procedure for this code will be computable relative to $\Oracle$, so the length of the code will provide an upper bound on the Kolmogorov complexity of $w$ (modulo an additive constant).
			
			Let $c_0(m)$ be the sequence of $\lfloor \log_2m\rfloor+1$ ones followed by a single zero and let $c_1(m)$ be any fixed binary code for the integer $m$ of length exactly $\lfloor \log_2m\rfloor+1$. Note that for any $c \in 2^\ast$, the pair $(m, c)$ is uniquely determined by $c_0(m)\concat c_1(m)\concat c$. Also note that \[|c_0(m)\concat c_1(m)| \leq 2\log_2m + O(1) = o_{m \to \infty} (|\Phi_m|).\]
			
			Consider the set
				\[
					\Lambda_m \defeq \Phi_m \setminus \bigcup_{i=1}^k\Phi_{\ell_i} C_{m,i}.
				\]
			We can view $w\vert\Lambda_m$ as a binary word of length $s|\Lambda_m|$, which we denote by $c_2(m, w)$. Note that the length of $c_2(m,w)$ is determined by $m$ and satisfies
			\[
				|c_2(m,w)| = s|\Lambda_m| \leq \epsilon s|\Phi_m|,
			\]
			since $\Phi_m$ is $(1-\epsilon)$-covered by the family $(\Phi_{\ell_i}C_{m,i})_{i=1}^k$.
			
			For each $1 \leq i \leq k$ and $\gamma \in C_{m,i}$, we can encode the mapping
			\[
				w_{i, \gamma} \colon \Phi_{\ell_i} \to 2^s \colon \delta \mapsto w(\delta \gamma)
			\]
			as a binary word of length $s|\Phi_{\ell_i}|$. For each binary word $u$ of length $s|\Phi_{\ell_i}|$, let $\eta_{i,u}(m,w)$ be the frequency of $u$ among the words of the form $w_{i,\gamma}$, i.e., let
			\[
				\eta_{i, u}(m, w) \defeq |\set{\gamma \in C_{m,i}\,:\, w_{i,\gamma} = u}|.
			\]
			By definition,
			\begin{equation}\label{eq:sum_of_frequences}
				\sum_u \eta_{i, u}(m,w) = |C_{m,i}|,
			\end{equation}
			where the summation is over all binary words of length $s|\Phi_{\ell_i}|$. Since $0 \leq \eta_{i, u}(m, w) \leq |C_{m,i}|$, we can encode $\eta_{i, u}(m, w)$ by a binary word $c_3(m, w, i, u)$ of length exactly
			\[
				\lfloor\log_2|C_{m,i}|\rfloor + 1 \leq \log_2|C_{m,i}| + O(1) \leq \log_2|\Phi_m| + O(1),
			\]
			so the length of $c_3(m,w,i,u)$ is determined by~$m$. Let $c_3(m,w,i)$ denote the concatenation of all the words of the form $c_3(m,w,i,u)$ with $u$ ranging over the binary words of length $s|\Phi_{\ell_i}|$, and let \[c_3(m,w) \defeq c_3(m,w,1) \concat \ldots\concat c_3(m,w,k).\] The length of $c_3(m,w)$ is at most $O(\log_2|\Phi_m|) = o_{m\to\infty} (|\Phi_m|)$.
			
			Now we consider the word
			\[
				w_i \defeq w \vert (\Phi_{\ell_i} C_{m,i}).
			\]
			Since $w_i$ is determined by the family $(w_{i,\gamma}\,:\, \gamma \in C_{m,i})$, there are at most
			\[
				\frac{|C_{m,i}|!}{\prod_u \eta_{i,u}(m,w)!}
			\]
			options for $w_i$, where the product is over all binary words of length $s|\Phi_{\ell_i}|$. Due to Stirling's formula and equation~\eqref{eq:sum_of_frequences}, we have
			\[
				\log_2\left(\frac{|C_{m,i}|!}{\prod_u \eta_{i,u}(m,w)!}\right) \leq -|C_{m,i}| \sum_u \frac{\eta_{i,u}(m,w)}{|C_{m,i}|}\log_2\left(\frac{\eta_{i,u}(m,w)}{|C_{m,i}|}\right).
			\]
			Thus, provided that $m$ and all the $\eta_{i,u}(m,w)$'s are given, $w_i$ can be encoded by a binary word $c_4(m,w,i)$ of length
			\[
				\left\lceil\log_2\left(\frac{|C_{m,i}|!}{\prod_u \eta_{i,u}(m,w)!}\right)\right\rceil \leq -|C_{m,i}| \sum_u \frac{\eta_{i,u}(m,w)}{|C_{m,i}|}\log_2\left(\frac{\eta_{i,u}(m,w)}{|C_{m,i}|}\right) + O(1).
			\]
			Let \[c_4(m,w) \defeq c_4(m,w,1)\concat \ldots\concat c_4(m,w,k).\] The length of $c_4(m,w)$ is at most
			\[
				-\sum_{i=1}^k|C_{m,i}| \sum_u \frac{\eta_{i,u}(m,w)}{|C_{m,i}|}\log_2\left(\frac{\eta_{i,u}(m,w)}{|C_{m,i}|}\right) + o_{m\to\infty}(|\Phi_m|).
			\]
			
			Our code for $(m,w)$ is the concatenation
			\[
				\code(m,w) \defeq c_0(m)\concat c_1(m) \concat c_2(m,w) \concat c_3(m,w) \concat c_4(m,w).
			\]
			It is clear that $\code(m,w)$ uniquely determines $m$ and $w$ and, moreover, the map
			\[
				\code(m,w) \mapsto (m,w)
			\]
			is computable relative to $\Oracle$.
			Combining the above upper bounds for the lengths of $c_0(m)$, $c_1(m)$, $c_2(m,w)$, $c_3(m,w)$, and $c_4(m,w)$, we get
			\[
				|\code(m,w)| \leq \epsilon s|\Phi_m| -\sum_{i=1}^k|C_{m,i}| \sum_u \frac{\eta_{i,u}(m,w)}{|C_{m,i}|}\log_2\left(\frac{\eta_{i,u}(m,w)}{|C_{m,i}|}\right) + o_{m\to\infty}(|\Phi_m|).
			\]
			Since $K_\Oracle(w) \leq |\code(m,w)| + O(1)$, the same asymptotic upper bound holds for $K_\Oracle(w)$ as well.
			
			Applying this analysis to a point $x \in X$, we obtain
			\begin{equation}\label{eq:Kolmogorov_at_a_point}
				\frac{K_\Oracle(f_m(x))}{|\Phi_m|} \leq \epsilon s -\sum_{i=1}^k\frac{|C_{m,i}|}{|\Phi_m|} \sum_u \frac{\eta_{i,u}(m,f_m(x))}{|C_{m,i}|}\log_2\left(\frac{\eta_{i,u}(m,f_m(x))}{|C_{m,i}|}\right) + o_{m\to\infty}(1).
			\end{equation}
			
			\begin{claim}\label{claim:oneterm}
				For each $m \geq m_0$, $1 \leq i \leq k$, and a binary word $u$ of length $s|\Phi_{\ell_i}|$, we have
				\[
					\int_X \frac{\eta_{i,u}(m,f_m(x))}{|C_{m,i}|} \D{\mu(x)} = \mu(f_{\ell_i}^{-1}(u)).
				\]	
			\end{claim}
			\begin{claimproof}
				Recall that, by definition,
				\[
					\eta_{i, u}(m, w) = |\set{\gamma \in C_{m,i}\,:\, w_{i,\gamma} = u}| = |\set{\gamma \in C_{m,i}\,:\, w(\delta\gamma) = u(\delta) \text{ for all } \delta \in \Phi_{\ell_i}}|.
				\]
				Notice that, for each $\delta \in \Phi_i$, $f_m(x)(\delta \gamma) = f(x)(\delta\gamma) = f(\gamma \cdot x)(\delta) = f_{\ell_i}(\gamma \cdot x)(\delta)$, so we have
				\begin{align*}
					\eta_{i, u}(m, f_m(x)) = |\set{\gamma \in C_{m,i}\,:\,f_m(x)(\delta\gamma) = u(\delta) \text{ for all } \delta \in \Phi_{\ell_i}}|= |\set{\gamma \in C_{m,i}\,:\, f_{\ell_i}(\gamma \cdot x) = u}|.
				\end{align*}
				Therefore,
				\begin{align*}
					\int_X \frac{\eta_{i,u}(m,f_m(x))}{|C_{m,i}|} \D{\mu(x)} &= \frac{1}{|C_{m,i}|}\int_X |\set{\gamma \in C_{m,i}\,:\, f_{\ell_i}(\gamma \cdot x) = u}| \D{\mu(x)}\\
					&= \frac{1}{|C_{m,i}|}\int_X |\set{\gamma \in C_{m,i}\,:\, f_{\ell_i}(x) = u}| \D{\mu(x)}\\
					&= \frac{1}{|C_{m,i}|} \cdot |C_{m,i}|\cdot \mu(f_{\ell_i}^{-1}(u)) = \mu(f_{\ell_i}^{-1}(u)),
				\end{align*}
				where the second equality holds since $\mu$ is $\alpha$-invariant.
			\end{claimproof}
			
			The function $\alpha \mapsto -\alpha \log_2\alpha$ is concave for $0\leq \alpha \leq 1$, so, by Claim~\ref{claim:oneterm}, for $1 \leq i \leq k$, we have
			\begin{align*}
				&-\int_X \sum_u \frac{\eta_{i,u}(m,f_m(x))}{|C_{m,i}|}\log_2\left(\frac{\eta_{i,u}(m,f_m(x))}{|C_{m,i}|}\right) \D{\mu(x)}\\
				\leq\,& -\sum_u \int_X\frac{\eta_{i,u}(m,f_m(x))}{|C_{m,i}|} \D{\mu(x)} \log_2 \left(\int_X\frac{\eta_{i,u}(m,f_m(x))}{|C_{m,i}|} \D{\mu(x)}\right)\\
				=\,&-\sum_u \mu(f_{\ell_i}^{-1}(u)) \log_2\mu(f_{\ell_i}^{-1}(u)) \\
				=\,& h_\mu(f_{\ell_i}).
			\end{align*}
			Combining this with \eqref{eq:Kolmogorov_at_a_point} gives
			\[
				\int_X\frac{K_\Oracle(f_m(x))}{|\Phi_m|}\D{\mu(x)} \leq \epsilon s + \sum_{i=1}^k \frac{|C_{m,i}|}{|\Phi_m|} h_\mu(f_{\ell_i}) + o_{m \to \infty}(1).
			\]
			Recall that, by the choice of $n$,
			\[
				\frac{h_\mu(f_\ell)}{|\Phi_\ell|} \leq H_{\mu}(\alpha,f) + \epsilon
			\]
			for all $\ell \geq n$. Since $n \leq \ell_1 < \ldots < \ell_k$, we get
			\[
				\int_X\frac{K_\Oracle(f_m(x))}{|\Phi_m|}\D{\mu(x)} \leq \epsilon s + \sum_{i=1}^k \frac{|C_{m,i}|}{|\Phi_m|} \cdot  |\Phi_{\ell_i}|(H_\mu(\alpha,f)+\epsilon) + o_{m \to \infty}(1).
			\]
			Since the sets $(\Phi_{\ell_i}C_{m,i})_{i=1}^k$ are pairwise disjoint, and for each $1 \leq i \leq k$, the family of sets $(\Phi_{\ell_i}\gamma)_{\gamma \in C_{m,i}}$ is $\epsilon$-disjoint, we have
			\[
				|\Phi_m| \geq \sum_{i=1}^k |\Phi_{\ell_i}C_{m,i}| \geq (1-\epsilon) \sum_{i=1}^k |\Phi_{\ell_i}||C_{m,i}|,
			\]
			so
			\begin{equation}\label{eq:Kolmogorov_integrated}
				\int_X\frac{K_\Oracle(f_m(x))}{|\Phi_m|}\D{\mu(x)} \leq \epsilon s + \frac{1}{1-\epsilon} (H_\mu(\alpha,f) + \epsilon) + o_{m \to \infty}(1).
			\end{equation}
			Since~\eqref{eq:Kolmogorov_integrated} holds for every $\epsilon \in (0;1)$ and for every sufficiently large $m$, we finally obtain
			\[
				\int_X\frac{K_\Oracle(f_m(x))}{|\Phi_m|}\D{\mu(x)} \leq H_\mu(\alpha,f) + o_{m \to \infty} (1),
			\]
			as desired.
		\end{proof}
		
		\subsection{Building the instances}\label{prg:finish}
		
		\mbox{}
		
		\smallskip
		
		\noindent Define
		\[
			G \defeq \begin{cases}
				G(a, S) &\text{if }S\text{ is finite};\\
				G_\ell(a, S) &\text{if }S\text{ is infinite}.
			\end{cases}
		\]
		If $S$ is finite, let $\operatorname{Cay}(\Gamma, S)$ denote the corresponding (unlabeled) Cayley graph; otherwise, assume that $\operatorname{Cay}(\Gamma, S)$ is $S$\=/labeled. Fix $\epsilon \in (0;1]$ such that for every $\epsilon$-correct uniformly discrete Borel instance $\B$ over $G$, there is a Borel map $f\colon X \to [0;1]$ with $\mu(\Def_\B(f)) < 1$. Our goal is to show that $H_\mu(\alpha) = \infty$.
		
		For each pair of nonnegative integers $s$, $t \in \N$ with $s \geq t$, we will construct a uniformly discrete Borel instance $\B(s, t)$ over $G$. For convenience, we will view each bad event $B \in \B(s,t)$ as a set of partial maps in $\finf{X}{2^s}$ instead of the usual $\finf{X}{[0;1]}$.
		
		For $n \in \N$, let $G_n \defeq \operatorname{Cay}(\Gamma, S) \vert \Phi_n$. By the choice of the F\o{}lner sequence $(\Phi_n)_{n=0}^\infty$, each graph $G_n$ is connected. Given an isomorphism $\phi \colon \Phi_n \to X$ between the graphs $G_n$ and $G \vert \im(\phi)$, let $B_\phi(s,t)$ denote the bad event with domain $\im(\phi)$ consisting of all maps $w \colon \im(\phi) \to 2^s$ such that
		\[
			K_\Oracle (w \circ \phi) \leq (s-t)|\Phi_n|.
		\]
		Let $\B(s,t)$ denote the collection of all bad events $B_\phi(s,t)$ defined above. It is clear that $\B(s,t)$ is a uniformly discrete Borel instance of the~LLL over $G$. We will show that there is some $t \in \N$ such that for all $s \geq t$, the instance $\B(s, t)$ is also $\epsilon$-correct.
		
		Set
		\[
			d \defeq \begin{cases}
				|S\cup S^{-1}| &\text{if }S\text{ is finite};\\
				2 &\text{if }S\text{ is infinite}.
			\end{cases}
		\]
		
		\begin{lemma}\label{lemma:counting_embeddings}
			For all $\gamma \in \Phi_n$ and $x \in X$, the number of isomorphic embeddings $\phi \colon \Phi_n \to X$ of $G_n$ into $G$ with $\phi(\gamma) = x$ does not exceed $d^{|\Phi_n|}$.
		\end{lemma}
		\begin{proof}
			If $S$ is finite, then $\Delta(G) \leq |S \cup S^{-1}| = d$; if $S$ is infinite, then for any given $\delta \in S$ and any $y \in X$, the graph $G$ can contain at most $2$ edges labeled by $\delta$ that are incident to $y$. Now the statement follows from the connectedness of~$G_n$. 
		\end{proof}

		\begin{lemma}\label{lemma:counting_neighbours}
			Let $B \in \B(s,t)$ and $k \in \N$. Then
			\[
				|\set{B'\in \Nbhd_{\B(s,t)}(B) \,:\, |\dom(B')| = k}| \leq |\dom(B)| \cdot kd^k.
			\]
		\end{lemma}
		\begin{proof}
			If there is no $n \in \N$ with $|\Phi_n| = k$, then there is nothing to prove, so suppose that $|\Phi_n| = k$ (such $n$ is unique since the sequence $(\Phi_n)_{n=0}^\infty$ is strictly increasing). Consider any $B'\in \Nbhd_{\B(s,t)}(B)$ with $|\dom(B')| = k$. Then $B' = B_\phi(s,t)$ for some embedding $\phi \colon \Phi_n \to X$ of $G_n$ into $G$. As $B'\in \Nbhd_{\B(s,t)}(B)$, we have $\im(\phi) \cap \dom(B) \neq \0$, i.e., there exist $\gamma \in \Phi_n$ and $x \in \dom(B)$ such that $\phi(\gamma) = x$. Now we have $|\dom(B)|$ choices for $x$, $k$ choices for $\gamma$, and, by Lemma~\ref{lemma:counting_embeddings}, at most $d^k$ choices for $\phi$ given $\gamma$ and $x$.
		\end{proof}
		
		By Proposition~\ref{prop:KolmogorovProb}, if $B \in \B(s,t)$ and $|\dom(B)| \eqqcolon n$, then
		\[
			\mathbb{P}[B] < 2^{-tn+1}.
		\]
		In the light of Lemma~\ref{lemma:counting_neighbours}, to show that $\B(s, t)$ is $\epsilon$-correct, it suffices to find a sequence $(\omega_n)_{n=1}^\infty$ with each $\omega_n \in [0;1)$ such that for all $n \in \N$,
		\begin{equation}\label{eq:first_bound}
			2^{-tn+1} \leq \epsilon^n \omega_n \prod_{k=1}^\infty (1-\omega_k)^{nkd^k}.
		\end{equation}
		Note that inequality~\eqref{eq:first_bound} does not mention $s$; in other words, if it holds for some $(\omega_n)_{n=1}^\infty$ and for all $n \in \N$, then $\B(s,t)$ is $\epsilon$-correct for \emph{all} $s \geq t$.
		
		To solve \eqref{eq:first_bound}, let $\delta > 0 $ be sufficiently small and set $\omega_n \defeq \delta^n$. For every $n \in \N$, we have
		\[
			1-\omega_n > e^{-2\omega_n}.
		\]
		Moreover, if we choose $\delta < 1/d$, then the series
		\[
			\sum_{k=1}^\infty \omega_k k d^k = \sum_{k=1}^\infty k (\delta d)^k
		\]
		converges; denote its sum by $c$. Now we have
		\[
			\prod_{k=1}^\infty (1-\omega_k)^{nkd^k} > e^{-2cn},
		\]
		so \eqref{eq:first_bound} holds as long as
		\[
			2^{-tn+1} \leq \epsilon^n \cdot  \delta^n \cdot  e^{-2cn},
		\]
		for which it suffices to have
		\begin{equation}\label{eq:t_bound}
			t \geq 1 - \log_2(\epsilon \delta e^{-2c}).
		\end{equation}
		Choose any $t \in \N$ that satisfies~\eqref{eq:t_bound}; for all $s \geq t$, the instance $\B(s, t)$ is $\epsilon$-correct.
		
		If $\B(s, t)$ is $\epsilon$-correct, then there must exist a Borel map $f \colon X \to 2^s$ with $\mu(\Def_\B(f)) < 1$. Since the action is ergodic, the set $[X \setminus \Def_\B(f)]_{E_G}$ is conull. We claim that for all $x \in [X \setminus \Def_\B(f)]_{E_G}$,
		\[
			\liminf_{m \to \infty} \frac{K_\Oracle(f_m(x))}{|\Phi_m|} \geq s-t.
		\]
		Indeed, let $x \in X$ and let $y \in X \setminus \Def_\B(f)$ be such that $x\, E_G \, y$. The sequence $(\Phi_n)_{n \in \N}$ is increasing and exhaustive, so there is $m_0 \in \N$ such that $y \in \Phi_m \cdot x$ for all $m \geq m_0$. Since $y \not \in \Def_\B(f)$, for each $m \geq m_0$, the restriction of $f$ to $\Phi_m \cdot x$ satisfies the constraints laid down by $\B(s,t)$. By definition, this means that $K_\Oracle(f_m(x)) > (s-t) |\Phi_m|$, as claimed.
		
		Using Fatou's lemma together with Lemma~\ref{lemma:complexity_vs_entropy}, we obtain
		\begin{align}
			s-t &\leq \int_X \liminf_{m \to \infty} \frac{K_\Oracle(f_m(x))}{|\Phi_m|}\D{\mu(x)} \nonumber\\
			&\leq \liminf_{m \to \infty} \int_X \frac{K_\Oracle(f_m(x))}{|\Phi_m|}\D{\mu(x)} \label{eq:final_bound}\\
			&\leq \limsup_{m \to \infty} \int_X \frac{K_\Oracle(f_m(x))}{|\Phi_m|} \D{\mu(x)}\leq H_\mu(\alpha, f).\nonumber
		\end{align}
		Since for any $s \geq t$, we can find $f$ such that \eqref{eq:final_bound} holds, $H_\mu(\alpha) = \infty$, as desired.

		\printbibliography
		
	\begin{appendices}
		
	
	
	
		
	\section{Proof of Theorem~\ref{theo:MoserTardos}}\label{app:MoserTardos}

	\noindent For completeness, we present here a self-contained account of the proof of Theorem~\ref{theo:MoserTardos}. Recall that for this theorem, we have fixed a set $X$ and a correct instance $\B$ over $X$. The theorem reads:
	
	\begin{theocopy}{theo:MoserTardos}\label{theo:MTcopy}
		Let $\omega \colon \B \to [0;1)$ be a function witnessing the correctness of $\B$ and let $S \in \dom(\B)$. Then
		\begin{equation*}
		\sum_{\mathscr{P} \in \mathbf{Piles}(S)} \lambda^{\supp(\mathscr{P})\times \N} \left(\mathbf{App}(\mathscr{P})\right) \leq \sum_{\substack{B \in \B \,:\\\dom(B) = S}}\frac{\omega(B)}{1 - \omega(B)}.
		\end{equation*}
	\end{theocopy}	

	In order to prove Theorem~\ref{theo:MTcopy}, we establish a correspondence between neat piles and labeled trees of a certain kind. Let us start with some general definitions. Let $A$ be a set. We use $A^\ast$ to denote the set of all finite sequences of elements of $A$. For $w \in A^\ast$, $|w|$ denotes the length of $w$, so that \[w \eqqcolon (w_0, \ldots, w_{|w|-1}).\] If $w \neq \0$, then let $\mathbf{tail}(w) \defeq w_{|w|-1}$ denote the last element of $w$. For $w$, $w' \in A^\ast$, we say that $w$ is a \emph{prefix} of $w'$ (or an \emph{initial segment} of $w'$; notation: $w \subseteq w'$) if either $w = w'$, or else, there is a sequence $u \in A^\ast$ such that $w' = w \concat u$ (where $\concat$ denotes concatenation). A \emph{tree} over $A$ is a set $\mathscr{T} \subseteq A^\ast \setminus \set{\0}$ that is closed under taking nonempty initial segments and contains a unique sequence of length $1$, called the \emph{root} of $\mathscr{T}$.
	
	Given a linear order $<$ on $A$, the corresponding \emph{lexicographical order} $<_{\operatorname{lex}}$ on $A^\ast$ is defined in the usual way; i.e., for distinct $w$, $w' \in A^\ast$, we have $w <_{\operatorname{lex}}w'$ if and only if either $w \subset w'$, or else, $w' \not \subset w$ and the smallest index $i$ such that $w_i \neq w'_i$ satisfies $w_i < w'_i$. We will use the following properties of the lexicographical order:
	\begin{itemize}
		\item[--] if $w \subset w'$, then $w <_{\operatorname{lex}} w'$;
		\item[--] if $w <_{\operatorname{lex}} w'$ and $w \not \subset w'$, then $w \concat u <_{\operatorname{lex}} w'\concat u$ for all $u \in A^\ast$.
	\end{itemize}
	
	Now we return to our problem. Fix an arbitrary linear order $<$ on $\dom(\B)$ and the induced lexicographical order $<_{\operatorname{lex}}$ on $\dom(\B)^\ast$. Consider a neat pile $\mathscr{P}$ with a unique top element~$\tau_0$. Given any $\tau \in \mathscr{P}$, a \emph{$\tau$-path} in $\mathscr{P}$ is a sequence $\tau_0$, $\tau_1$, \ldots, $\tau_k$ of elements of $\mathscr{P}$ such that $\tau_k = \tau$ and \[\tau_k \prec \ldots \prec \tau_1 \prec \tau_0.\]
	
	Notice that if $\tau$, $\tau_1$, $\tau_2 \in \mathscr{P}$ satisfy $\tau_1 \prec \tau$, $\tau_2 \prec \tau$, and $\dom(\tau_1) = \dom(\tau_2)$, then $\tau_1 = \tau_2$. Indeed, suppose that $\tau_1 \neq \tau_2$. By the definition of the relation $\prec$, there exist $x_1$, $x_2 \in \dom(\tau_1) = \dom(\tau_2)$ such that $\tau_1(x_1) = \tau(x_1) - 1$ and $\tau_2(x_2) = \tau(x_2)-1$. Since $\mathscr{P}$ is neat, $\tau_1(x_2) < \tau(x_2)$; since the graphs of $\tau_1$ and $\tau_2$ are disjoint, we get $\tau_1(x_2) < \tau_2(x_2)$. Similarly, we have $\tau_2(x_1) < \tau_1(x_1)$, which contradicts the neatness of~$\mathscr{P}$.
	
	From the above observation it follows that for any given $w \in \dom(\B)^\ast$, there exists at most one path $\tau_0$, \ldots, $\tau_{|w|-1}$ in $\mathscr{P}$ such that $\dom(\tau_i) = w_i$ for all $i$. If such a path exists, then we denote its last element $\tau_{|w|-1}$ by $\tau_{\mathscr{P}}(w)$. Note that if $v \subset w$ and $\tau_\mathscr{P}(w)$ is defined, then $\tau_\mathscr{P}(v)$ is also defined and we have $\tau_\mathscr{P}(v) \neq \tau_\mathscr{P}(w)$. For $\tau \in \mathscr{P}$, let $w_{\mathscr{P}}(\tau)$ denote the lexicographically largest sequence $w \in \dom(\B)^\ast$ such that $\tau_{\mathscr{P}}(w) = \tau$ (the set of such sequences is nonempty and finite, so this definition makes sense).
	
	\begin{lemma}\label{lemma:tree_order}
		Let $\mathscr{P}$ be a neat pile with a unique top element. Let $\tau$, $\tau' \in \mathscr{P}$ and let $x \in \dom(\tau)\cap \dom(\tau')$. Then
		\[
			\tau'(x) < \tau(x) \,\Longleftrightarrow\, w_{\mathscr{P}}(\tau) <_{\operatorname{lex}} w_{\mathscr{P}}(\tau').
		\]
	\end{lemma}
	\begin{proof}
		Set $w \defeq w_{\mathscr{P}}(\tau)$ and $w' \defeq w_{\mathscr{P}}(\tau')$. For concreteness, assume $\tau'(x) < \tau(x)$. Then there exists a sequence $\tau_1$, \ldots, $\tau_k$ of elements of $\mathscr{P}$ such that $\tau_1 = \tau$, $\tau_k = \tau'$, and $\tau_k \prec \ldots \prec \tau_1$. Thus, the sequence \[v \defeq w \concat (\dom(\tau_2), \ldots, \dom(\tau_k))\] satisfies $\tau_{\mathscr{P}}(v) = \tau'$, which yields
		$w <_{\operatorname{lex}} v \leq_{\operatorname{lex}} w'$, as desired.
	\end{proof}
	
	Call a sequence $w \in \dom(\B)^\ast$ \emph{proper} if $w_i \cap w_{i+1} \neq \0$ for all $0 \leq i < |w| -1$. Note that if $\mathscr{P}$ is a neat pile with a unique top element, then $\tau_\mathscr{P}(w)$ can only be defined for proper $w$. A tree $\mathscr{T}$ over $\dom(\B)$ is \emph{proper} if every element of $\mathscr{T}$ is proper. For $S \in \dom(\B)$, let $\mathbf{Trees}(S)$ denote the set of all proper finite trees with root $(S)$.
			
	\begin{lemma}\label{lemma:injection}
		Let $S \in \dom(\B)$ and let $\mathscr{P} \in \mathbf{Piles}(S)$. Then
		\[
				\mathscr{T}_\mathscr{P} \defeq \set{w_{\mathscr{P}}(\tau)\,:\, \tau \in \mathscr{P}} \in \mathbf{Trees}(S).
		\]
		Furthermore, the map $\mathbf{Piles}(S) \to \mathbf{Trees}(S) \colon \mathscr{P} \mapsto \mathscr{T}_\mathscr{P}$ is injective.
	\end{lemma}
	\begin{figure}[h]
		\centering
				\begin{tikzpicture}[scale=0.6]
				
				\begin{scope}[xshift=-10cm]
				\draw (0, -1) -- (0,6);
				\draw (-1, 0) -- (10,0);
				
				\node at (-0.5, 1) {$0$};
				\node at (-0.5, 3) {$1$};
				\node at (-0.5, 5) {$2$};
				
				\node[circle, draw, inner sep=1pt, outer sep=0](10) at (1,1) {$\tau_1$};
				\node[circle, draw, inner sep=1pt, outer sep=0](20) at (3,1) {$\tau_1$};
				\node[circle, draw, inner sep=1pt, outer sep=0](21) at (3,3) {$\tau_3$};
				\node[circle, draw, inner sep=1pt, outer sep=0](30) at (5,1) {$\tau_3$};
				\node[circle, draw, inner sep=1pt, outer sep=0](31) at (5,3) {$\tau_4$};
				\node[circle, draw, inner sep=1pt, outer sep=0](32) at (5,5) {$\tau_5$};
				\node[circle, draw, inner sep=1pt, outer sep=0](40) at (7,1) {$\tau_2$};
				\node[circle, draw, inner sep=1pt, outer sep=0](41) at (7,3) {$\tau_3$};
				\node[circle, draw, inner sep=1pt, outer sep=0](42) at (7,5) {$\tau_4$};
				\node[circle, draw, inner sep=1pt, outer sep=0](50) at (9,1) {$\tau_2$};
				\node[circle, draw, inner sep=1pt, outer sep=0](51) at (9,3) {$\tau_4$};
				
				\node at (1,-0.5) {$x_1$};
				\node at (3,-0.5) {$x_2$};
				\node at (5,-0.5) {$x_3$};
				\node at (7,-0.5) {$x_4$};
				\node at (9,-0.5) {$x_5$};
				
				\node[anchor=west] at (10, 0) {$X$};
				\node[anchor=south] at (0, 6) {$\N$};
				
				\draw (10) -- (20);
				\draw (21) -- (30) -- (41);
				\draw (31) -- (42) -- (51);
				\draw (40) -- (50);
				\end{scope}
				
				\begin{scope}[xshift=6cm]
					\node[draw=none, inner sep=2pt, outer sep=0] (1) at (0,6) {$\set{x_3}$};
					\node[draw=none, inner sep=2pt, outer sep=0] (2) at (0,4) {$\set{x_3, x_4, x_5}$};
					\node[draw=none, inner sep=2pt, outer sep=0] (3) at (-2,2) {$\set{x_2, x_3, x_4}$};
					\node[draw=none, inner sep=2pt, outer sep=0] (4) at (2,2) {$\set{x_4, x_5}$};
					\node[draw=none, inner sep=2pt, outer sep=0] (5) at (-2,0) {$\set{x_1, x_2}$};
					
					\draw (1) -- (2) -- (3) -- (5) (2) -- (4);
					
					\node at (-2.5, 5.5) {$\mathscr{T}_1$};
				\end{scope}
				
				\begin{scope}[xshift=13cm]
				\node[draw=none, inner sep=2pt, outer sep=0] (1) at (0,6) {$\set{x_3}$};
				\node[draw=none, inner sep=2pt, outer sep=0] (2) at (0,4) {$\set{x_3, x_4, x_5}$};
				\node[draw=none, inner sep=2pt, outer sep=0] (3) at (0,2) {$\set{x_2, x_3, x_4}$};
				\node[draw=none, inner sep=2pt, outer sep=0] (4) at (2,0) {$\set{x_4, x_5}$};
				\node[draw=none, inner sep=2pt, outer sep=0] (5) at (-2,0) {$\set{x_1, x_2}$};
				
				\draw (1) -- (2) -- (3) -- (5) (3) -- (4);
				
				\node at (-2.5, 5.5) {$\mathscr{T}_2$};
				\end{scope}
				\end{tikzpicture}
		\caption{A neat pile $\mathscr{P} = \set{\tau_1, \tau_2, \tau_3, \tau_4, \tau_5}$. Depending on whether $\set{x_2, x_3, x_4} < \set{x_4, x_5}$ or not, we either have $\mathscr{T}_\mathscr{P} = \mathscr{T}_1$ or $\mathscr{T}_\mathscr{P} = \mathscr{T}_2$.}
	\end{figure}
	
	\begin{proof}				
		Let $\mathscr{P} \in \mathbf{Piles}(S)$ and let $\mathscr{T} \defeq \mathscr{T}_\mathscr{P}$. Suppose that $w \in \mathscr{T}$ and $v \in \dom(\B)^\ast \setminus \set{\0}$ is an initial segment of $w$. We claim that $w_\mathscr{P}(\tau_\mathscr{P}(v)) = v$, and hence $v \in \mathscr{T}$. Indeed, let $w = v \concat u$ for $u \in \dom(\B)^\ast$. If $w_\mathscr{P}(\tau_\mathscr{P}(v)) \neq v$, then $v <_{\operatorname{lex}} w_\mathscr{P}(\tau_\mathscr{P}(v))$ and $v \not \subset w_\mathscr{P}(\tau_\mathscr{P}(v))$, so $w = v \concat u <_{\operatorname{lex}} w_\mathscr{P}(\tau_\mathscr{P}(v)) \concat u$. However, $\tau_\mathscr{P}(w_\mathscr{P}(\tau_\mathscr{P}(v)) \concat u) = \tau_\mathscr{P}(w)$, which contradicts the choice of $w$. Therefore, $\mathscr{T}$ is closed under taking nonempty initial segments. The rest of the proof that $\mathscr{T} \in \mathbf{Trees}(S)$ is straightforward.
		
		To see that the map $\mathscr{P} \mapsto \mathscr{T}_\mathscr{P}$ is injective, consider any $\tau \in \mathscr{P}$ and $x \in \dom(\tau)$. It is easy to see that
		\[
			\tau(x) = |\set{\tau' \in \mathscr{P} \,:\, x \in \dom(\tau') \text{ and } \tau'(x) < \tau(x)}|.
		\]
		Therefore, by Lemma~\ref{lemma:tree_order}, for any $w \in \mathscr{T}$ and $x \in \textbf{tail}(w)$,
		\begin{align*}
			\tau_\mathscr{P}(w)(x) &= |\set{\tau' \in \mathscr{P} \,:\, x \in \dom(\tau') \text{ and } \tau'(x) < \tau_\mathscr{P}(w)(x)}|\\
			&=|\set{w' \in \mathscr{T} \,:\, x \in \textbf{tail}(w') \text{ and } w <_{\operatorname{lex}} w'}|.
		\end{align*}
		The last expression only depends on $\mathscr{T}$; hence, $\mathscr{P}$ can be recovered from $\mathscr{T}$, as desired.
	\end{proof}

	Recall that for each $\mathscr{P} \in \mathbf{Piles}(S)$,
	\[
		\lambda^{\supp(\mathscr{P}) \times \N} \left(\mathbf{App}(\mathscr{P})\right) = \prod_{\tau \in \mathscr{P}} \mathbb{P}[\dom(\tau)] = \prod_{w \in \mathscr{T}_\mathscr{P}} \mathbb{P}[\mathbf{tail}(w)].
	\]
	Hence, due to Lemma~\ref{lemma:injection},
	\[
		\sum_{\mathscr{P} \in \mathbf{Piles}(S)} \lambda^{\supp(\mathscr{P})\times \N} \left(\mathbf{App}(\mathscr{P})\right) \leq \sum_{\mathscr{T} \in \mathbf{Trees}(S)} \, \prod_{w \in \mathscr{T}} \mathbb{P}[\mathbf{tail}(w)].
	\]
	Theorem~\ref{theo:MTcopy} now follows from the following lemma:
	\begin{lemma}
		Let $\omega \colon \B \to [0;1)$ be a function witnessing the correctness of $\B$ and let $S \in \dom(\B)$. Then
		\begin{equation*}
		\sum_{\mathscr{T} \in \mathbf{Trees}(S)} \, \prod_{w \in \mathscr{T}} \mathbb{P}[\mathbf{tail}(w)] \leq \sum_{\substack{B \in \B \,:\\\dom(B) = S}}\frac{\omega(B)}{1 - \omega(B)}.
		\end{equation*}
	\end{lemma}
	\begin{proof}
		For $B \in \B$, set
		\[
			\rho(B) \defeq \frac{\omega(B)}{1 - \omega(B)}. 
		\]
		By the correctness of $\B$, each $B \in \B$ satisfies
		\[
		\mathbb{P}[B] \leq \omega(B) \prod_{B' \in \Nbhd_\B(B)} (1 - \omega(B')) = \frac{\rho(B)}{1+\rho(B)} \prod_{B' \in \Nbhd_\B(B)} \frac{1}{1+\rho(B')}.
		\]
		which yields
		\[
		\rho(B) \geq \mathbb{P}[B] (1+\rho(B)) \prod_{B' \in \Nbhd_\B(B)} (1+\rho(B')) = \mathbb{P}[B]\prod_{\substack{B'\in\B\,:\\ \dom(B') \cap \dom(B) \neq \0}} (1+\rho(B')).
		\]
		This implies that for $S \in \dom(\B)$,
		\begin{equation}\label{eq:final}
		\sum_{\substack{B \in \B \,:\\\dom(B) = S}}\rho(B) \geq \sum_{\substack{B \in \B \,:\\\dom(B) = S}}\mathbb{P}[B]\prod_{\substack{B'\in\B\,:\\ \dom(B') \cap S \neq \0}} (1+\rho(B')) \geq \mathbb{P}[S]\prod_{\substack{B'\in\B\,:\\ \dom(B') \cap S \neq \0}} (1+\rho(B')).
		\end{equation}
		
		For a finite tree $\mathscr{T}$ over $\dom(\B)$, define its \emph{weight} to be
		\[
		\mathbf{w}(\mathscr{T}) \defeq \prod_{w \in \mathscr{T}} \mathbb{P}[\mathbf{tail}(w)].
		\]
		For each $n \geq 1$, let $\mathbf{Trees}_{\leq n}(S)$ denote the subset of $\mathbf{Trees}(S)$ consisting of all trees of height at most~$n$ (where the \emph{height} of a tree $\mathscr{T}$ is the largest length of its elements). Then $\mathbf{Trees}(S) = \bigcup_{n=1}^\infty \mathbf{Trees}_{\leq n} (S)$ and this union is increasing, so it suffices to show that for all $n \geq 1$,
		\begin{equation}\label{eq:induction}
			\sum_{\mathscr{T} \in \mathbf{Trees}_{\leq n}(S)} \mathbf{w}(\mathscr{T}) \leq \sum_{\substack{B \in \B \,:\\\dom(B) = S}}\rho(B).
		\end{equation}
		We prove~\eqref{eq:induction} by induction on $n$. The unique tree in $\mathbf{Trees}_{\leq 1}(S)$ is $\set{(S)}$, and, by~\eqref{eq:final}, we have
		\[
			\mathbf{w}(\set{(S)}) = \mathbb{P}[S]  \leq \sum_{\substack{B \in \B \,:\\\dom(B) = S}}\rho(B).
		\]
		Suppose that~\eqref{eq:induction} holds for some $n \geq 1$. For $\mathscr{T} \in \mathbf{Trees}_{\leq n + 1}(S)$ and $S' \in \dom(\B)$ with $S' \cap S \neq \0$, let
		\[
			\mathscr{T}_{S'} \defeq \set{v \in \dom(\B)^\ast \setminus \set{\0}\,:\, (S) \concat v \in \mathscr{T} \text{ and } v_0 = S'}.
		\]
		In other words, if $(S, S') \in \mathscr{T}$, then $\mathscr{T}_{S'}$ is the subtree of $\mathscr{T}$ rooted at $(S, S')$; and otherwise $\mathscr{T}_{S'} = \0$. This gives a bijection between $\mathbf{Trees}_{\leq n + 1}(S)$ and the set
		\[
			\prod_{\substack{S' \in \dom(\B):\\ S' \cap S \neq \0}} (\set{\0} \cup \mathbf{Trees}_{\leq n}(S')).
		\]
		Moreover, if we set $\mathbf{w}(\0) \defeq 1$, then
		\[
			\mathbf{w}(\mathscr{T}) = \mathbb{P}[S]\prod_{\substack{S' \in \dom(\B):\\ S' \cap S \neq \0}} \mathbf{w}(\mathscr{T}_{S'}).
		\]
		Thus, we obtain
		\begin{align}\label{eq:ind}
			\sum_{\mathscr{T} \in \mathbf{Trees}_{\leq n+1}(S)} \mathbf{w}(\mathscr{T}) &= \sum_{\mathscr{T} \in \mathbf{Trees}_{\leq n+1}(S)} \mathbb{P}[S]\prod_{\substack{S' \in \dom(\B):\\ S' \cap S \neq \0}} \mathbf{w}(\mathscr{T}_{S'})\nonumber\\
			&= \mathbb{P}[S] \prod_{\substack{S' \in \dom(\B):\\ S' \cap S \neq \0}} \left(1 + \sum_{\mathscr{T} \in \mathbf{Trees}_{\leq n}(S')} \mathbf{w}(\mathscr{T})\right).
		\end{align}
		By the induction hypothesis,
		\begin{align*}
			\prod_{\substack{S' \in \dom(\B):\\ S' \cap S \neq \0}} \left(1 + \sum_{\mathscr{T} \in \mathbf{Trees}_{\leq n}(S')} \mathbf{w}(\mathscr{T})\right) &\leq \prod_{\substack{S' \in \dom(\B):\\ S' \cap S \neq \0}} \left(1 +\sum_{\substack{B' \in \B \,:\\\dom(B') = S'}}\rho(B')\right) \\
			&\leq \prod_{\substack{S' \in \dom(\B):\\ S' \cap S \neq \0}} \prod_{\substack{B' \in \B \,:\\\dom(B') = S'}} \left(1 +\rho(B')\right) = \prod_{\substack{B'\in\B\,:\\ \dom(B') \cap S \neq \0}} (1+\rho(B')).
		\end{align*}
		It remains to plug this into~\eqref{eq:ind} and apply~\eqref{eq:final}.
	\end{proof}
	
	\section{Theorems of Kim, Johansson, and Kahn}\label{app:KJK}
	
	\noindent Our main results are designed so that many classical combinatorial arguments can be transferred to the measurable setting simply by replacing the~LLL with either Theorem~\ref{theo:approxLLL} or Theorem~\ref{theo:Thm1}, depending on the desired outcome. Almost no additional work is required; the relevant instances of the~LLL and the verification of their correctness remain unmodified, so the only things to check being the Borelness of the instances and their hereditary local finiteness (for Theorem~\ref{theo:approxLLL}) or invariance (for Theorem~\ref{theo:Thm1}), which usually are rather straightforward. In the introduction we mentioned Theorems~\ref{theo:col_large_g}, \ref{theo:FreeEdge}, \ref{theo:approxVertex}, and~\ref{theo:approxEdge} as particular examples of this approach; they are obtained by substituting Theorems~\ref{theo:approxLLL} and~\ref{theo:Thm1} into the proofs of Theorems~\ref{theo:Kim}, \ref{theo:Johansson}, and~\ref{theo:Kahn}.
	
	The classical proofs of Theorems~\ref{theo:Kim}, \ref{theo:Johansson}, and~\ref{theo:Kahn} themselves are quite technical and somewhat lengthy (and, arguably, exemplify some of the finest and most sophisticated known applications of the~LLL). The relevant instances of the~LLL are carefully engineered to ensure their correctness and yet to give their solutions sufficiently strong properties. An excellent exposition of all three proofs, along with the intuition that guides them, can be found in~\cite{MolloyReed}, and we do not attempt to reproduce it here. However, for the interested reader, we very briefly sketch in this appendix the main ideas of the method used to prove Theorems~\ref{theo:Kim}, \ref{theo:Johansson}, and~\ref{theo:Kahn}, omitting most of the technical details.
	
	Consider a graph $G$ on a set $X$ with finite maximum degree $d$. A useful, although typically not strictly necessary, observation (see~\cite[Section~1.5]{MolloyReed}) is that we can usually arrange $G$ to be $d$\=/regular by attaching to each vertex $x \in X$ with $\deg_G(x) < d$ an infinite rooted tree whose root has degree $d - \deg_G(x)$ and all of whose other vertices have degree~$d$.\footnote{This method of making a graph regular is different from the one described in~\cite[Section~1.5]{MolloyReed}, where the attention is focused on finite graphs only.} Formally, let $Y$ be the set of all pairs of the form $(x, s)$, where $x \in X$ and $s = (s_1, \ldots, s_n)$ is a finite sequence of integers such that 
	$1 \leq s_1 \leq d - \deg_G(x)$ and $1 \leq s_i \leq d-1$ for all $1 < i \leq n$ (with the possibility of $n=0$ and $s = \0$ allowed). 
	Define a graph $H$ on $Y$ as follows:
	\[
		(x,s) \, H \, (y, t) \,\vcentcolon\Longleftrightarrow\, (x\,G\,y\text{ and }s = t = \0) \text{ or } (x = y \text{ and } (s \subset t \text{ or } t \subset s)).
	\]
	Clearly, $H$ is $d$-regular and the map $x \mapsto (x, \0)$ is an isomorphic embedding of $G$ into~$H$ (denote its image by $G^\ast$). Any cycle in $H$ must be contained in $G^\ast$, so $g(H) = g(G^\ast) = g(G)$. Note that if $G$ is a Borel graph on a standard Borel space $X$, then $Y$ and $H$ are Borel subsets of $\HF(E_G)$.
	
	To illustrate the technique employed in the proofs of Theorems~\ref{theo:Kim}, \ref{theo:Johansson}, and~\ref{theo:Kahn}, we will outline here a proof of the following weakening of Johansson's theorem: \emph{There is $\epsilon > 0$ such that any triangle-free graph $G$ with maximum degree $d \in \N$ satisfies $\chi(G) \leq (1-\epsilon)d + o(d)$} (see~\cite[Theorem~10.2]{MolloyReed}).
	
	Fix $\epsilon > 0$, and let $G$ be a $d$-regular triangle-free graph on a set $X$, with $d$ sufficiently large. A coloring of $G$ is built in two steps. First, we apply the~LLL to produce a partial coloring of $G$ that exhibits the same local behavior as a ``typical'' random coloring. On the second step, we extend this partial coloring to a full coloring of $G$ using the fact that the uncolored part of the graph is ``sparse.'' Let us start by explaining the second step. Suppose that we are given a subset $X' \subseteq X$ and a proper coloring $f \colon X' \to \N$ of $G \vert X'$. We extend $f$ to a proper coloring of the whole graph $G$ in the following ``greedy'' way. Fix a proper coloring $c \colon X \setminus X' \to \N$ of $G \vert (X \setminus X')$ (which exists since $G$ is locally finite). Set $f_0 \defeq f$ and for all $n \in \N$, define $f_{n+1} \colon \dom(f_n) \cup c^{-1}(n) \to \N$ as follows:
	\[
	f_{n+1}(x) \defeq \begin{cases}
	f_n(x) &\text{if }x \in \dom(f_n);\\
	\min \set{i \in \N \,:\, f_n(y) \neq i \text{ for all } y \in G_x \cap \dom(f_n)} &\text{if } c(x) = n.
	\end{cases}
	\]
	Set $f_\infty \defeq \bigcup_{n=0}^\infty f_n$. By construction, $f_\infty$ is a proper coloring of $G$. How many colors does it use? For every $n \in \N$ and $x \in c^{-1}(n)$, there can be at most $d$ distinct colors assigned by $f_n$ to the neighbors of $x$, and so $f_{n+1}(x) \leq d$. Hence, $f_\infty(x) \leq d$ for all $x \in X \setminus X'$, i.e., $f_\infty \vert (X \setminus X')$ uses at most $d+1$ colors.
	
	This upper bound on $f_{n+1}(x)$ for $x \in c^{-1}(n)$ is sharp only if $f_n$ assigns distinct colors to all the neighbors of~$x$. This observation motivates the following definition: A partial proper coloring $f \colon X' \to \N$ is \emph{good} if for every $x \in X \setminus X'$, the following set has cardinality at least $\epsilon d$:
	\[
		\left\{i \in \N\,:\, |\{y \in G_x \cap X' \,:\, f(y) = i\}| \geq 2\right\}.
	\]
	If $f$ is good, then $f_\infty(x) \leq (1-\epsilon)d$ for all $x \in X \setminus X'$. Thus, to complete the proof, we only need to find a good partial coloring $f \colon X \rightharpoonup (1-\epsilon)d + o(d)$.
	
	Here the~LLL comes into play. For simplicity, assume that $d$ is even. For a function $f \colon X \to d/2$, let
	\[
		X_f \defeq \set{x \in X \,:\, f(x) \neq f(y) \text{ for all } y \in G_x}.
	\]
	We want to find a map $f \colon X \to d/2$ such that $f\vert X_f$ is a good partial coloring; the definition of $X_f$ ensures that this coloring is proper. This condition can be easily phrased as an instance of the~LLL, and it turns out that for sufficiently small~$\epsilon$, this instance is correct, and, therefore, a desired good partial coloring exists. The correctness of the instance follows, roughly speaking, from the observation that if all the $d$ neighbors of a vertex $x \in X$ are colored randomly using only $d/2$ colors, then one should expect many colors to be repeated, and since $G$ is triangle-free, shared colors between vertices in the neighborhood of $x$ do not force them to be removed from~$X_f$. The rigorous verification of this fact constitutes the most laborious part of the argument. For the details, see~\cite[Theorem~10.2]{MolloyReed}.

	The proofs of Theorems~\ref{theo:Kim}, \ref{theo:Johansson}, and~\ref{theo:Kahn} follow a similar strategy to the argument outlined above but with a number of clever technical twists. The main difference is that the~LLL is applied repeatedly to produce partial colorings of larger and larger subsets of $X$ (or, in the case of Theorem~\ref{theo:Kahn}, larger and larger subsets of the edge set of $G$). Another difference is that the final step, instead of completing the coloring ``greedily,'' also uses the~LLL, in the form of the following lemma:
	\begin{lemma}[{\cite[Theorem~4.3]{MolloyReed}}]\label{lemma:prob_col}
		Let $G$ be a locally finite graph with vertex set $X$. Let $k \in \N\setminus \set{0}$. Suppose that $L \colon X \to \fins{\N}$ is a function such that for all $x \in X$, $|L(x)| \geq k$ and for each $n \in L(x)$,
		\[
			|\set{y \in G_x \,:\, n \in L(y)}| \leq k/8.
		\]
		Then there exists a proper coloring $f \colon X \to \N$ with $f(x) \in L(x)$ for all $x \in X$.
	\end{lemma}
	Lemma~\ref{lemma:prob_col} is established using a straightforward application of the~LLL. Note that the instance to which the~LLL is applied there is only invariant under those isomorphisms between connected components of $G$ that preserve the value of $L$. In the cases that we are considering, $L$ is defined using the outcomes of the previous applications of the~LLL, so Theorem~\ref{theo:Thm1} applies.
	
	As mentioned before, the precise intricate construction of the instances of the~LLL used in the proofs of Theorems~\ref{theo:Kim}, \ref{theo:Johansson}, and~\ref{theo:Kahn} lies outside the scope of this article. However, the general scheme described above already shows that the~LLL can be replaced by Theorem~\ref{theo:approxLLL} or Theorem~\ref{theo:Thm1}; the details of the proofs---verifying the correctness of the instances---do not require any modification.
	
\end{appendices}
	
\end{document}